\documentclass[10pt,reqno]{amsart}
\usepackage{fullpage} 
\usepackage[toc,page,title,titletoc,header]{appendix}
\usepackage{amsthm,amsmath,amsfonts,amssymb,euscript,hyperref,graphics,color,slashed,mathrsfs,graphicx,mathtools,enumerate,cite,extarrows,multirow}
\hypersetup{colorlinks, citecolor=blue, urlcolor=blue, linkcolor=blue}
\usepackage[english]{babel}
\usepackage{bm}

\allowdisplaybreaks

\numberwithin{equation}{section}

\newtheorem{theorem}{Theorem}[section]
\newtheorem{lemma}[theorem]{Lemma}

\newtheorem{corollary}[theorem]{Corollary}

\newtheorem{remark}[theorem]{Remark}

\title[Rigidity of MHD in thin domain]{Ideal MHD. Part II:\\ Rigidity from infinity for ideal Alfv\'en waves in 3D thin domains}

\author[Mengni Li]{Mengni Li}
\address{Southeast University, No.2 SEU Road, Nanjing 211189, P.R. China}
\email{krisymengni@163.com,  lmn@seu.edu.cn}

\begin{document}
\begin{abstract}
This paper concerns 
the rigidity from infinity for Alfv\'en waves governed by  ideal incompressible magnetohydrodynamic  equations subjected to 
strong background magnetic fields 
along the $x_1$-axis in 3D  thin domains $\Omega_\delta=\mathbb{R}^2\times(-\delta,\delta)$ with $\delta\in(0,1]$ and slip boundary conditions. 
We show that 
in any thin domain $\Omega_\delta$,  
Alfv\'en waves must vanish identically if their scattering fields vanish at infinities. 
As an application, the rigidity of Alfv\'en waves in $\Omega_{\delta}$, propagating along the horizontal direction, can be approximated by the rigidity of Alfv\'en waves in $\mathbb{R}^2$ 
when $\delta$ is sufficiently small. Our proof relies on 
the uniform (with respect to $\delta$) weighted energy estimates with a position parameter in weights to track the center of Alfv\'en waves. The key issues in the analysis 
include dealing with the nonlinear nature of Alfv\'en waves and the geometry of thin domains.

\smallskip
	
\noindent
\textbf{Running title}: Rigidity of MHD	in thin domain
\end{abstract}

\maketitle
\tableofcontents

\section{Introduction}

Magnetohydrodynamics (MHD for short) is the study of interactions between magnetic fields and electrically conducting fluids. 
It is known that, in a perfectly conducting fluid, magnetic field lines 
tend to be frozen into the fluid in the sense that they move with the fluid. On this basis, one may assume that the fluid flows along a strong constant background magnetic field $B_0$ and it is perturbed by a small velocity field $v$ perpendicular to $B_0$. 
There subsequently exist two kinds of restoring forces in MHD leading to different waves, namely, the magnetic pressure force parallel to magnetic field lines produces 
magnetosonic waves propagating orthogonal to $B_0$, whereas the  magnetic tension force perpendicular 
 to magnetic field lines generates a completely new type of waves called Alfv\'en waves propagating 
 along $B_0$. The phenomenon of Alfv\'en waves is significant 
  in wide applications  such as solar wind, dynamo theory, and magnetic reconnection, 
and it  particularly underpins the existing explanations for the origin of the Earth's magnetic field and of the solar field. Readers are referred to the textbook \cite{Davidson} for more details and related topics.

Historically, the MHD theory was pioneered in 1942 \cite{Alfven} by  the Swedish plasma physicist Hannes Alfv\'en, who was awarded the Nobel prize in 1970 for his outstanding contributions to this field, and the Alfv\'en waves were named after him.  
An interesting case, together with its linearized analysis considered  therein, can be revisited as follows: 
If the electric conductivity is set to be infinite, the fluid density and permeability to be $1$, and $B_0$ to be homogenous and parallel to the $x_1$-axis, then 
by elementary calculation the 3D MHD equations become a 
1D wave equation for the magnetic field $b$ as 
$\frac{\partial^2 b}{\partial x_1^2}=\frac{4\pi}{B_0^2}\frac{\partial^2 b}{\partial t^2}$, which implies that the Alfv\'en waves move along $B_0$ in both directions with the velocity $V_A=\frac{B_0}{\sqrt{4\pi}}$. 
Since then, the MHD equations have been considered extensively by mathematicians and physicists working in this field, especially when studying the dynamics of Alfv\'en waves. 

It behooves us to first review some significant  progress  
made in studying the global existence 
of Alfv\'en waves governed by  incompressible MHD systems  in strong magnetic field backgrounds with small perturbations. In \cite{Bardos-Sulem-Sulem}, Bardos, Sulem, and Sulem  
investigated  the ideal case via convolutions with fundamental solutions, which gave rise to global existence results 
 in H\"older spaces. 
In \cite{Lin-Zhang, Xu-Zhang}, Lin, Xu, and Zhang   obtained global existence results
 in energy spaces for the viscous case with strong fluid viscosity 
but no Ohmic dissipation. However, the latter works required the smallness of 
data to rely on viscosity so that the Fourier methods effective in 
studying Navier--Stokes equations can be adapted. Later on, He, Xu, and Yu \cite{He-Xu-Yu} contributed  to the ideal case  as well as the viscous case with small diffusion. In their work, with inspiration from the stability of Minkowski spacetime \cite{Christodoulou-Klainerman}, 
the global nonlinear stability of Alfv\'en waves was proved by more natural energy methods, and moreover the smallness of data was independent of viscosity.
 Alternative proofs can also be found in 
 the work of Cai and Lei \cite{Cai-Lei} as well as  the works of Wei and 
 Zhang  \cite{Wei-Zhang-2017,Wei-Zhang-2018}.  
 
 In recent years, the 
energy methods in \cite{He-Xu-Yu} have been extended in the following 
two aspects. One is a crosswise extension \cite{Xu}, where an important and
 interesting type of  
3D thin domains $\Omega_{\delta}=\mathbb{R}^2\times (-\delta,\delta)$ with a thickness parameter $\delta$ (the thickness of domain as $2\delta$) that is sufficiently small has been  
considered.  
In \cite{Xu}, Xu derived global existence and uniform (with respect to $\delta$) energy estimates  of Alfv\'en waves in $\Omega_\delta$.  
It was then proved that  3D Alfv\'en waves in $\Omega_\delta$ propagating along the horizontal direction can be approximated by 2D Alfv\'en waves in $\mathbb{R}^2$ as
 $\delta$ goes to zero. 
The other aspect in which the aforementioned energy methods have been extended 
is a lengthwise extension \cite{Li-Yu}, where the scattering behavior of 3D 
 Alfv\'en waves at infinities has  been further 
 studied.  
 Specifically, by choosing suitable position parameters in weighted energy estimates, Li and Yu constructed the rigidity from infinity theorems for Alfv\'en waves in $\mathbb{R}^{3}$ (which indeed hold in $\mathbb{R}^{2}$ as well, see the appendix in this paper) that Alfv\'en waves must vanish identically if their scattering fields vanish at infinities. This conclusion is consistent with the physical phenomenon 
that there are no Alfv\'en waves emanating from the plasma if no waves are detected by faraway observers. Based on these results, 
it appears promising that the rigidity from infinity for 3D Alfv\'en waves in thin domains $\Omega_\delta$ can  also be constructed. In fact, to demonstrate this rigidity will be the main purpose of the present  paper.

It is meaningful to point out that thin-domain problems are encountered in  mathematical models  from many applications.
 For example, in  solid mechanics and especially in elasticity, one considers thin rods, beams, plates, and shells; in fluid dynamics, one is concerned with fluid lubrication, blood circulation,  ocean dynamics, and meteorology  problems; 
  and in MHD, one deals with  solar tachocline, 
 wave heating in the solar corona, thin airfoils in MHD, and shallow water MHD. Other  applied areas   include physiology, nanotechnology, material engineering, etc.  Readers can also consult  \cite{Raugel,Hale-Raugel} about more details regarding physical backgrounds. Most of the above problems are described by various 
partial differential equations on thin domains, their  solutions are always compared with corresponding ones on 
domains of  lower dimension with the thin directions reduced, 
 and a fairly satisfactory understanding has been achieved for
 what impact the thickness of thin domains has on solutions. 
From pioneering works to recent ones, we list some representative 
contributions to the global existence of solutions in thin domains as follows: See the works  of Hale and Raugel \cite{Hale-Raugel-JMPA,Hale-Raugel-TAMS} for reaction-diffusion equations and damped hyperbolic equations; the work of Raugel and Sell \cite{Raugel-Sell-1} for Navier-Stokes equations;  the work of Marsden, Ratiu, and Raugel \cite{Marsden-Ratiu-Raugel}  for Euler equations; as well as the aforementioned work \cite{Xu}   for MHD equations.   The works  \cite{Prizzi-Rybakowski-2,Iftimie-Raugel,Cai-Wang,Elsken-Prizzi} and 
 their related references, in which more equations are considered or further results obtained, are also relevant.

We now turn to  
review the wave scattering theory and especially its rigidity from infinity aspect. 
While we do not give a complete history of this classical topic here, 
special mention should be made of the simplest instance therein, i.e., the smooth solution $\phi
$ to the 3D linear wave equation $\Box\phi=0$. 
In the situation  where $\phi=O(\frac{1}{t})$, if the initial data are in a sufficiently 
 mathematically elegant  space 
 	[all we require is that for any 
 radius	 $r$ and any center $y$, the integral $\int_{B_r(y)} (|\phi(0,x)|+\partial\phi(0,x)|)d\sigma(x)$,  	with $\sigma$ the induced surface measure, is uniformly bounded], then the scattering field is exactly the Radon transform of the initial data, and hence the uniqueness of the Radon transform yields the rigidity from infinity that the solution must vanish identically if its scattering field vanishes at infinities. This result can be found in \cite{Lax-P} and other relevant works, 
such as  \cite{Friedlander80,Ludwig}. We point out that this kind of rigidity is also called unique continuation. Readers are referred to the survey paper of Ionescu and Klainerman  \cite{Ionescu_Klainerman_Survey} for more details on an  essential role of unique continuation in  studying the uniqueness of black holes. Such an idea was then applied to the above example of linear waves in more general forms by Alexakis and Shao, see \cite{A-S} for instance. 
 Regarding the method, these unique continuation results were mainly obtained by making use of the Carleman  estimates.

Our construction of scattering fields for 3D Alfv\'en waves in $\Omega_{\delta}$ (see Theorem \ref{thm1}) is similar to the corresponding one for the above 3D linear waves, and 
our main rigidity from infinity result for 3D Alfv\'en waves in $\Omega_{\delta}$ 
(see Theorem \ref{thm2}) is indeed an  analogue of these unique continuation results in spirit.  These results are consistent with the physical intuition that  the 3D Alfv\'en waves  produced from the plasma in thin domains are characterized by their scattering fields detected by faraway observers, and accordingly  there are no  Alfv\'en waves  
in thin domains if there are no waves  detected by faraway observers. 
 However, both the nonlinear nature of Alfv\'en waves and the geometry of thin domains  distinguish our problem from the aforementioned situations. 
 Instead of applying Carleman-type estimates, we will investigate  the uniform (with respect to $\delta$) weighted energy estimates  with a position parameter in weights to track the center of Alfv\'en waves. Although difficulties in the nonlinear setting still exist, the energy method will be more natural and effective  for Alfv\'en waves. We will also see that the rigidity of 3D Alfv\'en waves in $\Omega_{\delta}$ converges to the rigidity of 2D Alfv\'en waves in $\mathbb{R}^2$ in the limit that the thickness parameter $\delta$ goes to zero (see Corollary \ref{thm3}). 
 This is consistent with the geometric  fact that $\Omega_{\delta}$ can be viewed as $\mathbb{R}^2$ when $\delta\to 0$ as well as the 
 approximation result in \cite{Xu} that  3D Alfv\'en waves in $\Omega_{\delta}$ converge to  2D Alfv\'en waves in $\mathbb{R}^2$ 
when $\delta\to 0$.  Moreover, the rigidity of 2D Alfv\'en waves in $\mathbb{R}^2$ deduced from the approximation in Corollary \ref{thm3} coincides with the 2D version (see Theorem \ref{rigidity theorem 1 2d}) of the rigidity from infinity for Alfv\'en waves established in \cite{Li-Yu}. For clarity of presentation,  
we will put together these relations in Figure \ref{fig:relations}.

\subsubsection*{\bf Structure of the paper}

 In Section \ref{sec:notation}, we introduce  
 some preliminary notations for  MHD equations in 3D thin domains $\Omega_{\delta}$, adapt 
 the uniform weighted energy estimates for the study of rigidity 
 and revisit the 
 necessary proof with main steps. This new version of uniform weighted energy estimates  with a position parameter in weights 
 also allows us to construct the global solution (Alfv\'en waves) in $\Omega_{\delta}$. We  devote Section \ref{sec:more estimates} 
 to the  estimates on the nonlinear term and  Section \ref{sec:pressure}  
to the estimates on the pressure term. 
All these estimates are the main ingredients in 
studying the dynamical behavior of Alfv\'en waves in $\Omega_{\delta}$. Based on these estimates, Section \ref{sec:scattering} then 
involves constructing scattering fields of Alfv\'en waves in $\Omega_{\delta}$ together with their weighted Sobolev spaces at infinities, and describing the large time behavior of  Alfv\'en waves in $\Omega_{\delta}$ with vanishing scattering fields 
by smallness conditions of energies; see our first main Theorem \ref{thm1} for details. In Section \ref{sec:rigidity}, we prove our second main Theorem \ref{thm2} which concerns the rigidity from infinity for Alfv\'en waves in $\Omega_{\delta}$. In Section \ref{sec:approximation}, the 
convergence between the rigidity of Alfv\'en waves in $\Omega_{\delta}$ propagating along the horizontal direction and the rigidity of Alfv\'en waves in $\mathbb{R}^2$  
is provided in Corollary \ref{thm3} as an immediate consequence because of the rigidity condition that the scattering fields are vanishing at infinities. Finally, in the appendix,  we 
provide Theorem \ref{rigidity theorem 1 2d} concerning the rigidity for Alfv\'en waves in $\mathbb{R}^2$ as an extension of \cite{Li-Yu} concerning  the  rigidity  for Alfv\'en waves in $\mathbb{R}^3$. We remark  that though derivations are different, both Corollary \ref{thm3} (as $\delta$ goes to zero) and 
Theorem \ref{rigidity theorem 1 2d} demonstrate the  the rigidity   for Alfv\'en waves in $\mathbb{R}^2$, and  in particular these two perspectives are complementary with each other and coexist in a harmonious way.

\subsubsection*{\bf Acknowledgement} 
Part of this 
work was carried out during the author's PhD time at Tsinghua University. 
The author is deeply indebted to Prof. Pin Yu  for many  discussions and suggestions on the problem. 
The author also would like to acknowledge Prof. Li Xu for 
sharing experiences and valuable insights on  \cite{Xu}. This work was partially 
supported by NSFJS (No. BK20220792). 

\section{Notations and 
energy estimates}\label{sec:notation} 
Throughout this paper, we consider 3D thin domains  $\Omega_{\delta}=\mathbb{R}^2\times(-\delta,\delta)$ with $\delta\in(0,1]$. 
In $\Omega_\delta$, it is clear that 
the horizontal variables are $x_1$ and $x_2$ while the vertical variable is $x_3$ and the thin direction is 
the  vertical direction. 
We are interested in the most interesting situation where a strong background magnetic field $B_0$ presents along the horizontal direction and therefore a small initial perturbation will generate Alfv\'en waves which propagate along $B_0$.  
Without loss of generality, we assume  $B_0$ to be a 
constant magnetic field  parallel to the $x_1$-axis: $B_0=(1,0,0)$. The geometry of this part is illustrated in the following Figure \ref{fig:Omegadelta} with a stereo 
view on the left hand side and a sectional view on the right hand side.

\begin{figure}[ht]
		\vspace{-0.1cm}
	\centering
	\includegraphics[width=5in]{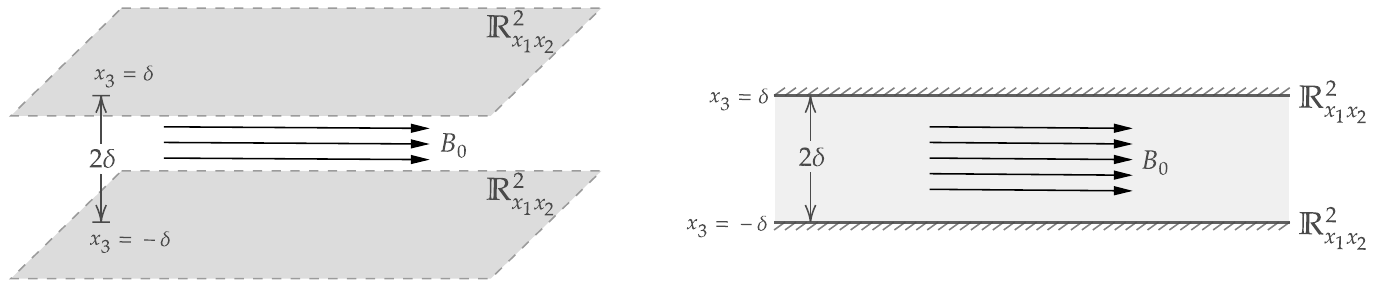}
	\vspace{-0.1cm}
	\caption{$\Omega_\delta=\mathbb{R}^2\times(-\delta,\delta)$}
	\label{fig:Omegadelta}
		\vspace{-0.1cm}
\end{figure}

Let us 
consider the ideal incompressible MHD equations in $\Omega_\delta$ with slip boundary conditions:
\begin{equation}\label{MHD general}
	\begin{cases}
		&\partial_t  v+ v\cdot \nabla v = -\nabla p + (\nabla\times b)\times b, \\
		&\partial_t b + v\cdot \nabla b =  b \cdot \nabla v,\\
		&\operatorname{div} v =0,\ \ 
		\operatorname{div} b =0,\\
		&v^3|_{x_3=\pm \delta}=0,\ \ b^3|_{x_3=\pm\delta}=0,
\end{cases}\end{equation}
where $b=(b^1,b^2,b^3)$ is the magnetic field, $v=(v^1,v^2,v^3)$  the velocity of fluid, and $p$  the scalar pressure of fluid. By  
writing $(\nabla\times b)\times b=-\nabla(\frac{1}{2}| b|^2)+b \cdot \nabla  b$ and  $p'=p+\frac{1}{2}|b|^2$, the first equation in \eqref{MHD general} can be rephrased as 
\[\partial_t  v+ v\cdot \nabla v = -\nabla p' + b \cdot \nabla b.\]
For the sake of simplicity, we will still use $p$ to denote $p'$ in what follows.

We now employ the Els\"{a}sser variables $Z_+ = v +b$ and $Z_- = v-b$ to diagonalize the system \eqref{MHD general} as  
\begin{equation}\label{MHD in Elsasser}\begin{cases}
		&\partial_t  Z_+ +Z_- \cdot \nabla Z_+ = -\nabla p, \\
		&\partial_t  Z_- +Z_+ \cdot \nabla Z_- = -\nabla p,\\
		&\operatorname{div} Z_+ =0,\ \
		\operatorname{div} Z_- =0,\\ 
		&Z_+^3|_{x_3=\pm\delta}=0,\ \ Z_-^3|_{x_3=\pm\delta}=0.
\end{cases}\end{equation}
The fluctuations $z_{+}=Z_{+}-B_0$ and $z_{-}=Z_{-}-(-B_0)=Z_-+B_0$ then
 verify  
\begin{equation}\label{MHD equation}\begin{cases}
		&\partial_{t}z_{+}+(z_--B_0)\cdot \nabla z_{+} =-\nabla p,\\
		&\partial_{t}z_{-}+(z_++B_0)\cdot \nabla z_{-} =-\nabla p,\\
		&\operatorname{div} z_{+}=0,\ \
		\operatorname{div} z_{-}=0,\\
		&z_+^3|_{x_3=\pm\delta}=0,\ \ z_-^3|_{x_3=\pm\delta}=0.
\end{cases}\end{equation}
When $t=0$, 
the initial data $(z_{+,0}(x),z_{-,0}(x))$ also satisfy the divergence  conditions  $\operatorname{div} z_{+,0}=0$,  
$\operatorname{div} z_{-,0}=0$ and the boundary conditions 
$z_{+,0}^3|_{x_3=\pm\delta}=0$,  $z_{-,0}^3|_{x_3=\pm\delta}=0$.

To investigate the  
influence of the thickness of   $\Omega_\delta$, we first perform the following rescalings  
for any $x\in \Omega_1$:  
\begin{align*}
	z_{\ (\delta)}^h(t,x_h, x_3) &:=z^h(t,x_h,\delta x_3),\\
	z_{\ (\delta)}^3(t,x_h, x_3) &:=\delta^{-1}z^3(t,x_h,\delta x_3),\\
	p_{(\delta)}(t,x_h, x_3) &:=p(t,x_h,\delta x_3).
\end{align*}
Hereinafter, we always denote $z=(z^h,z^3)$,  $z^h=(z^1,z^2)$ and $x_h=(x_1,x_2)$. 
Moreover, all above notations for $z$ adapt not only to $z_\pm$ and  $z_{\pm(\delta)}$ here  but also to $z_{\pm,0}$ and $z_{\pm(\delta),0}$ in the sequel. 
Then the MHD system \eqref{MHD equation} in $\Omega_\delta$ for $z_{\pm}$ can be written as the following rescaled system in $\Omega_1$ for $z_{\pm(\delta)}$: 
\begin{equation}\label{eq:rescale}
	\begin{cases}
		&	\partial_tz_{+(\delta)}+\big(z_{-(\delta)}-B_0\big)\cdot\nabla z_{+(\delta)}=-\nabla_\delta p_{(\delta)},\\
		&	\partial_tz_{-(\delta)}+\big(z_{+(\delta)}+B_0\big)\cdot\nabla z_{-(\delta)}=-\nabla_\delta p_{(\delta)},\\
		&	\operatorname{div} z_{+(\delta)}=0,\ \ 
		\operatorname{div} z_{-(\delta)}=0,\\
		&z_{+(\delta)}^3|_{x_3=\pm 1}=0,\ \ z_{-(\delta)}^3|_{x_3=\pm 1}=0,
	\end{cases}
\end{equation}
where  $\nabla_\delta=(\partial_1,\partial_2,\delta^{-2}\partial_3)$, and  
the initial data $(z_{+(\delta),0}(x),z_{-(\delta),0}(x))$ also satisfy the divergence  conditions  $\operatorname{div} z_{+(\delta),0}=0$,  
$\operatorname{div} z_{-(\delta),0}=0$ and the boundary conditions 
$z_{+(\delta),0}^3|_{x_3=\pm 1}=0$,  $z_{-(\delta),0}^3|_{x_3=\pm 1}=0$. 
In particular, the anisotropic property of the rescaled system \eqref{eq:rescale} will help us deal with the thickness of $\Omega_{\delta}$ in the rest of this paper.

We turn to review the following Sobolev lemma for functions defined on thin domains $\Omega_\delta$. The interested readers can consult Lemma 2.6 (ii) in \cite{Xu} about its proof.
\begin{lemma}[Sobolev lemma]\label{lemma:sobolev}
	For any $f(x)\in H^2(\Omega_{\delta})$, we have 
	\[\|f\|_{L^\infty_x}\leqslant C\sum_{k+l\leqslant 2}\delta^{l-\frac{1}{2}}\|\nabla^k_h\partial_3^lf\|_{L^2_x}.\]
\end{lemma}
 
Let us introduce characteristic functions as  
\begin{equation}\label{def:u}
u_\pm=u_\pm(t,x_1)=x_1\mp t,
\end{equation}
and  weight functions as 
\begin{equation}\label{def:weight}
\langle u_\pm\rangle=(1+|u_\pm\mp a|^2)^{\frac{1}{2}}=(1+|x_1\mp(t+a)|^2)^{\frac{1}{2}},
\end{equation}
where $a$ represents the position parameter tracking the centers of Alfv\'en waves. Some elementary properties of $\langle u_\pm\rangle$ can be collected as the following lemma. The proof is by direct calculation and so is omitted here. 
 
\begin{lemma}[Properties of weights]\label{lemma:weights}
For any  
 $\sigma\in (0,\frac{1}{3})$, there  hold:
\begin{enumerate}[(i)]
\item For $|x_h-y_h|\leqslant 2$, we have 
\begin{equation}\label{eq:weight1}
	\left(\langle u_\mp\rangle^{1+\sigma}\langle u_\pm\rangle^{\frac{1}{2}(1+\sigma)}\right)(t,x)\lesssim\left(\langle u_\mp\rangle^{1+\sigma}\langle u_\pm\rangle^{\frac{1}{2}(1+\sigma)}\right)(t,y).
\end{equation}
\item For $|x_h-y_h|\geqslant 1$, we have
\begin{equation}\label{eq:weight2}
\left(\langle u_\mp\rangle^{1+\sigma}\langle u_\pm\rangle^{\frac{1}{2}(1+\sigma)}\right)(t,x)\lesssim |x_h-y_h|^{\frac{3}{2}(1+\sigma)}\left(\langle u_\mp\rangle^{1+\sigma}\langle u_\pm\rangle^{\frac{1}{2}(1+\sigma)}\right)(t,y).
\end{equation}
\item For all $k\in\mathbb{Z}_{\geqslant 0}$ with $k\leqslant 3$, we have 
\begin{equation}\label{eq:weight3}
\left|\nabla^k\left(\langle u_\mp\rangle^{1+\sigma}\langle u_\pm\rangle^{\frac{1}{2}(1+\sigma)}\right)\right|\lesssim\langle u_\mp\rangle^{1+\sigma}\langle u_\pm\rangle^{\frac{1}{2}(1+\sigma)}.
\end{equation}
\item For all $k,l\in\mathbb{Z}_{\geqslant 0}$ with $k+l\leqslant 2$, we have 
\begin{equation}\label{eq:weightde}
\left|\nabla_h^k\partial_3^l\left(\frac{\langle u_\pm\rangle^{1+\sigma}}{\langle u_\mp\rangle^{\frac{1}{2}(1+\sigma)}}\right)\right|\lesssim\frac{\langle u_\pm\rangle^{1+\sigma}}{\langle u_\mp\rangle^{\frac{1}{2}(1+\sigma)}}.
\end{equation}
\item For the product of $\langle u_+\rangle$ and $\langle u_-\rangle$, we have 
\begin{equation}\label{eq:product}
\langle u_+\rangle\langle u_-\rangle\gtrsim 1+|t+a|.
\end{equation}
\end{enumerate}
Here, the notation $A\lesssim B$ means that there is a universal constant $C$ (independent of $a$)  such that $A\leqslant CB$; the notation $A\gtrsim B$ means that there is a universal constant $C$ (independent of $a$) such that $A\geqslant CB$. 
\end{lemma}

 Now we state the following weighted div-curl lemma. This lemma will help us control the gradient of vectors by their divergence and curl. In particular, the readers are referred to Lemma 2.4 in \cite{Xu} for a divergence-free version of this lemma.
\begin{lemma}[Weighted div-curl lemma]\label{lemma:div}
	Let $\lambda(x)\geqslant 1$ 
	be a smooth positive function on $\Omega_\delta$ with the additional property $|\nabla\lambda|\lesssim \lambda$. For any smooth vector field $v(x)\in H^1(\Omega_\delta)$, we have 
	\begin{equation}\label{eq:d-c}
		\big\|\sqrt\lambda\nabla v\big\|_{L^2(\Omega_\delta)}^2 \lesssim \big\|\sqrt\lambda\operatorname{div }\nabla v\big\|_{L^2(\Omega_\delta)}^2+  \big\|\sqrt{\lambda}\operatorname{curl }\nabla v\big\|_{L^2(\Omega_\delta)}^2+ \big\|\sqrt\lambda v\big\|_{L^2(\Omega_\delta)}^2+\Big|\int_{\partial\Omega_\delta} \lambda (v^h\cdot\nabla_h v^3-v^3\nabla_hv^h)dx_h\Big|,
	\end{equation}
	provided $\sqrt{\lambda}v\in L^2({\Omega_\delta})$ and $\sqrt{\lambda}\nabla v\in L^2({\Omega_\delta})$.
\end{lemma}
\begin{proof}
We first recall the vector calculus identity 
\begin{equation*}
	-\Delta v=-\nabla(\operatorname{div}v)+\operatorname{curl}\operatorname{curl}v.
\end{equation*}
Multiplying this identity by $\lambda v$ and then integrating over $\Omega_\delta$ lead us to 
\begin{equation*} 
	-\int_{\Omega_\delta}\lambda v\cdot \Delta vdx=-\int_{\Omega_\delta}\lambda v\cdot \nabla(\operatorname{div}v)dx+\int_{\Omega_\delta}\lambda v\cdot \operatorname{curl}\operatorname{curl}vdx.
\end{equation*}

After integration by parts, we infer that
\begin{align*}
	&-\int_{\Omega_\delta}\lambda v\cdot \Delta vdx		
	=-\int_{\partial\Omega_\delta} \lambda v\cdot \nabla v\cdot ndS+\int_{\Omega_\delta} \lambda|\nabla v|^2dx+\int_{\Omega_\delta}\nabla \lambda\cdot v\cdot \nabla vdx,\\
	&-\int_{\Omega_\delta}\lambda v\cdot \nabla(\operatorname{div}v)dx
	=-\int_{\partial\Omega_\delta} \lambda v\cdot (\operatorname{div} v) ndS+\int_{\Omega_\delta} \lambda|\operatorname{div} v|^2dx+\int_{\Omega_\delta}\nabla \lambda\cdot  v \operatorname{div} vdx,\\
	&\int_{\Omega_\delta}\lambda v\cdot \operatorname{curl}\operatorname{curl}vdx
	=-\int_{\partial\Omega_\delta} \lambda v\cdot (\operatorname{curl} v\times n)dS+\int_{\Omega_\delta} \lambda|\operatorname{curl} v|^2dx+\int_{\Omega_\delta}(\nabla \lambda\wedge v)\cdot \operatorname{curl} vdx,
\end{align*}
where $n$ is the unit outward normal of $\partial\Omega_{\delta}$ and $dS$ is the surface measure of $\partial\Omega_{\delta}$. 
Therefore we derive
\begin{align*}
	\int_{\Omega_\delta} \lambda|\nabla v|^2dx&=\int_{\partial\Omega_\delta} \lambda v\cdot \big[\nabla v\cdot n-
	(\operatorname{div}v) n-(\operatorname{curl}v\times n)\big]dS+\int_{\Omega_\delta} \lambda|\operatorname{div} v|^2dx+\int_{\Omega_\delta} \lambda|\operatorname{curl} v|^2dx\\
	&\ \ \ \ -\int_{\Omega_\delta}\nabla \lambda\cdot v\cdot \nabla vdx+\int_{\Omega_\delta}\nabla \lambda\cdot v \operatorname{div} vdx+\int_{\Omega_\delta}(\nabla \lambda\wedge v)\cdot \operatorname{curl} vdx.\stepcounter{equation}\tag{\theequation}\label{div1}
\end{align*}

By direct calculation, we acquire 
\begin{equation}\label{div2}
\int_{\partial\Omega_\delta} \lambda v\cdot \big[\nabla v\cdot n-
(\operatorname{div}v) n-(\operatorname{curl}v\times n)\big]dS=\int_{\partial\Omega_\delta} \lambda v^j(\partial_jv^in_i-\partial_iv^in_j)dS.
\end{equation}
In view of \eqref{div1} and \eqref{div2}, we deduce that 
\begin{align*}
	\int_{\Omega_\delta} \lambda|\nabla v|^2dx&\leqslant\Big|\int_{\partial\Omega_\delta} \lambda v^j(\partial_jv^in_i-\partial_iv^in_j)dS\Big|+\int_{\Omega_\delta} \lambda|\operatorname{div} v|^2dx+\int_{\Omega_\delta} \lambda|\operatorname{curl} v|^2dx\\
	&\ \ \ \ +\int_{\Omega_\delta}|\nabla \lambda||v||\nabla v|dx
	+\int_{\Omega_\delta}|\nabla \lambda| |v||\operatorname{div} v|dx+\int_{\Omega_\delta}|\nabla \lambda||v||\operatorname{curl} v|dx. 
\end{align*}

Noticing that $n\big|_{\partial\Omega_{\delta}}=n\big|_{x_3=\pm\delta}=(0,0,\pm 1)^T$ and $dS=dx_h$, we obtain
\[
	\Big|\int_{\partial\Omega_\delta} \lambda v^j(\partial_jv^in_i-\partial_iv^in_j)dS\Big|=\Big|\int_{\partial\Omega_\delta} \lambda (v\cdot\nabla v^3-v^3\operatorname{div}v)dx_h\Big|=\Big|\int_{\partial\Omega_\delta} \lambda (v^h\cdot\nabla_h v^3-v^3\nabla_hv^h)dx_h\Big|.\]
Then using Cauchy-Schwarz inequality gives rise to 
\begin{align*}
\int_{\mathbb{R}^3} \lambda|\nabla v|^2dx&\leqslant\Big|\int_{\partial\Omega_\delta} \lambda (v^h\cdot\nabla_h v^3-v^3\nabla_hv^h)dx_h\Big|+\frac{3}{2}\int_{\mathbb{R}^3} \lambda|\operatorname{div} v|^2dx
+\frac{3}{2}\int_{\mathbb{R}^3} \lambda|\operatorname{curl} v|^2dx\\
&\ \ \ \ +\frac{3}{2}\int_{\mathbb{R}^3}\frac{|\nabla\lambda|^2}{\lambda}|v|^2dx
+\frac{1}{2}\int_{\mathbb{R}^3}\lambda|\nabla v|^2dx.
\end{align*}
 	Since 
 	$|\nabla\lambda|\lesssim\lambda$ gives $\frac{|\nabla\lambda|}{\sqrt{\lambda}}\lesssim\sqrt{\lambda}$, we finally summarize that 
 	\begin{equation*}
 		\big\|\sqrt\lambda\nabla v\big\|_{L^2(\mathbb{R}^3)}^2 \lesssim\big\|\sqrt\lambda\operatorname{div }v\big\|_{L^2(\mathbb{R}^3)}^2+ \big\|\sqrt\lambda\operatorname{curl }v\big\|_{L^2(\mathbb{R}^3)}^2+ \big\|\sqrt\lambda v\big\|_{L^2(\mathbb{R}^3)}^2+\Big|\int_{\partial\Omega_\delta} \lambda (v^h\cdot\nabla_h v^3-v^3\nabla_hv^h)dx_h\Big|.
 	\end{equation*}
This ends the proof of the lemma. 
\end{proof}

As a consequence of Lemma \ref{lemma:div}, there also holds:
 \begin{lemma}[Weighted div-curl lemma with higher order derivatives]\label{lemma:divcurl}
	Let $\lambda(x)\geqslant 1$ 
be a smooth positive function on $\Omega_\delta$ with the additional property $|\nabla\lambda|\lesssim \lambda$. For any smooth vector field $v(x)\in H^m(\Omega_\delta)$ $(m\in \mathbb{Z}_{\geqslant 1})$, we have 
\begin{align*}
	\big\|\sqrt\lambda\nabla^k v\big\|_{L^2(\Omega_\delta)}^2 &\lesssim\sum_{l=0}^{k-1}\big\|\sqrt\lambda\operatorname{div }\nabla^lv\big\|_{L^2(\Omega_\delta)}^2+\sum_{l=0}^{k-1} \big\|\sqrt{\lambda}\operatorname{curl }\nabla^lv\big\|_{L^2(\Omega_\delta)}^2+ \big\|\sqrt\lambda v\big\|_{L^2(\Omega_\delta)}^2\\
	&\ \ \ \ +\sum_{l=0}^{k-1}\Big|\int_{\partial\Omega_\delta} \lambda (\nabla^{l}v^h\cdot\nabla^{l}\nabla_h v^3-\nabla^{l}v^3\cdot\nabla^{l}\nabla_hv^h)dx_h\Big|,\stepcounter{equation}\tag{\theequation}\label{eq:d-c2}
\end{align*}
provided $\sqrt{\lambda}v\in L^2(\Omega_\delta)$ and $\sqrt{\lambda}\nabla^k v\in L^2(\Omega_\delta)$, where $1\leqslant k\leqslant m$. 
 \end{lemma}
 
 \begin{proof}
For any $1\leqslant k\leqslant m$, we have
$\nabla^{k-1}v\in H^{m-k+1}(\Omega_{\delta})\subset H^1(\Omega_{\delta})$. Thus, applying \eqref{eq:d-c} to the vector field $\nabla^{k-1}v$ yields
\begin{align*}
\big\|\sqrt\lambda\nabla^k v\big\|_{L^2(\Omega_\delta)}^2 &\lesssim\big\|\sqrt\lambda\operatorname{div }\nabla^{k-1}v\big\|_{L^2(\Omega_\delta)}^2+ \big\|\sqrt\lambda\operatorname{curl }\nabla^{k-1}v\big\|_{L^2(\Omega_\delta)}^2+ \big\|\sqrt\lambda\nabla^{k-1} v\big\|_{L^2(\Omega_\delta)}^2\\
&\ \ \ \ +\Big|\int_{\partial\Omega_\delta} \lambda (\nabla^{k-1}v^h\cdot\nabla^{k-1}\nabla_h v^3-\nabla^{k-1}v^3\cdot\nabla^{k-1}\nabla_hv^h)dx_h\Big|.
\end{align*}
 By induction on $k$, we can infer 
 \eqref{eq:d-c2} immediately. 
Hence the lemma is proved.
 \end{proof}

Let $t^*$ be the lifespan of solutions to the system \eqref{MHD equation}, which indeed also adapts to the rescaled system \eqref{eq:rescale}. Here and subsequently, 
we follow the notations of energy  and flux used in \cite{Xu}. For all $k,l\in\mathbb{Z}_{\geqslant 0}$, we denote the weighted energy norms on $[0,t^*]\times\Omega_\delta$ as  
\begin{align*}
	&E^{(k,l)}_{\pm}(z_\pm(t))=
	\sum_{|\alpha_h|=k}\big\|\langle u_\mp\rangle^{1+\sigma}\partial_h^{\alpha_h}\partial_3^lz_{\pm}(t,x)\big\|_{L^2(\Omega_{\delta})}^2,\ \ E^{(k,l)}_{\pm}(z_\pm)=\sup_{0\leqslant t\leqslant t^*}E^{(k,l)}_{\pm}(z_\pm(t)),\\
	&E^{(k,l)}_{\pm}(z_\pm^h(t))=
	\sum_{|\alpha_h|=k}\big\|\langle u_\mp\rangle^{1+\sigma}\partial_h^{\alpha_h}\partial_3^lz_{\pm}^h(t,x)\big\|_{L^2(\Omega_{\delta})}^2,\ \ E^{(k,l)}_{\pm}(z_\pm^h)=\sup_{0\leqslant t\leqslant t^*}E^{(k,l)}_{\pm}(z_\pm^h(t)),\\
	&E^{(k,l)}_{\pm}(z_\pm^3(t))=
	\sum_{|\alpha_h|=k}\big\|\langle u_\mp\rangle^{1+\sigma}\partial_h^{\alpha_h}\partial_3^lz_{\pm}^3(t,x)\big\|_{L^2(\Omega_{\delta})}^2,\ \ E^{(k,l)}_{\pm}(z_\pm^3)=\sup_{0\leqslant t\leqslant t^*}E^{(k,l)}_{\pm}(z_\pm^3(t)),
\end{align*}
and the weighted flux norms on $[0,t^*]\times\Omega_\delta$  as 
\begin{align*}
	&F^{(k,l)}_{\pm}(z_\pm(t))=\sum_{|\alpha_h|=k}\int_{[0,t]\times\Omega_\delta}\frac{\langle u_{\mp}\rangle^{2(1+\sigma)}}{\langle u_{\pm}\rangle^{1+\sigma}}|\partial_h^{\alpha_h}\partial_3^lz_{\pm}(\tau,x)|^2dxd\tau,\ \ F^{(k,l)}_{\pm}(z_\pm)=F^{(k,l)}_{\pm}(z_\pm(t^*)),\\
	&F^{(k,l)}_{\pm}(z_\pm^h(t))=\sum_{|\alpha_h|=k}\int_{[0,t]\times\Omega_\delta}\frac{\langle u_{\mp}\rangle^{2(1+\sigma)}}{\langle u_{\pm}\rangle^{1+\sigma}}|\partial_h^{\alpha_h}\partial_3^lz_{\pm}^h(\tau,x)|^2dxd\tau,\ \ F^{(k,l)}_{\pm}(z_\pm^h)=F^{(k,l)}_{\pm}(z_\pm^h(t^*)),\\
	&F^{(k,l)}_{\pm}(z_\pm^3(t))=\sum_{|\alpha_h|=k}\int_{[0,t]\times\Omega_\delta}\frac{\langle u_{\mp}\rangle^{2(1+\sigma)}}{\langle u_{\pm}\rangle^{1+\sigma}}|\partial_h^{\alpha_h}\partial_3^lz_{\pm}^3(\tau,x)|^2dxd\tau,\ \  F^{(k,l)}_{\pm}(z^3_\pm)=F^{(k,l)}_{\pm}(z^3_\pm(t^*)).
\end{align*}
	For any $\alpha_h\in(\mathbb{Z}_{\geqslant0})^2$ and $l\geqslant 0$, we see that  
\begin{align*}
	&E_\pm^{(\alpha_h,l)}(z^h_{\pm(\delta)})=\delta^{2(l-\frac{1}{2})}E_\pm^{(\alpha_h,l)}(z^h_{\pm}),\ \ 
	E_\pm^{(\alpha_h,0)}(z^3_{\pm(\delta)})=\delta^{-3}E_\pm^{(\alpha_h,0)}(z^3_{\pm}),\\
	&E_\pm^{(\alpha_h,l)}(z^3_{\pm(\delta)})=\delta^{2(l-\frac{3}{2})}E_\pm^{(\alpha_h,l)}(z^3_{\pm})=\delta^{2(l-\frac{3}{2})}E_\pm^{(\alpha_h,l-1)}(\nabla_h\cdot z^h_{\pm})\ \text{ for }l\geqslant 1,\\
	&F_\pm^{(\alpha_h,l)}(z^h_{\pm(\delta)})=\delta^{2(l-\frac{1}{2})}F_\pm^{(\alpha_h,l)}(z^h_{\pm}),\ \ 
	F_\pm^{(\alpha_h,0)}(z^3_{\pm(\delta)})=\delta^{-3}F_\pm^{(\alpha_h,0)}(z^3_{\pm}),\\
	&F_\pm^{(\alpha_h,l)}(z^3_{\pm(\delta)})=\delta^{2(l-\frac{3}{2})}F_\pm^{(\alpha_h,l)}(z^3_{\pm})=\delta^{2(l-\frac{3}{2})}F_\pm^{(\alpha_h,l-1)}(\nabla_h\cdot z^h_{\pm})\ \text{ for }l\geqslant 1.
\end{align*}
We denote the energy and flux on  $[0,t^*]\times\Omega_1$ as
\begin{align*}
	&	E_\pm^{(k)}(z_{\pm(\delta)})=\!\sum_{k'+l'=k}\!E_\pm^{(k',l')}(z_{\pm(\delta)}),\ \ E_\pm^{(k)}(z^h_{\pm(\delta)})=\!\sum_{k'+l'=k}\!E_\pm^{(k',l')}(z^h_{\pm(\delta)}),\ \ E_\pm^{(k)}(z^3_{\pm(\delta)})=\!\sum_{k'+l'=k}\!E_\pm^{(k',l')}(z^3_{\pm(\delta)}),\\
	&	F_\pm^{(k)}(z_{\pm(\delta)})=\!\sum_{k'+l'=k}\!F_\pm^{(k',l')}(z_{\pm(\delta)}),\ \ F_\pm^{(k)}(z^h_{\pm(\delta)})=\!\sum_{k'+l'=k}\!F_\pm^{(k',l')}(z^h_{\pm(\delta)}),\ \ F_\pm^{(k)}(z^3_{\pm(\delta)})=\!\sum_{k'+l'=k}\!F_\pm^{(k',l')}(z^3_{\pm(\delta)}).
\end{align*}

We are now ready to review the global existence and uniform (with respect to $\delta$) weighted energy estimates of the system \eqref{MHD equation}: 
\begin{theorem}[Uniform weighted energy estimates in $\Omega_{\delta}$; adapted from Theorem 1.1 in  \cite{Xu} for the study of rigidity] 
	\label{lemma:global}
	Let $N\in \mathbb{Z}_{\geqslant 5}$, $\delta\in(0,1]$ and  $\sigma\in(0,\frac{1}{3})$.  
	There exists a universal constant $\varepsilon_0\in(0,1)$ such that if the initial data $(z_{+,0}(x),z_{-,0}(x))$ of the system \eqref{MHD equation}  satisfy
	\begin{equation*} 
		\mathcal{E}(0):= 	\sum_{+,-}\bigg(\!\sum_{k+l\leqslant 2N}\!\delta^{2(l-\frac{1}{2})}E_{\pm}^{(k,l)}(z_{\pm,0})	+\!\sum_{k\leqslant 2N-1}\!\delta^{-3}E_{\pm}^{(k,0)}(z_{\pm,0}^3) +\!\sum_{k+l\leqslant N+2}\!\delta^{2(l-\frac{1}{2})}E_{\pm}^{(k,l)}(\partial_3z_{\pm,0})\bigg) \leqslant\varepsilon_0^2,
	\end{equation*}
	then the system \eqref{MHD equation}  admits a unique global  solution $\big(z_+(t,x),z_-(t,x)\big)$. 
	Moreover,  
	there is a universal constant $C$ such that the following  uniform (with respect to $\delta$) weighted energy estimates hold:
	\begin{align*}
		\mathcal{E}:= & \sum_{+,-}\bigg(\sum_{k+l\leqslant 2N}\delta^{2(l-\frac{1}{2})}\big(E_{\pm}^{(k,l)}(z_{\pm})+F_{\pm}^{(k,l)}(z_{\pm})\big)   +\sum_{k\leqslant 2N-1}\delta^{-3}\big(E_{\pm}^{(k,0)}(z_{\pm}^3)+F_{\pm}^{(k,0)}(z_{\pm}^3)\big)\\
		&\ \ \ \ \ \ \  +\sum_{k+l\leqslant N+2}\delta^{2(l-\frac{1}{2})}\big(E_{\pm}^{(k,l)}(\partial_3z_{\pm})+F_{\pm}^{(k,l)}(\partial_3z_{\pm})\big)\bigg)\\
		&\leqslant C\mathcal{E}(0).\stepcounter{equation}\tag{\theequation}\label{eq:estimate}
	\end{align*} 
In particular, both the constants  $\varepsilon_0$ and $C$ are independent of the thickness parameter $\delta$ and the position parameter $a$. These facts are indeed important keys to ensuring the further study on rigidity. 
\end{theorem}

The proof of Theorem \ref{lemma:global} is based on the standard continuity method, and 
we refer the readers to Theorem 1.1 in \cite{Xu} for the details.  
For the sake of use, we recall the main bootstrap argument therein as follows.

\textbf{(Bootstrap Assumption)} We assume that 
\begin{equation}\label{assumption1}
	\|z_\pm^1\|_{L^\infty_tL^\infty_x}\leqslant 1,
\end{equation}
and 
\begin{align*}
	&\delta^{2(l-\frac{1}{2})}E_\pm^{(k,l)}(z_\pm)\leqslant 2C_1\varepsilon^2,\ \ \ \ \ \delta^{2(l-\frac{1}{2})}F_\pm^{(k,l)}(z_\pm)\leqslant 2C_1\varepsilon^2,\ \ \ \ \text{for }k+l\leqslant 2N,\\
	&\delta^{-3}E_\pm^{(k,0)}(z_\pm^3)\leqslant 2C_1\varepsilon^2,\ \ \ \ \ \ \ \ \  \delta^{-3}F_\pm^{(k,0)}(z_\pm^3)\leqslant 2C_1\varepsilon^2,\ \ \ \ \ \  \ \text{for }k\leqslant 2N-1,\\
	&\delta^{2(l-\frac{1}{2})}E_\pm^{(k,l)}(\partial_3z_\pm)\leqslant 2C_1\varepsilon^2,\ \  \delta^{2(l-\frac{1}{2})}F_\pm^{(k,l)}(\partial_3z_\pm)\leqslant 2C_1\varepsilon^2,\ \text{for }k+l\leqslant N+2,\stepcounter{equation}\tag{\theequation}\label{assumption2}
\end{align*}
where $C_1$ can be determined by the energy estimates. We remark here that the assumption \eqref{assumption2} is legitimate since  
it holds for the initial data and then remains correct for at least a short time interval $[0,t^*]$.

\textbf{(Bootstrap Expectation)} We show that   
there exists a universal constant   $\varepsilon_0$ such that for all $\varepsilon\leqslant \varepsilon_0$, the constants in \eqref{assumption1}-\eqref{assumption2} can be improved to their corresponding halves, i.e. 
\begin{equation}\label{improve1}
	\|z_\pm^1\|_{L^\infty_tL^\infty_x}\leqslant \frac{1}{2},
\end{equation}
and 
\begin{align*}
	&\delta^{2(l-\frac{1}{2})}E_\pm^{(k,l)}(z_\pm)\leqslant C_1\varepsilon^2,\ \ \ \ \  \delta^{2(l-\frac{1}{2})}F_\pm^{(k,l)}(z_\pm)\leqslant C_1\varepsilon^2,\ \ \ \ \text{for }k+l\leqslant 2N,\\
	&\delta^{-3}E_\pm^{(k,0)}(z_\pm^3)\leqslant C_1\varepsilon^2,\ \ \ \ \ \ \ \ \  \delta^{-3}F_\pm^{(k,0)}(z_\pm^3)\leqslant C_1\varepsilon^2,\ \ \ \ \ \ \ \text{for }k\leqslant 2N-1,\\
	&\delta^{2(l-\frac{1}{2})}E_\pm^{(k,l)}(\partial_3z_\pm)\leqslant C_1\varepsilon^2,\ \  \delta^{2(l-\frac{1}{2})}F_\pm^{(k,l)}(\partial_3z_\pm)\leqslant C_1\varepsilon^2,\ \text{for }k+l\leqslant N+2.\stepcounter{equation}\tag{\theequation}\label{improve2}
\end{align*}

We point out that both the constants $\varepsilon_0$ and $C_1$ are independent of the thickness parameter $\delta$, the position parameter $a$ as well as the lifespan $[0,t^*]$. Based on the last fact, 
there exists an endless continuation of the lifespan from $[0,t^*]$ to $[0,+\infty]$, which thus leads to the global existence result. Therefore  
the proof of  
this theorem reduces to showing the uniform energy estimates \eqref{eq:estimate}, that is, we only need to show \eqref{improve1}-\eqref{improve2} under 
 \eqref{assumption1}-\eqref{assumption2}. The rest of the proof is similar to Page 22-48 in \cite{Xu} with the only difference  
 being the appearance of the position parameter $a$. In fact, as discussed in Lemma \ref{lemma:weights}, the position parameter only influences the properties of weights, which also  
 accounts for the independence of universal constants  
 on the position parameter.

 \begin{remark}
 In the rest of this paper, we will directly use  \eqref{improve1}-\eqref{improve2} to derive other estimates. 
\end{remark}

Based on the notations before, we have
\begin{align*}
	\mathcal{E}_\delta:=	& \sum_{+,-}\bigg(\sum_{k\leqslant 2N}\big(E_{\pm}^{(k)}(z^h_{\pm(\delta)})+F_{\pm}^{(k)}(z^h_{\pm(\delta)})\big) +\sum_{k\leqslant 2N-1}\big(E_{\pm}^{(k)}(z_{\pm(\delta)}^3)+F_{\pm}^{(k)}(z_{\pm(\delta)}^3)\big)\\
&\ \ \ \ \ \  +\delta^2\big(E_{\pm}^{(2N)}(z_{\pm(\delta)}^3)+F_{\pm}^{(2N)}(z_{\pm(\delta)}^3)\big)
+\delta^{-2}\sum_{k\leqslant N+2}\big(E_{\pm}^{(k)}(\partial_3z^h_{\pm(\delta)})+F_{\pm}^{(k)}(\partial_3z^h_{\pm(\delta)})\big)\bigg)\\
\sim&\ \mathcal{E}.
\end{align*}
We remark that similar equivalences will not be repeated in the sequel. Then	the following  corollary holds 
as a direct renormalization consequence of Theorem \ref{lemma:global}:
\begin{corollary}[Uniform weighted energy estimates in $\Omega_1$]
	\label{coro1}
	Let $N\in \mathbb{Z}_{\geqslant 5}$, $\delta\in(0,1]$ and  $\sigma\in(0,\frac{1}{3})$. 
	There exists a universal constant $\varepsilon_1\in(0,\varepsilon_0]$ such that if the initial data $(z_{+(\delta),0}(x),z_{-(\delta),0}(x))$ of the rescaled system \eqref{eq:rescale}  satisfy
	\begin{equation*} 
			\mathcal{E}_{\delta}(0):=	\sum_{+,-}\bigg(\!	\sum_{k\leqslant 2N} E_{\pm}^{(k)}(z^h_{\pm(\delta),0})	+\!\!\sum_{k\leqslant 2N-1}E_{\pm}^{(k)}(z_{\pm(\delta),0}^3)+\delta^2E_{\pm}^{(2N)}(z_{\pm(\delta),0}^3) +\delta^{-2}\!\!\sum_{k\leqslant N+2}E_{\pm}^{(k)}(\partial_3z^h_{\pm(\delta),0})\!\bigg)\leqslant\varepsilon_1^2,
	\end{equation*}
	then the rescaled system \eqref{eq:rescale}  admits a unique global  solution $\big(z_+(t,x),z_-(t,x)\big)$. 
	Moreover,  
	there is a universal constant $C$ such that the following  uniform (with respect to $\delta$) weighted energy estimates hold:
	\[\mathcal{E}_\delta\leqslant C\mathcal{E}_\delta(0).\]
	In particular, both the constants  $\varepsilon_1$ and $C$ are independent of the thickness parameter $\delta$ and the position parameter $a$. These facts are indeed important keys to ensuring the further study on rigidity. 
\end{corollary}

To end this section, we recall the approximation theory 
of 
the global solution (Alfv\'en waves) to the system \eqref{MHD equation} in $\Omega_{\delta}$ as $\delta\to 0$:

\begin{theorem}[Asymptotics of the global solution  from $\Omega_{\delta}$ to $\mathbb{R}^2$ as  $\delta$ goes to zero; extracted from Theorem 1.3 in \cite{Xu}] 
	\label{lemma:approx}
	Let $N\in \mathbb{Z}_{\geqslant 5}$, $\delta\in(0,1]$ and  $\sigma\in(0,\frac{1}{3})$.  
	Assume that   
	the initial data $(z_{+(\delta),0},z_{-(\delta),0})$ converge to $(z_{+(0),0},z_{-(0),0})$ in $H^{N+1}(\Omega_1)$
	with respect to $\delta$:
	\begin{equation*}
		\lim_{\delta\to 0}	\big(z_{\pm(\delta),0}^h(x_h,x_3),z_{\pm(\delta),0}^3(x_h,x_3)\big)
	=\big(z_{\pm(0),0}^h(x_h),0\big)\ \ \text{ in }H^{N+1}(\Omega_1),
\end{equation*}
i.e. 
\begin{equation*}
	\lim_{\delta\to 0}\sum_{k\leqslant N+1}\big\|\langle u_\mp\rangle^{1+\sigma}\nabla^k(z_{\pm(\delta),0} -z_{\pm(0),0} )\big\|_{L^2(\Omega_1)}=0,
\end{equation*}
where $z_{\pm(0),0}^h$ satisfies $\nabla_h\cdot z_{\pm(0),0}^h=0$. 
If $\big(z_{+(\delta)}(t,x),z_{-(\delta)}(t,x)\big)$ is a solution to the rescaled system  \eqref{eq:rescale}, then there exist functions $z_{\pm(0)}^h(t,x_h)$ such that for any $x_3\in(-1,1)$, there hold
\begin{equation*}
	\begin{split}
		\lim_{\delta\to 0}z_{\pm(\delta)}^h(t,x_h,x_3)&=z_{\pm(0)}^h(t,x_h)\ \ \text{ in }H^N(\mathbb{R}^2),\\
		\lim_{\delta\to 0}z_{\pm(\delta)}^3(t,x_h,x_3)&=0\ \ \text{ in }H^{N-1}(\mathbb{R}^2).
\end{split}\end{equation*}
\end{theorem}

\begin{remark}
In particular, $\big(z_{+(0)}^h(t,x_h),z_{-(0)}^h(t,x_h)\big)$ solves the 2D version of the rescaled system \eqref{eq:rescale} with the initial data $\big(z_{+(0),0}^h(t,x_h),z_{-(0),0}^h(t,x_h)\big)$. Theorem \ref{lemma:approx} indeed shows that 3D Alfv\'en waves in $\Omega_\delta$ propagating along the horizontal direction can be approximated by 2D Alfv\'en waves in $\mathbb{R}^2$ as
$\delta$ goes to zero. 
\end{remark}

\section{Estimates on the nonlinear term}\label{sec:more estimates}

In this section, let us provide some important estimates for the nonlinear term.

\begin{lemma}[Estimate for $\mathbf{I}_{\pm}^{(\alpha_h,l)}$]\label{lemma1}
For any $\alpha_h\in(\mathbb{Z}_{\geqslant 0})^2$ and $l\in\mathbb{Z}_{\geqslant 0}$ with $|\alpha_h|+l\leqslant 2N-1$, there holds  
\[\delta^{l+\frac{1}{2}}\big\|\langle u_\mp\rangle^{1+\sigma}\langle u_\pm\rangle^{\frac{1}{2}(1+\sigma)}\mathbf{I}_{\pm}^{(\alpha_h,l)}\big\|_{L^2_tL^2_x}\lesssim C_1\varepsilon^2,\]
where \[\mathbf{I}_{\pm}^{(\alpha_h,l)}:=-\partial_h^{\alpha_h}\partial_3^l(\nabla z_\mp\cdot\nabla z_\pm).\]
We remark that given $N\in\mathbb{Z}_{\geqslant 5}$, this result also holds for any $\alpha_h\in(\mathbb{Z}_{\geqslant 0})^2$ and any  $l\in\mathbb{Z}_{\geqslant 0}$ with $0\leqslant|\alpha_h|+l\leqslant N+2$.
\end{lemma}
\begin{proof}
	We only derive the estimate for $\mathbf{I}_+^{(\alpha_h,l)}$. The estimate on $\mathbf{I}_-^{(\alpha_h,l)}$ can be given in the same way.

By the divergence free condition $\operatorname{div}z_-=0$, we see that 
\begin{align*}
	\nabla z_-\cdot\nabla z_+
	&=\nabla z_-^k\cdot\partial_kz_+=\nabla z_-^h\cdot\nabla_hz_++\nabla z_-^3\cdot\partial_3z_+\\
	&=(\nabla_hz_-^h,\partial_3z_-^h) \cdot\nabla_hz_++(\nabla_hz_-^3,\partial_3 z_-^3)\cdot\partial_3z_+\\
	&=(\nabla_hz_-^h,\partial_3z_-^h) \cdot\nabla_hz_++(\nabla_hz_-^3,-\nabla_h z_-^h)\cdot\partial_3z_+,  
\end{align*}
and therefore
\begin{align*}
	\big|\mathbf{I}_{+}^{(\alpha_h,l)}\big|
	&\lesssim \sum_{\beta_h\leqslant \alpha_h\atop l_1\leqslant l}\Big(\underbrace{\big|\partial_h^{\alpha_h-\beta_h}\partial_3^{l-l_1}\nabla_h z_-^h\big|\cdot \big|\partial_h^{\beta_h}\partial_3^{l_1}\nabla_hz_+\big|}_{\mathbf{I}_{+,1}^{(\beta_h,l_1)}} +\underbrace{\big|\partial_h^{\alpha_h-\beta_h}\partial_3^{l-l_1}\nabla_h z_-\big|\cdot \big|\partial_h^{\beta_h}\partial_3^{l_1}\partial_3z_+\big|}_{\mathbf{I}_{+,2}^{(\beta_h,l_1)}}\\ 
	&\ \ \ \ \ \ \ \ \ \ \ \ \ \ \ \ \ \  +\underbrace{\big|\partial_h^{\alpha_h-\beta_h}\partial_3^{l-l_1}\partial_3 z_-^h\big|\cdot \big|\partial_h^{\beta_h}\partial_3^{l_1}\nabla_hz_+\big| }_{\mathbf{I}_{+,3}^{(\beta_h,l_1)}}\Big).
\end{align*}
In this way we obtain that 
\begin{align*}
	\delta^{l+\frac{1}{2}}\big\|\langle u_-\rangle^{1+\sigma}\langle u_+\rangle^{\frac{1}{2}(1+\sigma)}\mathbf{I}_{+}^{(\alpha_h,l)}\big\|_{L^2_tL^2_x}
	&\lesssim \delta^{l+\frac{1}{2}}\big\|\langle u_-\rangle^{1+\sigma}\langle u_+\rangle^{\frac{1}{2}(1+\sigma)}\mathbf{I}_{+,1}^{(\beta_h,l_1)}\big\|_{L^2_tL^2_x}\\
	&\ \ \ \ +\delta^{l+\frac{1}{2}}\big\|\langle u_-\rangle^{1+\sigma}\langle u_+\rangle^{\frac{1}{2}(1+\sigma)}\mathbf{I}_{+,2}^{(\beta_h,l_1)}\big\|_{L^2_tL^2_x}\\
	&\ \ \ \ +\delta^{l+\frac{1}{2}}\big\|\langle u_-\rangle^{1+\sigma}\langle u_+\rangle^{\frac{1}{2}(1+\sigma)}\mathbf{I}_{+,3}^{(\beta_h,l_1)}\big\|_{L^2_tL^2_x}.\stepcounter{equation}\tag{\theequation}\label{eqrho1}
\end{align*}

According to the size of $|\beta_h|+l_1$, we now have  two cases:
\[|\beta_h|+l_1\leqslant N-1\ \ \text{ and }\ \ N\leqslant|\beta_h|+l_1\leqslant  |\alpha_h|+l\leqslant 2N-1.\]
\subsubsection*{\bf Case 1:  $|\beta_h|+l_1\leqslant N-1$} Thanks to  $(|\beta_h|+l_1+1)+2\leqslant N+2$, we can use the Sobolev lemma to bound  $L^\infty_x$ norms of $\nabla_h^{|\beta_h|+1}\partial_3^{l_1} z_{+}$ in $\mathbf{I}_{+,1}^{(\beta_h,l_1)}$ as well as in  $\mathbf{I}_{+,3}^{(\beta_h,l_1)}$ and $\nabla_h^{|\beta_h|}\partial_3^{l_1+1} z_{+}$ in $\mathbf{I}_{+,2}^{(\beta_h,l_1)}$ respectively. From this, we have 
\begin{align*}
	&\ \ \ \ \delta^{l+\frac{1}{2}}\big\|\langle u_-\rangle^{1+\sigma}\langle u_+\rangle^{\frac{1}{2}(1+\sigma)}\mathbf{I}_{+,1}^{(\beta_h,l_1)}\big\|_{L^2_tL^2_x}\\
	&\stackrel{\text{H\"older}}{\lesssim} \delta^{l+\frac{1}{2}}\big\|\langle u_+\rangle^{1+\sigma}\nabla_h^{|\alpha_h-\beta_h|+1}\partial_3^{l-l_1}z_-^h\big\|_{L^\infty_tL^2_x}\cdot\Big\|\frac{\langle u_-\rangle^{1+\sigma}}{\langle u_+\rangle^{\frac{1}{2}(1+\sigma)}}\nabla_h^{|\beta_h|+1}\partial_3^{l_1}z_+\Big\|_{L^2_tL^\infty_x}\\
	&\stackrel{\text{Lemma } \ref{lemma:sobolev}}{\lesssim}\delta^{l+\frac{1}{2}}\big\|\langle u_+\rangle^{1+\sigma}\nabla_h^{|\alpha_h-\beta_h|+1}\partial_3^{l-l_1}z_-^h\big\|_{L^\infty_tL^2_x}\cdot \sum_{k_2+l_2\leqslant 2}\delta^{l_2-\frac{1}{2}}\Big\|\nabla_h^{k_2}\partial_3^{l_2}\Big(\frac{\langle u_-\rangle^{1+\sigma}}{\langle u_+\rangle^{\frac{1}{2}(1+\sigma)}}\nabla_h^{|\beta_h|+1}\partial_3^{l_1}z_+\Big)\Big\|_{L^2_tL^2_x} \\
	&\stackrel{\eqref{eq:weightde}}{\lesssim}\delta^{l+\frac{1}{2}}\big\|\langle u_+\rangle^{1+\sigma}\nabla_h^{|\alpha_h-\beta_h|+1}\partial_3^{l-l_1}z_-^h\big\|_{L^\infty_tL^2_x}\cdot  \sum_{k_2+l_2\leqslant 2}\delta^{l_2-\frac{1}{2}} \Big\|\frac{\langle u_-\rangle^{1+\sigma}}{\langle u_+\rangle^{\frac{1}{2}(1+\sigma)}}\nabla_h^{|\beta_h|+1+k_2}\partial_3^{l_1+l_2}z_+\Big\|_{L^2_tL^2_x}\\
	&\lesssim\sum_{k_1\leqslant|\alpha_h|+1}\delta^{l-l_1-\frac{1}{2}}\big(E_-^{(k_1,l-l_1)}(z_-)\big)^{\frac{1}{2}}\cdot\sum_{k_2+l_2\leqslant N+2}\delta^{l_2-\frac{1}{2}}\big(F_+^{(k_2,l_2)}(z_+)\big)^{\frac{1}{2}}
	\stackrel{\eqref{improve2}}{\lesssim} C_1\varepsilon^2,\stepcounter{equation}\tag{\theequation}\label{eqA5}\\
	&\ \ \ \ \delta^{l+\frac{1}{2}}\big\|\langle u_-\rangle^{1+\sigma}\langle u_+\rangle^{\frac{1}{2}(1+\sigma)}\mathbf{I}_{+,2}^{(\beta_h,l_1)}\big\|_{L^2_tL^2_x}\\
	&\stackrel{\text{H\"older}}{\lesssim} \delta^{l+\frac{1}{2}}\big\|\langle u_+\rangle^{1+\sigma}\nabla_h^{|\alpha_h-\beta_h|+1}\partial_3^{l-l_1}z_-\big\|_{L^\infty_tL^2_x}\cdot \Big\|\frac{\langle u_-\rangle^{1+\sigma}}{\langle u_+\rangle^{\frac{1}{2}(1+\sigma)}}\nabla_h^{|\beta_h|}\partial_3^{l_1+1}z_+\Big\|_{L^2_tL^\infty_x}\\
	&\stackrel{\text{Lemma } \ref{lemma:sobolev}}{\lesssim}\delta^{l+\frac{1}{2}}\big\|\langle u_+\rangle^{1+\sigma}\nabla_h^{|\alpha_h-\beta_h|+1}\partial_3^{l-l_1}z_-\big\|_{L^\infty_tL^2_x}\cdot\sum_{k_2+l_2\leqslant 2}\delta^{l_2-\frac{1}{2}}\Big\|\nabla_h^{k_2}\partial_3^{l_2}\Big(\frac{\langle u_-\rangle^{1+\sigma}}{\langle u_+\rangle^{\frac{1}{2}(1+\sigma)}}\nabla_h^{|\beta_h|}\partial_3^{l_1+1}z_+\Big)\Big\|_{L^2_tL^2_x}\\
	&\stackrel{\eqref{eq:weightde}}{\lesssim} \delta^{l+\frac{1}{2}}\big\|\langle u_+\rangle^{1+\sigma}\nabla_h^{|\alpha_h-\beta_h|+1}\partial_3^{l-l_1}z_-\big\|_{L^\infty_tL^2_x}\cdot\sum_{k_2+l_2\leqslant 2}\delta^{l_2-\frac{1}{2}}\Big\|\frac{\langle u_-\rangle^{1+\sigma}}{\langle u_+\rangle^{\frac{1}{2}(1+\sigma)}}\nabla_h^{|\beta_h|+k_2}\partial_3^{l_1+1+l_2}z_+\Big\|_{L^2_tL^2_x}\\
	&\lesssim
	  \sum_{k_1\leqslant|\alpha_h|+1}\delta^{l-l_1-\frac{1}{2}}\big(E_-^{(k_1,l-l_1)}(z_-)\big)^{\frac{1}{2}}\cdot\sum_{k_2+l_2\leqslant N+1}\delta^{l_2+\frac{1}{2}}\big(F_+^{(k_2,l_2+1)}(z_+)\big)^{\frac{1}{2}}
	  \stackrel{\eqref{improve2}}{\lesssim} C_1\varepsilon^2,\stepcounter{equation}\tag{\theequation}\label{eqA6}\\
	&\ \ \ \ \delta^{l+\frac{1}{2}}\big\|\langle u_-\rangle^{1+\sigma}\langle u_+\rangle^{\frac{1}{2}(1+\sigma)}\mathbf{I}_{+,3}^{(\beta_h,l_1)}\big\|_{L^2_tL^2_x}\\
	&\stackrel{\text{H\"older}}{\lesssim} \delta^{l+\frac{1}{2}}\big\|\langle u_+\rangle^{1+\sigma}\nabla_h^{|\alpha_h-\beta_h|}\partial_3^{l-l_1+1}z_-^h\big\|_{L^\infty_tL^2_x} \cdot  \Big\|\frac{\langle u_-\rangle^{1+\sigma}}{\langle u_+\rangle^{\frac{1}{2}(1+\sigma)}}\nabla_h^{|\beta_h|+1}\partial_3^{l_1}z_+\Big\|_{L^2_tL^\infty_x}\\
	& \stackrel{\text{Lemma } \ref{lemma:sobolev}}{\lesssim}\delta^{l+\frac{1}{2}}\big\|\langle u_+\rangle^{1+\sigma}\nabla_h^{|\alpha_h-\beta_h|}\partial_3^{l-l_1+1}z_-^h\big\|_{L^\infty_tL^2_x} \cdot \sum_{k_2+l_2\leqslant 2}\delta^{l_2-\frac{1}{2}}\Big\|\nabla_h^{k_2}\partial_3^{l_2}\Big(\frac{\langle u_-\rangle^{1+\sigma}}{\langle u_+\rangle^{\frac{1}{2}(1+\sigma)}}\nabla_h^{|\beta_h|+1}\partial_3^{l_1}z_+\Big)\Big\|_{L^2_tL^2_x}\\
	&\stackrel{\eqref{eq:weightde}}{\lesssim} \delta^{l+\frac{1}{2}}\big\|\langle u_+\rangle^{1+\sigma}\nabla_h^{|\alpha_h-\beta_h|}\partial_3^{l-l_1+1}z_-^h\big\|_{L^\infty_tL^2_x} \cdot \sum_{k_2+l_2\leqslant 2}\delta^{l_2-\frac{1}{2}}\Big\|\frac{\langle u_-\rangle^{1+\sigma}}{\langle u_+\rangle^{\frac{1}{2}(1+\sigma)}}\nabla_h^{|\beta_h|+1+k_2}\partial_3^{l_1+l_2}z_+\Big\|_{L^2_tL^2_x}\\
	&\lesssim
	 \sum_{k_1\leqslant|\alpha_h|}\delta^{l-l_1+\frac{1}{2}}\big(E_-^{(k_1,l-l_1+1)}(z_-)\big)^{\frac{1}{2}}\cdot\sum_{k_2+l_2\leqslant N+2} \delta^{l_2-\frac{1}{2}}\big(F_+^{(k_2,l_2)}(z_+)\big)^{\frac{1}{2}}
	 \stackrel{\eqref{improve2}}{\lesssim} C_1\varepsilon^2.\stepcounter{equation}\tag{\theequation}\label{eqA7}
\end{align*}
In this case, substituting \eqref{eqA5}-\eqref{eqA7} into \eqref{eqrho1} immediately gives the desired result.

\subsubsection*{\bf Case 2:  $N\leqslant|\beta_h|+l_1\leqslant|\alpha_h|+l\leqslant 2N-1$} 
	However, in this case, $L^\infty_x$ estimates are not directly applicable to $\nabla_h^{|\beta_h|+1}\partial_3^{l_1} z_{+}$ in $\mathbf{I}_{+,1}^{(\beta_h,l_1)}$ as well as in $\mathbf{I}_{+,3}^{(\beta_h,l_1)}$  and $\nabla_h^{|\beta_h|}\partial_3^{l_1+1} z_{+}$ in $\mathbf{I}_{+,2}^{(\beta_h,l_1)}$ anymore. This is because $(|\beta_h|+l_1+1)+2> N+2$ and one cannot afford more than $N+2$ derivatives to close the energy estimates in flux terms.  
	 Instead, we shall adapt $L^\infty_x$ estimates to $\nabla_h^{|\alpha_h-\beta_h|+1}\partial_3^{l-l_1}z_-^h$ in $\mathbf{I}_{+,1}^{(\beta_h,l_1)}$,$\nabla_h^{|\alpha_h-\beta_h|+1}\partial_3^{l-l_1}z_-$ in $\mathbf{I}_{+,2}^{(\beta_h,l_1)}$  and  $\nabla_h^{|\alpha_h-\beta_h|}\partial_3^{l-l_1+1}z_-^h$ in $\mathbf{I}_{+,3}^{(\beta_h,l_1)}$  via the Sobolev lemma as substitutes:
\begin{align*}
&\ \ \ \ \delta^{l+\frac{1}{2}}\big\|\langle u_-\rangle^{1+\sigma}\langle u_+\rangle^{\frac{1}{2}(1+\sigma)}\mathbf{I}_{+,1}^{(\beta_h,l_1)}\big\|_{L^2_tL^2_x}\\
&\stackrel{\text{H\"older}}{\lesssim} \delta^{l+\frac{1}{2}}\big\|\langle u_+\rangle^{1+\sigma}\nabla_h^{|\alpha_h-\beta_h|+1}\partial_3^{l-l_1}z_-^h\big\|_{L^\infty_tL^\infty_x}\cdot \Big\|\frac{\langle u_-\rangle^{1+\sigma}}{\langle u_+\rangle^{\frac{1}{2}(1+\sigma)}}\nabla_h^{|\beta_h|+1}\partial_3^{l_1}z_+\Big\|_{L^2_tL^2_x}\\
&\stackrel{\text{Lemma } \ref{lemma:sobolev}}{\lesssim}\delta^{l+\frac{1}{2}}
\sum_{k_2+l_2\leqslant 2}\delta^{l_2-\frac{1}{2}}\big\|\nabla_h^{k_2}\partial_3^{l_2}\big(\langle u_+\rangle^{1+\sigma}\nabla_h^{|\alpha_h-\beta_h|+1}\partial_3^{l-l_1}z_-^h\big)\big\|_{L^\infty_tL^2_x}\cdot \Big\|\frac{\langle u_-\rangle^{1+\sigma}}{\langle u_+\rangle^{\frac{1}{2}(1+\sigma)}}\nabla_h^{|\beta_h|+1}\partial_3^{l_1}z_+\Big\|_{L^2_tL^2_x}\\
&\stackrel{\eqref{eq:weightde}}{\lesssim}	\delta^{l+\frac{1}{2}}\sum_{k_2+l_2\leqslant 2}\delta^{l_2-\frac{1}{2}}\big\|\langle u_+\rangle^{1+\sigma}\nabla_h^{|\alpha_h-\beta_h|+1+k_2}\partial_3^{l-l_1+l_2}z_-^h\big\|_{L^\infty_tL^2_x}\cdot \Big\|\frac{\langle u_-\rangle^{1+\sigma}}{\langle u_+\rangle^{\frac{1}{2}(1+\sigma)}}\nabla_h^{|\beta_h|+1}\partial_3^{l_1}z_+\Big\|_{L^2_tL^2_x}\\
&\lesssim \sum_{k_2+l_2\leqslant N+2}\delta^{l_2-\frac{1}{2}}\big(E_-^{(k_2,l_2)}(z_-)\big)^{\frac{1}{2}}\cdot\sum_{k_1\leqslant|\alpha_h|+1}\delta^{l_1-\frac{1}{2}}\big(F_+^{(k_1,l_1)}(z_+)\big)^{\frac{1}{2}}
\stackrel{\eqref{improve2}}{\lesssim} C_1\varepsilon^2,\stepcounter{equation}\tag{\theequation}\label{eqA8}\\
&\ \ \ \ \delta^{l+\frac{1}{2}}\big\|\langle u_-\rangle^{1+\sigma}\langle u_+\rangle^{\frac{1}{2}(1+\sigma)}\mathbf{I}_{+,2}^{(\beta_h,l_1)}\big\|_{L^2_tL^2_x}\\
&\stackrel{\text{H\"older}}{\lesssim} \delta^{l+\frac{1}{2}}\big\|\langle u_+\rangle^{1+\sigma}\nabla_h^{|\alpha_h-\beta_h|+1}\partial_3^{l-l_1}z_-\big\|_{L^\infty_tL^\infty_x}\cdot  \Big\|\frac{\langle u_-\rangle^{1+\sigma}}{\langle u_+\rangle^{\frac{1}{2}(1+\sigma)}}\nabla_h^{|\beta_h|}\partial_3^{l_1+1}z_+\Big\|_{L^2_tL^2_x}\\
&\stackrel{\text{Lemma } \ref{lemma:sobolev}}{\lesssim}\delta^{l+\frac{1}{2}}
\sum_{k_2+l_2\leqslant 2}\delta^{l_2-\frac{1}{2}}\big\|\nabla_h^{k_2}\partial_3^{l_2}\big(\langle u_+\rangle^{1+\sigma}\nabla_h^{|\alpha_h-\beta_h|+1}\partial_3^{l-l_1}z_-\big)\big\|_{L^\infty_tL^2_x}\cdot  \Big\|\frac{\langle u_-\rangle^{1+\sigma}}{\langle u_+\rangle^{\frac{1}{2}(1+\sigma)}}\nabla_h^{|\beta_h|}\partial_3^{l_1+1}z_+\Big\|_{L^2_tL^2_x}\\
&\stackrel{\eqref{eq:weightde}}{\lesssim}\delta^{l+\frac{1}{2}}	\sum_{k_2+l_2\leqslant 2}\delta^{l_2-\frac{1}{2}}
\big\|\langle u_+\rangle^{1+\sigma}\nabla_h^{|\alpha_h-\beta_h|+1+k_2}\partial_3^{l-l_1+l_2}z_-\big\|_{L^\infty_tL^2_x}\cdot  \Big\|\frac{\langle u_-\rangle^{1+\sigma}}{\langle u_+\rangle^{\frac{1}{2}(1+\sigma)}}\nabla_h^{|\beta_h|}\partial_3^{l_1+1}z_+\Big\|_{L^2_tL^2_x}\\
&\lesssim
\sum_{k_2+l_2\leqslant N+2}\delta^{l_2-\frac{1}{2}}\big(E_-^{(k_2,l_2)}(z_-)\big)^{\frac{1}{2}}\cdot\sum_{k_1\leqslant|\alpha_h|}\delta^{l_1+\frac{1}{2}}\big(F_+^{(k_1,l_1+1)}(z_+)\big)^{\frac{1}{2}}
\stackrel{\eqref{improve2}}{\lesssim} C_1\varepsilon^2,\stepcounter{equation}\tag{\theequation}\label{eqA9}\\
&\ \ \ \ \delta^{l+\frac{1}{2}}\big\|\langle u_-\rangle^{1+\sigma}\langle u_+\rangle^{\frac{1}{2}(1+\sigma)}\mathbf{I}_{+,3}^{(\beta_h,l_1)}\big\|_{L^2_tL^2_x}\\
&\stackrel{\text{H\"older}}{\lesssim} \delta^{l+\frac{1}{2}}\big\|\langle u_+\rangle^{1+\sigma}\nabla_h^{|\alpha_h-\beta_h|}\partial_3^{l-l_1+1}z_-^h\big\|_{L^\infty_tL^\infty_x}\cdot \Big\|\frac{\langle u_-\rangle^{1+\sigma}}{\langle u_+\rangle^{\frac{1}{2}(1+\sigma)}}\nabla_h^{|\beta_h|+1}\partial_3^{l_1}z_+\Big\|_{L^2_tL^2_x}\\
&\stackrel{\text{Lemma } \ref{lemma:sobolev}}{\lesssim}\delta^{l+\frac{1}{2}}
\sum_{k_2+l_2\leqslant 2}\delta^{l_2-\frac{1}{2}}\big\|\nabla_h^{k_2}\partial_3^{l_2}\big(\langle u_+\rangle^{1+\sigma}\nabla_h^{|\alpha_h-\beta_h|}\partial_3^{l-l_1+1}z_-^h\big)\big\|_{L^\infty_tL^2_x}\cdot \Big\|\frac{\langle u_-\rangle^{1+\sigma}}{\langle u_+\rangle^{\frac{1}{2}(1+\sigma)}}\nabla_h^{|\beta_h|+1}\partial_3^{l_1}z_+\Big\|_{L^2_tL^2_x}\\
&\stackrel{\eqref{eq:weightde}}{\lesssim}\delta^{l+\frac{1}{2}}\sum_{k_2+l_2\leqslant 2}\delta^{l_2-\frac{1}{2}}\big\|\langle u_+\rangle^{1+\sigma}\nabla_h^{|\alpha_h-\beta_h|+k_2}\partial_3^{l-l_1+1+l_2}z_-^h\big\|_{L^\infty_tL^2_x}\cdot \Big\|\frac{\langle u_-\rangle^{1+\sigma}}{\langle u_+\rangle^{\frac{1}{2}(1+\sigma)}}\nabla_h^{|\beta_h|+1}\partial_3^{l_1}z_+\Big\|_{L^2_tL^2_x}\\
&\lesssim	
\sum_{k_2+l_2\leqslant N+1}\delta^{l_2+\frac{1}{2}}\big(E_-^{(k_2,l_2+1)}(z_-)\big)^{\frac{1}{2}}\cdot\sum_{k_1\leqslant|\alpha_h|+1}\delta^{l_1-\frac{1}{2}}\big(F_+^{(k_1,l_1)}(z_+)\big)^{\frac{1}{2}}
\stackrel{\eqref{improve2}}{\lesssim} C_1\varepsilon^2.\stepcounter{equation}\tag{\theequation}\label{eqA10}
\end{align*}
In this case, using \eqref{eqA8}-\eqref{eqA10} together with \eqref{eqrho1} also yields the desired result.

Finally, collecting the above two cases finishes the proof of this lemma. 
\end{proof}

\begin{lemma}[Estimate for $\mathbf{J}_\pm^{(\alpha_h,l)}$]\label{lemma2}
	For any $\alpha_h\in(\mathbb{Z}_{\geqslant 0})^2$ and any $l\in\mathbb{Z}_{\geqslant 0}$ with $0\leqslant|\alpha_h|+l\leqslant 2N-1$, 
	there holds
	\[\delta^{l-\frac{1}{2}}\big\|\langle u_\mp\rangle^{1+\sigma}\langle u_\pm\rangle^{\frac{1}{2}(1+\sigma)}\mathbf{J}_\pm^{(\alpha_h,l)}\big\|_{L^2_tL^2_x}
	\lesssim C_1\varepsilon^2,\]
	where 
	\[\mathbf{J}_\pm^{(\alpha_h,l)}:=\partial_h^{\alpha_h}\partial_3^{l}(z_\mp\cdot \nabla z_\pm).\]
	We remark that given $N\in\mathbb{Z}_{\geqslant 5}$, this result also holds for any $\alpha_h\in(\mathbb{Z}_{\geqslant 0})^2$ and any $l\in\mathbb{Z}_{\geqslant 0}$ with $0\leqslant|\alpha_h|+l\leqslant N+2$.
\end{lemma}
\begin{proof}
Based on symmetry, we only need to derive bound related to $\mathbf{J}_{+}^{(\alpha_h,l)}$. 
	
	Our proof starts with the observation that 
	\begin{align*} \big|\mathbf{J}_+^{(\alpha_h,l)}\big|
		&=\big|\sum_{\beta_h\leqslant\alpha_h\atop l_1\leqslant l}\partial_h^{\alpha_h-\beta_h}\partial_3^{l-l_1}z_-\cdot \nabla\partial_h^{\beta_h}\partial_3^{l_1} z_{+}\big|\lesssim\sum_{\beta_h\leqslant\alpha_h\atop l_1\leqslant l}\big|\partial_h^{\alpha_h-\beta_h}\partial_3^{l-l_1}z_-\cdot \nabla\partial_h^{\beta_h}\partial_3^{l_1} z_{+}\big|\\
		&
		=\sum_{\beta_h\leqslant\alpha_h\atop l_1\leqslant l}\big|\partial_h^{\alpha_h-\beta_h}\partial_3^{l-l_1}z_-^h\cdot \partial_h\partial_h^{\beta_h}\partial_3^{l_1} z_{+}+\partial_h^{\alpha_h-\beta_h}\partial_3^{l-l_1}z_-^3\cdot \partial_3\partial_h^{\beta_h}\partial_3^{l_1} z_{+}\big|\\
		&\lesssim\sum_{\beta_h\leqslant\alpha_h\atop l_1\leqslant l}\Big(\underbrace{\big|\partial_h^{\alpha_h-\beta_h}\partial_3^{l-l_1}z_-^h\big|\cdot \big|\nabla_h^{|\beta_h|+1}\partial_3^{l_1} z_{+}\big|}_{\mathbf{J}_{+,1}^{(\beta_h,l_1)}}+\underbrace{\big|\partial_h^{\alpha_h-\beta_h}\partial_3^{l-l_1}z_-^3\big|\cdot \big|\nabla_h^{|\beta_h|}\partial_3^{l_1+1} z_{+}\big|}_{\mathbf{J}_{+,2}^{(\beta_h,l_1)}}\Big).
	\end{align*}
	As a result, there holds
	\begin{align*}
		\delta^{l-\frac{1}{2}}\big\|\langle u_-\rangle^{1+\sigma}\langle u_+\rangle^{\frac{1}{2}(1+\sigma)}\mathbf{J}_+^{(\alpha_h,l_1)}\big\|_{L^2_tL^2_x}&\lesssim \delta^{l-\frac{1}{2}}\big\|\langle u_-\rangle^{1+\sigma}\langle u_+\rangle^{\frac{1}{2}(1+\sigma)}\mathbf{J}_{+,1}^{(\beta_h,l_1)}\big\|_{L^2_tL^2_x}\\
		&\ \ \ \  +\delta^{l-\frac{1}{2}}\big\|\langle u_-\rangle^{1+\sigma}\langle u_+\rangle^{\frac{1}{2}(1+\sigma)}\mathbf{J}_{+,2}^{(\beta_h,l_1)}\big\|_{L^2_tL^2_x}.\stepcounter{equation}\tag{\theequation}\label{eq:I1I2}
	\end{align*}
	
Due to the number of derivatives (of $z_+$ related), we distinguish the following two cases:	
	\[|\beta_h|+l_1\leqslant N-1\ \ \text{ and }\ \ N\leqslant|\beta_h|+l_1\leqslant |\alpha_h|+l\leqslant 2N-1.\]
	\subsubsection*{\bf Case 1:  $|\beta_h|+l_1\leqslant N-1$} 
	Due to $(|\beta_h|+l_1+1)+2\leqslant N+2$, we can always derive $L^\infty_x$ estimates on $\nabla_h^{|\beta_h|+1}\partial_3^{l_1} z_{+}$ in $\mathbf{J}_{+,1}^{(\beta_h,l_1)}$ and $\nabla_h^{|\beta_h|}\partial_3^{l_1+1} z_{+}$ in $\mathbf{J}_{+,2}^{(\beta_h,l_1)}$ via the Sobolev lemma. Consequently, we can  carry out the estimates as follows:
	\begin{align*}
		&\ \ \ \ \delta^{l-\frac{1}{2}}\big\|\langle u_-\rangle^{1+\sigma}\langle u_+\rangle^{\frac{1}{2}(1+\sigma)}\mathbf{J}_{+,1}^{(\beta_h,l_1)}\big\|_{L^2_tL^2_x}\\
		&\stackrel{\text{H\"older}}{\lesssim}\delta^{l-\frac{1}{2}}\big\|\langle u_+\rangle^{1+\sigma}\partial_h^{\alpha_h-\beta_h}\partial_3^{l-l_1}z_-^h\big\|_{L^\infty_tL^2_x}\Big\|\frac{\langle u_-\rangle^{1+\sigma}}{\langle u_+\rangle^{\frac{1}{2}(1+\sigma)}}\nabla_h^{|\beta_h|+1}\partial_3^{l_1} z_{+}\Big\|_{L^2_tL^\infty_x}\\
		& \stackrel{\text{Lemma } \ref{lemma:sobolev}}{\lesssim}\delta^{l-\frac{1}{2}}\big\|\langle u_+\rangle^{1+\sigma}\partial_h^{\alpha_h-\beta_h}\partial_3^{l-l_1}z_-^h\big\|_{L^\infty_tL^2_x} \sum_{k_2+l_2\leqslant 2}\delta^{l_2-\frac{1}{2}}\Big\|\nabla_h^{k_2}\partial_3^{l_2}\Big(\frac{\langle u_-\rangle^{1+\sigma}}{\langle u_+\rangle^{\frac{1}{2}(1+\sigma)}}\nabla_h^{|\beta_h|+1}\partial_3^{l_1} z_{+}\Big)\Big\|_{L^2_tL^2_x}\\
		&\stackrel{\eqref{eq:weightde}}{\lesssim}\delta^{l-\frac{1}{2}}\big\|\langle u_+\rangle^{1+\sigma}\partial_h^{\alpha_h-\beta_h}\partial_3^{l-l_1}z_-^h\big\|_{L^\infty_tL^2_x}\sum_{k_2+l_2\leqslant 2}\delta^{l_2-\frac{1}{2}}\Big\|\frac{\langle u_-\rangle^{1+\sigma}}{\langle u_+\rangle^{\frac{1}{2}(1+\sigma)}}\nabla_h^{|\beta_h|+1+k_2}\partial_3^{l_1+l_2} z_{+}\Big\|_{L^2_tL^2_x}\\
		&\lesssim\sum_{k_1\leqslant|\alpha_h|}\delta^{l-l_1-\frac{1}{2}} \big(E_-^{(k_1,l-l_1)}(z_-^h)\big)^{\frac{1}{2}}\sum_{k_2+l_2\leqslant N+ 2}\delta^{l_2-\frac{1}{2}}\big(F_+^{(k_2,l_2)}(z_+)\big)^{\frac{1}{2}}
		\stackrel{\eqref{improve2}}{\lesssim} C_1\varepsilon^2,\stepcounter{equation}\tag{\theequation}\label{eqA1}\\
		&\ \ \ \ \delta^{l-\frac{1}{2}}\big\|\langle u_-\rangle^{1+\sigma}\langle u_+\rangle^{\frac{1}{2}(1+\sigma)}\mathbf{J}_{+,2}^{(\beta_h,l_1)}\big\|_{L^2_tL^2_x}\\
		&\stackrel{\text{H\"older}}{\lesssim}\delta^{l-\frac{1}{2}}\big\|\langle u_+\rangle^{1+\sigma}\partial_h^{\alpha_h-\beta_h}\partial_3^{l-l_1}z_-^3\big\|_{L^\infty_tL^2_x}\Big\|\frac{\langle u_-\rangle^{1+\sigma}}{\langle u_+\rangle^{\frac{1}{2}(1+\sigma)}}\nabla_h^{|\beta_h|}\partial_3^{l_1+1} z_{+}\Big\|_{L^2_tL^\infty_x}\\
		&\stackrel{\text{Lemma } \ref{lemma:sobolev}}{\lesssim}\delta^{l-\frac{1}{2}}\big\|\langle u_+\rangle^{1+\sigma}\partial_h^{\alpha_h-\beta_h}\partial_3^{l-l_1}z_-^3\big\|_{L^\infty_tL^2_x} \sum_{k_2+l_2\leqslant 2}\delta^{l_2-\frac{1}{2}}\Big\|\nabla_h^{k_2}\partial_3^{l_2}\Big(\frac{\langle u_-\rangle^{1+\sigma}}{\langle u_+\rangle^{\frac{1}{2}(1+\sigma)}}\nabla_h^{|\beta_h|}\partial_3^{l_1+1} z_{+}\Big)\Big\|_{L^2_tL^2_x}\\
		&\stackrel{\eqref{eq:weightde}}{\lesssim}\delta^{l-\frac{1}{2}}\big\|\langle u_+\rangle^{1+\sigma}\partial_h^{\alpha_h-\beta_h}\partial_3^{l-l_1}z_-^3\big\|_{L^\infty_tL^2_x}\sum_{k_2+l_2\leqslant 2}\delta^{l_2-\frac{1}{2}}\Big\|\frac{\langle u_-\rangle^{1+\sigma}}{\langle u_+\rangle^{\frac{1}{2}(1+\sigma)}}\nabla_h^{|\beta_h|+k_2}\partial_3^{l_1+1+l_2} z_{+}\Big\|_{L^2_tL^2_x}\\
		&\lesssim\sum_{k_1\leqslant|\alpha_h|}\delta^{l-l_1-\frac{1}{2}} \big(E_-^{(k_1,l-l_1)}(z_-^3)\big)^{\frac{1}{2}}\sum_{k_2+l_2\leqslant N+ 1}\delta^{l_2-\frac{1}{2}}\big(F_+^{(k_2,l_2)}(\partial_3z_+)\big)^{\frac{1}{2}}
		\stackrel{\eqref{improve2}}{\lesssim} C_1\varepsilon^2.\stepcounter{equation}\tag{\theequation}\label{eqA2}
	\end{align*}
	Together with \eqref{eq:I1I2}, these two estimates \eqref{eqA1}-\eqref{eqA2} lead us to the desired result for 
	this case.

	\subsubsection*{\bf Case 2:  $N\leqslant|\beta_h|+l_1\leqslant |\alpha_h|+l\leqslant 2N-1$}
	By virtue of 
	$(|\beta_h|+l_1+1)+2> N+2$, 
	the terms related to $z_+$ now
	cannot be controlled by the $L^\infty_x$ estimates via the  Sobolev lemma and energy flux estimates as the previous case. Likewise, we also turn our attention to $L^\infty_x$ estimates on the terms related to $z_-$ at this time:
	\begin{align*}
		&\ \ \ \ \delta^{l-\frac{1}{2}}	\big\|\langle u_-\rangle^{1+\sigma}\langle u_+\rangle^{\frac{1}{2}(1+\sigma)}\mathbf{J}_{+,1}^{(\beta_h,l_1)}\big\|_{L^2_tL^2_x}\\
		&\stackrel{\text{H\"older}}{\lesssim}\delta^{l-\frac{1}{2}}\big\|\langle u_+\rangle^{1+\sigma}\partial_h^{\alpha_h-\beta_h}\partial_3^{l-l_1}z_-^h\big\|_{L^\infty_tL^\infty_x}\Big\|\frac{\langle u_-\rangle^{1+\sigma}}{\langle u_+\rangle^{\frac{1}{2}(1+\sigma)}}\nabla_h^{|\beta_h|+1}\partial_3^{l_1} z_{+}\Big\|_{L^2_tL^2_x}\\
		& \stackrel{\text{Lemma } \ref{lemma:sobolev}}{\lesssim} \delta^{l-\frac{1}{2}}\sum_{k_2+l_2\leqslant 2}\delta^{l_2-\frac{1}{2}}\Big\|\nabla_h^{k_2}\partial_3^{l_2}\big(\langle u_+\rangle^{1+\sigma}\partial_h^{\alpha_h-\beta_h}\partial_3^{l-l_1}z_-^h\big)\Big\|_{L^\infty_tL^2_x}\Big\|\frac{\langle u_-\rangle^{1+\sigma}}{\langle u_+\rangle^{\frac{1}{2}(1+\sigma)}}\nabla_h^{|\beta_h|+1}\partial_3^{l_1} z_{+}\Big\|_{L^2_tL^2_x}\\
		&\stackrel{\eqref{eq:weightde}}{\lesssim} \delta^{l-\frac{1}{2}}\sum_{k_2+l_2\leqslant 2}\delta^{l_2-\frac{1}{2}}\big\|\langle u_+\rangle^{1+\sigma}\nabla_h^{|\alpha_h-\beta_h|+k_2}\partial_3^{l-l_1+l_2}z_-^h\big\|_{L^\infty_tL^2_x}\Big\|\frac{\langle u_-\rangle^{1+\sigma}}{\langle u_+\rangle^{\frac{1}{2}(1+\sigma)}}\nabla_h^{|\beta_h|+1}\partial_3^{l_1} z_{+}\Big\|_{L^2_tL^2_x}\\
		&\lesssim \sum_{k_2+l_2\leqslant N+1}\delta^{l_2-\frac{1}{2}} \big(E_-^{(k_2,l_2)}(z_-^h)\big)^{\frac{1}{2}}\sum_{k_1\leqslant|\alpha_h|+1}\delta^{l_1-\frac{1}{2}} \big(F_+^{(k_1,l_1)}(z_+)\big)^{\frac{1}{2}}
		\stackrel{\eqref{improve2}}{\lesssim} C_1\varepsilon^2,\stepcounter{equation}\tag{\theequation}\label{eqA3}\\
		&\ \ \ \ 	\delta^{l-\frac{1}{2}}\big\|\langle u_-\rangle^{1+\sigma}\langle u_+\rangle^{\frac{1}{2}(1+\sigma)}\mathbf{J}_{+,2}^{(\beta_h,l_1)}\big\|_{L^2_tL^2_x}\\
		&\stackrel{\text{H\"older}}{\lesssim}\delta^{l-\frac{1}{2}}\big\|\langle u_+\rangle^{1+\sigma}\partial_h^{\alpha_h-\beta_h}\partial_3^{l-l_1}z_-^3\big\|_{L^\infty_tL^\infty_x}\Big\|\frac{\langle u_-\rangle^{1+\sigma}}{\langle u_+\rangle^{\frac{1}{2}(1+\sigma)}}\nabla_h^{|\beta_h|}\partial_3^{l_1+1} z_{+}\Big\|_{L^2_tL^2_x}\\
		&\stackrel{\text{Lemma } \ref{lemma:sobolev}}{\lesssim} \delta^{l-\frac{1}{2}}\sum_{k_2+l_2\leqslant 2}\delta^{l_2-\frac{1}{2}}\Big\|\nabla_h^{k_2}\partial_3^{l_2}\big(\langle u_+\rangle^{1+\sigma}\partial_h^{\alpha_h-\beta_h}\partial_3^{l-l_1}z_-^3\big)\Big\|_{L^\infty_tL^2_x}\Big\|\frac{\langle u_-\rangle^{1+\sigma}}{\langle u_+\rangle^{\frac{1}{2}(1+\sigma)}}\nabla_h^{|\beta_h|}\partial_3^{l_1+1} z_{+}\Big\|_{L^2_tL^2_x}\\
		&\stackrel{\eqref{eq:weightde}}{\lesssim} \delta^{l-\frac{1}{2}}\sum_{k_2+l_2\leqslant 2}\delta^{l_2-\frac{1}{2}}\big\|\langle u_+\rangle^{1+\sigma}\nabla_h^{|\alpha_h-\beta_h|+k_2}\partial_3^{l-l_1+l_2}z_-^3\big\|_{L^\infty_tL^2_x}\Big\|\frac{\langle u_-\rangle^{1+\sigma}}{\langle u_+\rangle^{\frac{1}{2}(1+\sigma)}}\nabla_h^{|\beta_h|}\partial_3^{l_1+1} z_{+}\Big\|_{L^2_tL^2_x}\\
		&\lesssim\sum_{k_2+l_2\leqslant N+1}\delta^{l_2-\frac{3}{2}} \big(E_-^{(k_2,l_2)}(z_-^3)\big)^{\frac{1}{2}}\cdot\delta^{l_1+\frac{1}{2}} \big(F_+^{(|\beta_h|,l_1+1)}(z_+)\big)^{\frac{1}{2}}\\
		&=\Big(\sum_{k_2\leqslant N+1}\delta^{-\frac{3}{2}} \big(E_-^{(k_2,0)}(z_-^3)\big)^{\frac{1}{2}}+\sum_{k_2+l_2\leqslant N+1,\ l_2-1\geqslant 0}\!\!\!\!\!\!\delta^{l_2-\frac{3}{2}} \big(E_-^{(k_2,l_2-1)}(\underbrace{\partial_3z_-^3}_{=-\partial_hz_-^h})\big)^{\frac{1}{2}}\Big)\cdot\delta^{l_1+\frac{1}{2}} \big(F_+^{(|\beta_h|,l_1+1)}(z_+)\big)^{\frac{1}{2}}\\
		&\lesssim\Big(\sum_{k_2\leqslant N+1}\delta^{-\frac{3}{2}} \big(E_-^{(k_2,0)}(z_-^3)\big)^{\frac{1}{2}}+\sum_{k_2+1+l_2-1\leqslant N+1,\ l_2-1\geqslant 0}\!\!\!\!\!\!\delta^{l_2-\frac{3}{2}} \big(E_-^{(k_2+1,l_2-1)}(z_-^h)\big)^{\frac{1}{2}}\Big)\cdot\delta^{l_1+\frac{1}{2}} \big(F_+^{(|\beta_h|,l_1+1)}(z_+)\big)^{\frac{1}{2}}\\
		&\lesssim\Big(\sum_{k_2\leqslant N+1}\delta^{-\frac{3}{2}} \big(E_-^{(k_2,0)}(z_-^3)\big)^{\frac{1}{2}}+\sum_{k_2+l_2\leqslant N+1}\delta^{l_2-\frac{1}{2}} \big(E_-^{(k_2,l_2)}(z_-^h)\big)^{\frac{1}{2}}\Big)\sum_{k_1\leqslant|\alpha_h|}\delta^{l_1+\frac{1}{2}} \big(F_+^{(k_1,l_1+1)}(z_+)\big)^{\frac{1}{2}}
		\stackrel{\eqref{improve2}}{\lesssim} C_1\varepsilon^2.\stepcounter{equation}\tag{\theequation}\label{eqA4}
	\end{align*}
	In such a case, putting the estimates \eqref{eq:I1I2} and \eqref{eqA3}-\eqref{eqA4} together also gives the desired result.

	This ends the proof in the same manner.
\end{proof}

\begin{remark}
	By virtue of $\operatorname{div} z_\pm=0$, we notice that  $\partial_3z_\pm^3=-\partial_hz_\pm^h$. In fact, we have adopted this observation in the proof of \eqref{eqA4}. This can help us transform 
	some direct estimates on $\partial_3z_\pm^3$, which should have been bad, to good estimates on $\partial_hz_\pm^h$ and thus make up for 
	 the possible loss of $\delta$ in coefficients. 
	We remark that  this
	observation will be frequently used in this paper without further comment.
\end{remark}

\begin{remark}[Estimate for $\mathbf{J}_{\pm(\delta)}^{(\alpha_h,l)}$]\label{remarkJ}
Based on the notations in Section \ref{sec:notation},   Lemma \ref{lemma2} also gives the following renormalization result immediately: 
	For any $\alpha_h\in(\mathbb{Z}_{\geqslant 0})^2$ and any $l\in\mathbb{Z}_{\geqslant 0}$ with $0\leqslant|\alpha_h|+l\leqslant N+2$, 
	there holds
	\[\big\|\langle u_\mp\rangle^{1+\sigma}\langle u_\pm\rangle^{\frac{1}{2}(1+\sigma)}\mathbf{J}_{\pm(\delta)}^{(\alpha_h,l)}\big\|_{L^2_tL^2_x}
	\lesssim C_1\varepsilon^2,\]
	where 
	\[\mathbf{J}_{\pm(\delta)}^{(\alpha_h,l)}:=\partial_h^{\alpha_h}\partial_3^{l}(z_{\mp(\delta)}\cdot \nabla z_{\pm(\delta)}).\]
	This result will be used in the approximation part of Section \ref{sec:approximation}.
\end{remark}
  
\begin{lemma}[Estimate for $\mathbf{K}_\pm^{(\alpha_h,l)}$]\label{lemma3}
	For any $\alpha_h\in(\mathbb{Z}_{\geqslant 0})^2$ and any $l\in\mathbb{Z}_{\geqslant 0}$ with $0\leqslant|\alpha_h|+l\leqslant N+2$, 
	there holds
	\[\delta^{l-\frac{1}{2}}\big\|\langle u_\mp\rangle^{1+\sigma}\langle u_\pm\rangle^{\frac{1}{2}(1+\sigma)}\mathbf{K}_\pm^{(\alpha_h,l)}\big\|_{L^2_tL^2_x}
	\lesssim C_1\varepsilon^2,\]
	where 
	\[\mathbf{K}_\pm^{(\alpha_h,l)}:=\partial_h^{\alpha_h}\partial_3^{l}\partial_3(z_\mp\cdot \nabla z_\pm).\]
\end{lemma}
\begin{proof}
	Based on symmetry, it suffices to consider $\mathbf{K}_{+}^{(\alpha_h,l)}$. We also note that 
		\begin{align*} \big|\mathbf{K}_+^{(\alpha_h,l)}\big|
		&\lesssim\sum_{\beta_h\leqslant\alpha_h\atop l_1\leqslant l+1}\Big(\underbrace{\big|\partial_h^{\alpha_h-\beta_h}\partial_3^{l+1-l_1}z_-^h\big|\cdot \big|\nabla_h^{|\beta_h|+1}\partial_3^{l_1} z_{+}\big|}_{\mathbf{K}_{+,1}^{(\beta_h,l_1)}}+\underbrace{\big|\partial_h^{\alpha_h-\beta_h}\partial_3^{l+1-l_1}z_-^3\big|\cdot \big|\nabla_h^{|\beta_h|}\partial_3^{l_1+1} z_{+}\big|}_{\mathbf{K}_{+,2}^{(\beta_h,l_1)}}\Big).
	\end{align*}
Similar to \eqref{eqA3}-\eqref{eqA4}, we can infer that 
	\begin{align*}
	&\ \ \ \ \delta^{l-\frac{1}{2}}	\big\|\langle u_-\rangle^{1+\sigma}\langle u_+\rangle^{\frac{1}{2}(1+\sigma)}\mathbf{K}_{+,1}^{(\beta_h,l_1)}\big\|_{L^2_tL^2_x}\\
	&\lesssim \sum_{k_2+l_2\leqslant N+5}\delta^{l_2-\frac{1}{2}} \big(E_-^{(k_2,l_2)}(z_-^h)\big)^{\frac{1}{2}}\sum_{k_1\leqslant|\alpha_h|+1}\delta^{l_1-\frac{1}{2}} \big(F_+^{(k_1,l_1)}(z_+)\big)^{\frac{1}{2}}
	\lesssim C_1\varepsilon^2,\\
	&\ \ \ \ 	\delta^{l-\frac{1}{2}}\big\|\langle u_-\rangle^{1+\sigma}\langle u_+\rangle^{\frac{1}{2}(1+\sigma)}\mathbf{K}_{+,2}^{(\beta_h,l_1)}\big\|_{L^2_tL^2_x}\\
	&\lesssim\Big(\sum_{k_2\leqslant N+5}\delta^{-\frac{3}{2}} \big(E_-^{(k_2,0)}(z_-^3)\big)^{\frac{1}{2}}+\sum_{k_2+l_2\leqslant N+5}\delta^{l_2-\frac{1}{2}} \big(E_-^{(k_2,l_2)}(z_-^h)\big)^{\frac{1}{2}}\Big)\sum_{k_1\leqslant|\alpha_h|}\delta^{l_1-\frac{1}{2}} \big(F_+^{(k_1,l_1)}(\partial_3z_+)\big)^{\frac{1}{2}}
	\lesssim  C_1\varepsilon^2.
\end{align*}
Therefore we can conclude from these estimates to obtain the lemma. 
\end{proof}

\section{Estimates on the pressure term}\label{sec:pressure}
Due to the facts $\operatorname{div} z_\pm=0$ and the boundary conditions $z_+^3|_{x_3=\pm\delta}=0$, $z_-^3|_{x_3=\pm\delta}=0$, we turn to
take divergence of the first equation in \eqref{MHD equation} to derive the following system in 
$\Omega_\delta$:
\begin{equation}\label{eqpressure-system}
	\begin{cases}
		&-\Delta p=\partial_iz_+^j\partial_jz_-^i,\\
		&\partial_3p\big|_{x_3=\pm\delta}=0.
	\end{cases}
\end{equation}
Using the Green's function $G_\delta(x,y)$ on $\Omega_\delta\times\Omega_\delta$, we 
solve this system \eqref{eqpressure-system} and write the pressure as
\begin{equation}\label{eqpressure-formula}
	 p(t,x)=\int_{\Omega_\delta}  G_\delta(x,y)(\partial_iz_+^j\partial_jz_-^i)(t,y)dy.
\end{equation} 
However, it is rather difficult to give an explicit formula of $G_\delta(x,y)$ herein (up to pointwise convergence), which hinders further investigation on the pressure term $p$. Nevertheless, as we will see, what really counts in the rest of this paper is its gradient, and therefore we only need to give the explicit formula of $\nabla p$. 

In fact,
taking gradient of \eqref{eqpressure-formula} immediately yields 
\begin{equation}\label{nabla pressure}
	\nabla p(t,x)=\int_{\Omega_\delta} \nabla_xG_\delta(x,y)(\partial_iz_+^j\partial_jz_-^i)(t,y)dy,
\end{equation}
where
$\nabla_xG_\delta(x,y)$ is well defined on $\Omega_\delta\times\Omega_\delta$ by
\begin{equation}\label{Green1} \nabla_xG_\delta(x,y):=\frac{1}{4\pi}\sum_{k=-\infty}^{\infty}\nabla_x\frac{1}{\big(|x_h-y_h|^2+|(-1)^k(x_3-2k\delta)-y_3|^2\big)^{\frac{1}{2}}}.
\end{equation}
Here we refer the interested readers to Lemma 2.1 in \cite{Xu} for more details of \eqref{nabla pressure} and \eqref{Green1}. 
Before proceeding further, we first review Corollary 2.3 in \cite{Xu} with  proof, which provides an important bound on derivatives of the Green's function.
\begin{lemma}\label{lemma:boundgreen}
	For any $l\in\mathbb{Z}_{\geqslant 1}$, there holds
	\begin{equation}\label{bound of Green}
		\nabla_x^lG_\delta(x,y)\lesssim\frac{1}{\delta}\frac{1}{|x_h-y_h|^l}.
	\end{equation} 

\end{lemma}
\begin{proof}
First of all, \eqref{Green1} can be rephrased as
\begin{align*}
\nabla_xG_\delta(x,y)
	&=\frac{1}{4\pi}\bigg(\sum_{k=0}+\sum_{k=1}^{\infty}+\sum_{k=-\infty}^{-1}\bigg)\nabla_x\frac{1}{\big(|x_h-y_h|^2+|(-1)^k(x_3-2k\delta)-y_3|^2\big)^{\frac{1}{2}}}\\
	&=\frac{1}{4\pi}\bigg(\nabla_x\frac{1}{\big(|x_h-y_h|^2+|x_3-y_3|^2\big)^{\frac{1}{2}}}
	\\&\ \ \ \ \ \ \ \ \ 
	 +\sum_{k=1}^{\infty}\Big(\nabla_x\frac{1}{\big(|x_h-y_h|^2+|(x_{+,k})_3-y_3|^2\big)^{\frac{1}{2}}}+\nabla_x\frac{1}{\big(|x_h-y_h|^2+|(x_{-,k})_3-y_3|^2\big)^{\frac{1}{2}}}\Big)\bigg)\\
	&=\frac{1}{4\pi}\bigg(\nabla_x\frac{1}{|x-y|}+\sum_{k=1}^{\infty}\Big(\nabla_x\frac{1}{|x_{+,k}-y|}+\nabla_x\frac{1}{|x_{-,k}-y|}\Big)\bigg),
\end{align*}
where we denote
\begin{equation*}
		x_{+,k}=\big(x_h,(-1)^k(x_3-2k\delta)\big),\ \  x_{-,k}=\big(x_h,(-1)^k(x_3+2k\delta)\big),\ \   k\in\mathbb{Z}_{\geqslant 1}.
\end{equation*}
 
For any $l\in\mathbb{Z}_{\geqslant 1}$,  it subsequently follows that 
	\begin{equation*} 
		\nabla_x^lG_\delta(x,y)=\frac{1}{4\pi}\Big(\nabla_x^l\frac{1}{|x-y|}+\sum_{k=1}^\infty\big(\nabla_x^l\frac{1}{|x_{+,k}-y|}+\nabla_x^l\frac{1}{|x_{-,k}-y|}\big)\Big).
	\end{equation*}
By virtue of $(x,y)\in\Omega_\delta\times\Omega_\delta$, we have  $x_3,y_3\in(-\delta,\delta)$ and then  $\nabla_x^lG_\delta(x,y)$ can be bounded as follows:
\begin{align*}
	&\ \ \ \ |\nabla_x^lG_\delta(x,y)|\\
	&\lesssim\frac{1}{4\pi}\bigg(\frac{1}{|x-y|^{l+1}}+\sum_{k=1}^\infty\Big(\frac{1}{|x_{+,k}-y|^{l+1}}+\frac{1}{|x_{-,k}-y|^{l+1}}\Big)\bigg)\\
	&=\frac{1}{4\pi}\Bigg(\frac{1}{\big(|x_h-y_h|^2+|x_3-y_3|^2\big)^{\frac{l+1}{2}}}\\ 
	&\ \ \ \ \ \ \ \ +\sum_{k=1}^\infty\bigg(\frac{1}{\big(|x_h-y_h|^2+|x_3-(-1)^ky_3-2k\delta|^2\big)^{\frac{l+1}{2}}}+\frac{1}{\big(|x_h-y_h|^2+|x_3-(-1)^ky_3+2k\delta|^2\big)^{\frac{l+1}{2}}}\bigg)\Bigg)\\
	&\leqslant\frac{1}{4\pi}\Big(\frac{1}{|x_h-y_h|^{l+1}}+\sum_{k=1}^\infty\frac{2}{\big(|x_h-y_h|^2+|(2k-1)\delta|^2\big)^{\frac{l+1}{2}}}\Big)
	\leqslant\frac{3}{4\pi}\sum_{k=0}^\infty\frac{1}{\big(|x_h-y_h|^2+|2k\delta|^2\big)^{\frac{l+1}{2}}}\\
	&\leqslant \frac{3}{4\pi}\int_0^\infty\frac{1}{\big(|x_h-y_h|^2+|2\tau\delta|^2\big)^{\frac{l+1}{2}}}d\tau=\frac{3}{4\pi}\frac{1}{2\delta}\int_0^\infty\frac{1}{\big(|x_h-y_h|^2+|2\delta\tau|^2\big)^{\frac{l+1}{2}}}d(2\delta\tau)\\
	&=\frac{3}{4\pi}\frac{1}{2\delta}\int_0^\infty\frac{1}{\big(|x_h-y_h|^2+(|x_h-y_h|u)^2\big)^{\frac{l+1}{2}}}d(|x_h-y_h|u)\\
	&=\frac{3}{4\pi}\frac{1}{2\delta}\frac{1}{|x_h-y_h|^l}\int_0^\infty\frac{1}{\big(1+u^2\big)^{\frac{l+1}{2}}}du\lesssim\frac{1}{\delta}\frac{1}{|x_h-y_h|^l}.
\end{align*}
This proves \eqref{bound of Green} for all $l\in\mathbb{Z}_{\geqslant 1}$. 
Moreover, when $l=1$, this bound implies that the right hand side of \eqref{Green1} is summable, which guarantees the well-definedness of the definition  \eqref{Green1}. The proof of the lemma is now complete.
\end{proof}

Furthermore, to avoid the non-integrable singularity $\frac{1}{|x_h-y_h|}$,
the following facts on $\frac{1}{|x_h|^\gamma}$ are necessary throughout the pressure estimates:
\begin{equation}\label{observe integration}\begin{split}
		&\frac{1}{|x_h|^\gamma}\text{\Large$\text{\Large$\chi$}$}_{|x_h|\leqslant 2}\in L^1(\mathbb{R}^2)\iff \gamma<2;\ \ \ \ \ \
		\frac{1}{|x_h|^\gamma}\text{\Large$\text{\Large$\chi$}$}_{|x_h|\geqslant 1}\in L^1(\mathbb{R}^2)\iff \gamma>2;\\
		&\frac{1}{|x_h|^\gamma}\text{\Large$\text{\Large$\chi$}$}_{|x_h|\leqslant 2}\in L^2(\mathbb{R}^2)\iff \gamma<1;\ \ \ \ \ \ 
		\frac{1}{|x_h|^\gamma}\text{\Large$\text{\Large$\chi$}$}_{|x_h|\geqslant 1}\in L^2(\mathbb{R}^2)\iff \gamma>1.
	\end{split}
\end{equation}

Now we can proceed  
to derive bounds on derivatives of the pressure term; see the following lemmas.

\begin{lemma}\label{lemma:H}
	For any $\alpha_h\in(\mathbb{Z}_{\geqslant 0})^2$ with 
	$0\leqslant|\alpha_h|\leqslant N+2$, there holds 
	\begin{equation*}
		\delta^{-\frac{1}{2}}\big\|\langle u_-\rangle^{1+\sigma}\langle u_+\rangle^{\frac{1}{2}(1+\sigma)}\partial_{h}^{\alpha_h} \nabla p\big\|_{L^2_tL^2_x}
		\lesssim C_1\varepsilon^2.
	\end{equation*}
\end{lemma}
\begin{proof}
	We start from \eqref{nabla pressure} together with the facts $\operatorname{div}z_{\pm}=0$ and the boundary conditions $z^3_{+}|_{x_3=\pm\delta}=0$, $z^3_{-}|_{x_3=\pm\delta}=0$. 
	Applying derivatives $\partial_h^{\alpha_h}$ to $\nabla p$ and 	integrating by parts, we obtain
	\begin{align*}
		\partial_h^{\alpha_h} \nabla p(\tau,x)
		&=\int_{\Omega_\delta} \partial_{x_h}^{\alpha_h} \nabla_x G_\delta(x,y)(\partial_iz_+^j\partial_jz_-^i)(\tau,y)dy
		=\int_{\Omega_\delta} \partial_i\partial_j\partial_{x_h}^{\alpha_h} \nabla_x G_\delta(x,y)(z_+^jz_-^i)(\tau,y)dy\\
		&=(-1)^{|\alpha_h|}\int_{\Omega_\delta}\partial_i\partial_j \partial_{y_h}^{\alpha_h} \nabla_yG_\delta(x,y)(z_+^jz_-^i)(\tau,y)dy
		=-\int_{\Omega_\delta}\partial_i\partial_j G_\delta(x,y)\partial_{y_h}^{\alpha_h}\nabla_y(z_+^jz_-^i)(\tau,y)dy\\
		&=\underbrace{-\int_{\Omega_\delta}\partial_i\partial_j G_\delta(x,y)\theta(|x_h-y_h|)\partial_{h}^{\alpha_h}\nabla(z_+^jz_-^i)(\tau,y)dy}_{\mathbf{L}_{\mathbf{1}}^{(\alpha_h,0)}(\tau,x)}\\
		&\ \ \ \ \underbrace{-\int_{\Omega_\delta}\partial_i\partial_j G_\delta(x,y)\big(1-\theta(|x_h-y_h|)\big)\partial_{h}^{\alpha_h}\nabla(z_+^jz_-^i)(\tau,y)dy}_{\mathbf{L}_{\mathbf{2}}^{(\alpha_h,0)}(\tau,x)},\stepcounter{equation}\tag{\theequation}\label{partialp-decomposition}
	\end{align*} 
	where the smooth cut-off function $\theta(r)$ is chosen so that
	\begin{equation}\label{cutoff}
		\theta(r)=\begin{cases}
			&	1,\ \ \ \ \text{for }|r|\leqslant 1,\\
			&	0,\ \ \ \ \text{for }|r|\geqslant 2.
		\end{cases}
	\end{equation} 
	In view of the property of the cut-off function $\theta(r)$, we derive 
	\begin{align*}
		\delta^{-\frac{1}{2}}\big\|\langle u_-\rangle^{1+\sigma}\langle u_+\rangle^{\frac{1}{2}(1+\sigma)}\partial_h^{\alpha_h}\nabla p\big\|_{L^2_tL^2_x}&\lesssim \delta^{-\frac{1}{2}}\big\|\langle u_-\rangle^{1+\sigma}\langle u_+\rangle^{\frac{1}{2}(1+\sigma)}\mathbf{L}_{\mathbf{1}}^{(\alpha_h,0)}\big\|_{L^2_tL^2_x}\\
		&\ \ \ \ +\delta^{-\frac{1}{2}}\big\|\langle u_-\rangle^{1+\sigma}\langle u_+\rangle^{\frac{1}{2}(1+\sigma)}\mathbf{L}_{\mathbf{2}}^{(\alpha_h,0)}\big\|_{L^2_tL^2_x}.\stepcounter{equation}\tag{\theequation}\label{eq:H1H2}
	\end{align*}

	\subsubsection*{\bf Estimate of $\mathbf{L}_{\mathbf{1}}^{(\alpha_h,0)}$}
	There holds the following decomposition:
	\begin{equation}\label{eqH1-decomposition-total}
		\mathbf{L}_{\mathbf{1}}^{(\alpha_h,0)}(\tau,x)
		=\sum_{\beta_h\leqslant\alpha_h}C_{\alpha_h,\beta_h}\underbrace{\int_{\Omega_\delta}\partial_i\partial_j G_\delta(x,y)\theta(|x_h-y_h|)\nabla\big(\partial_{h}^{\alpha_h-\beta_h}z_+^j\partial_{h}^{\beta_h}z_-^i\big)(\tau,y)dy}_{\mathbf{L}_{\mathbf{1}}^{(\beta_h,0)}(\tau,x)}.
	\end{equation}
	According to the number of derivatives, we will distinguish two cases: 
	\[0\leqslant|\beta_h|\leqslant N-1\ \ \text{ and }\ \ N\leqslant|\beta_h|\leqslant |\alpha_h|.\]
	\subsubsection*{\bf Case 1:  $0\leqslant|\beta_h|\leqslant N-1$}
	Thanks to the fact $\operatorname{div}z_+=0$ and the boundary condition $ z^3_{+}|_{x_3=\pm\delta}=0$, integration by parts similar to \eqref{partialp-decomposition} then yields
	\begin{align*}
		\mathbf{L}_{\mathbf{1}}^{(\beta_h,0)}(\tau,x)
		&= \underbrace{-\int_{\Omega_\delta}\partial_i G_\delta(x,y)\theta(|x_h-y_h|)\nabla\big(\partial_{h}^{\alpha_h-\beta_h}z_+^j\partial_{h}^{\beta_h}\partial_jz_-^i\big)(\tau,y)dy}_{ \mathbf{L}_{\mathbf{11}}^{(\beta_h,0)}(\tau,x)}\\
		&\ \ \ \   \underbrace{-\int_{\Omega_\delta}\partial_i G_\delta(x,y)\partial_j\theta(|x_h-y_h|)\nabla\big(\partial_{h}^{\alpha_h-\beta_h}z_+^j\partial_{h}^{\beta_h}z_-^i\big)(\tau,y)dy}_{ \mathbf{L}_{\mathbf{12}}^{(\beta_h,0)}(\tau,x)}.\stepcounter{equation}\tag{\theequation}\label{eqH1-decomposition}
	\end{align*}
	Using definition, we deduce that  
	\begin{align*}
		&\ \ \ \ \delta^{-\frac{1}{2}}\big\|\langle u_-\rangle^{1+\sigma}\langle u_+\rangle^{\frac{1}{2}(1+\sigma)}\mathbf{L}_{\mathbf{11}}^{(\beta_h,0)}(\tau,x)\big\|_{L^2_tL^2_x}\\
		&\stackrel{\text{Lemma \ref{lemma:boundgreen}}}{\lesssim}\delta^{-\frac{1}{2}}\Big\|\langle u_-\rangle^{1+\sigma}\langle u_+\rangle^{\frac{1}{2}(1+\sigma)}\int_{-\delta}^{\delta}\int_{|x_h-y_h|\leqslant 2}\frac{1}{\delta}\frac{1}{|x_h-y_h|} \big|\nabla\big(\partial_{h}^{\alpha_h-\beta_h}z_+^j\partial_{h}^{\beta_h}\partial_jz_-^i\big)(\tau,y)\big|dy_hdy_3\Big\|_{L^2_tL^2_{x_h}L^2_{x_3}}\\
		&\stackrel{\eqref{eq:weight1}}{\lesssim}\delta^{-\frac{1}{2}}\Big\|\int_{-\delta}^{\delta}\int_{|x_h-y_h|\leqslant 2}\frac{1}{\delta}\frac{1}{|x_h-y_h|} \Big(\langle u_-\rangle^{1+\sigma}\langle u_+\rangle^{\frac{1}{2}(1+\sigma)}\big|\nabla(\partial_{h}^{\alpha_h-\beta_h}z_+^j\partial_{h}^{\beta_h}\partial_jz_-^i)\big|\Big)(\tau,y)dy_hdy_3\Big\|_{L^2_tL^2_{x_h}L^2_{x_3}}\\
		&=\delta^{-\frac{3}{2}}\bigg(\!\int_0^t\!\int_{-\delta}^{\delta}\Big\|\!\int_{-\delta}^{\delta}\!\int_{|x_h-y_h|\leqslant 2}\!\frac{1}{|x_h-y_h|} \Big(\!\langle u_-\rangle^{1+\sigma}\langle u_+\rangle^{\frac{1}{2}(1+\sigma)}\big|\nabla(\partial_{h}^{\alpha_h-\beta_h}z_+^j\partial_{h}^{\beta_h}\partial_jz_-^i)\big|\Big)(\tau,y)dy_hdy_3\Big\|^2_{L^2_{x_h}}\!\!dx_3d\tau\!\bigg)^{\frac{1}{2}}\\
		&=\delta^{-\frac{3}{2}}\bigg(\int_0^t2\delta\Big\|\int_{-\delta}^{\delta}\int_{|x_h-y_h|\leqslant 2}\frac{1}{|x_h-y_h|} \Big(\langle u_-\rangle^{1+\sigma}\langle u_+\rangle^{\frac{1}{2}(1+\sigma)}\big|\nabla(\partial_{h}^{\alpha_h-\beta_h}z_+^j\partial_{h}^{\beta_h}\partial_jz_-^i)\big|\Big)(\tau,y)dy_hdy_3\Big\|^2_{L^2_{x_h}}d\tau\bigg)^{\frac{1}{2}}\\
		&\stackrel{\text{Minkowski}}{\lesssim}\delta^{-1}\bigg(\!\int_0^t\!\int_{-\delta}^{\delta}\Big\|\int_{|x_h-y_h|\leqslant 2}\!\frac{1}{|x_h-y_h|} \Big(\langle u_-\rangle^{1+\sigma}\!\langle u_+\rangle^{\frac{1}{2}(1+\sigma)}\big|\nabla(\partial_{h}^{\alpha_h-\beta_h}z_+^j\partial_{h}^{\beta_h}\partial_jz_-^i)\big|\Big)(\tau,y)dy_h\Big\|^2_{L^2_{x_h}}\!dy_3d\tau\!\bigg)^{\frac{1}{2}}\\
		&\stackrel{\text{Young}}{\lesssim}\delta^{-1}\bigg(\int_0^t\int_{-\delta}^{\delta}\Big\|\frac{1}{|x_h|}\chi_{|x_h|\leqslant 2}\Big\|^2_{L^1_{x_h}} \Big\|\Big(\langle u_-\rangle^{1+\sigma}\langle u_+\rangle^{\frac{1}{2}(1+\sigma)}\big|\nabla(\partial_{h}^{\alpha_h-\beta_h}z_+^j\partial_{h}^{\beta_h}\partial_jz_-^i)\big|\Big)(\tau,x_h,y_3)\Big\|^2_{L^2_{x_h}}dy_3d\tau\bigg)^{\frac{1}{2}}\\
		&\lesssim\delta^{-1}\Big\|\Big(\langle u_-\rangle^{1+\sigma}\langle u_+\rangle^{\frac{1}{2}(1+\sigma)}\big|\nabla(\partial_{h}^{\alpha_h-\beta_h}z_+^j\partial_{h}^{\beta_h}\partial_jz_-^i)\big|\Big)(\tau,x_h,y_3)\Big\|_{L^2_t L^2_{x_h}L^1_{y_3}}\\
		&\lesssim\delta^{-1}\Big(\Big\|\Big(\langle u_-\rangle^{1+\sigma}\langle u_+\rangle^{\frac{1}{2}(1+\sigma)}\big|\partial_{h}^{\alpha_h-\beta_h}\nabla z_+^j\partial_{h}^{\beta_h}\partial_jz_-^i\big|\Big)(\tau,x_h,y_3)\Big\|_{L^2_t L^2_{x_h}L^1_{y_3}}\\
		&\ \ \ \ +\Big\|\Big(\langle u_-\rangle^{1+\sigma}\langle u_+\rangle^{\frac{1}{2}(1+\sigma)}\big|\partial_{h}^{\alpha_h-\beta_h}z_+^j\partial_{h}^{\beta_h}\partial_j\nabla z_-^i\big|\Big)(\tau,x_h,y_3)\Big\|_{L^2_t L^2_{x_h}L^1_{y_3}}\Big)\\
		&\stackrel{\text{H\"older}}{\lesssim}\delta^{-1}\Big(\big\|\langle u_+\rangle^{1+\sigma}\partial_{h}^{\beta_h}\partial_jz_-^i\big\|_{L^\infty_tL^\infty_{x_h}L^2_{x_3}}\cdot\Big\|\frac{\langle u_-\rangle^{1+\sigma}}{\langle u_+\rangle^{\frac{1}{2}(1+\sigma)}}\partial_{h}^{\alpha_h-\beta_h}\nabla z_+^j\Big\|_{L^2_tL^2_{x_h}L^2_{x_3}}\\
		&\ \ \ \ +\big\|\langle u_+\rangle^{1+\sigma}\partial_{h}^{\beta_h}\partial_j\nabla z_-^i\big\|_{L^\infty_tL^\infty_{x_h}L^2_{x_3}}\cdot\Big\|\frac{\langle u_-\rangle^{1+\sigma}}{\langle u_+\rangle^{\frac{1}{2}(1+\sigma)}}\partial_{h}^{\alpha_h-\beta_h} z_+^j\Big\|_{L^2_tL^2_{x_h}L^2_{x_3}}\Big)\\
		&\stackrel{\text{Sobolev (in $\mathbb{R}^2$)}}{\lesssim}\delta^{-1}\Big(\big\|\langle u_+\rangle^{1+\sigma}\partial_{h}^{\beta_h}\partial_jz_-^i\big\|_{L^\infty_tH^2_{x_h}L^2_{x_3}}\cdot\Big\|\frac{\langle u_-\rangle^{1+\sigma}}{\langle u_+\rangle^{\frac{1}{2}(1+\sigma)}}\partial_{h}^{\alpha_h-\beta_h}\nabla z_+^j\Big\|_{L^2_tL^2_x}\\
		&\ \ \ \ +\big\|\langle u_+\rangle^{1+\sigma}\partial_{h}^{\beta_h}\partial_j\nabla z_-^i\big\|_{L^\infty_tH^2_{x_h}L^2_{x_3}}\cdot\Big\|\frac{\langle u_-\rangle^{1+\sigma}}{\langle u_+\rangle^{\frac{1}{2}(1+\sigma)}}\partial_{h}^{\alpha_h-\beta_h} z_+^j\Big\|_{L^2_tL^2_x}\Big)\\
		&\lesssim \Big(\!\!\!\sum_{k\leqslant N+3}\!\!\!\delta^{-\frac{1}{2}}\big(E_-^{(k,0)}(z_-)\big)^{\frac{1}{2}}+\!\!\!\sum_{k\leqslant N+2}\!\!\!\delta^{-\frac{1}{2}}\big(E_-^{(k,0)}(\partial_3z_-)\big)^{\frac{1}{2}}\Big) 
	 \!\!\cdot\!\!\Big(\!\!\!\sum_{k\leqslant|\alpha_h|+1}\!\!\!\delta^{-\frac{1}{2}}\big(F_+^{(k,0)}(z_+)\big)^{\frac{1}{2}}+\!\!\!\sum_{k\leqslant|\alpha_h|}\!\!\!\delta^{-\frac{1}{2}}\big(F_+^{(k,0)}(\partial_3z_+)\big)^{\frac{1}{2}}\Big)\\
		&\stackrel{\eqref{improve2}}{\lesssim} C_1\varepsilon^2.\stepcounter{equation}\tag{\theequation}\label{eqH11}
	\end{align*}
	We can continue in this fashion to obtain   
	\begin{align*}
		&\ \ \ \ \delta^{-\frac{1}{2}}\big\|\langle u_-\rangle^{1+\sigma}\langle u_+\rangle^{\frac{1}{2}(1+\sigma)}\mathbf{L}_{\mathbf{12}}^{(\beta_h,0)}(\tau,x)\big\|_{L^2_tL^2_x}\\
		&\lesssim \Big(\!\!\!\sum_{k\leqslant N+2}\!\!\!\delta^{-\frac{1}{2}}\big(E_-^{(k,0)}(z_-)\big)^{\frac{1}{2}}+\!\!\!\sum_{k\leqslant N+1}\!\!\!\delta^{-\frac{1}{2}}\big(E_-^{(k,0)}(\partial_3z_-)\big)^{\frac{1}{2}}\Big) 
		\!\! \cdot\!\!\Big(\!\!\!\sum_{k\leqslant|\alpha_h|+1}\!\!\!\delta^{-\frac{1}{2}}\big(F_+^{(k,0)}(z_+)\big)^{\frac{1}{2}}+\!\!\!\sum_{k\leqslant|\alpha_h|}\!\!\!\delta^{-\frac{1}{2}}\big(F_+^{(k,0)}(\partial_3z_+)\big)^{\frac{1}{2}}\Big)\\		&\lesssim C_1\varepsilon^2.	\stepcounter{equation}\tag{\theequation}\label{eqH12}
	\end{align*}
	In this case, combining \eqref{eqH1-decomposition}, \eqref{eqH11} and \eqref{eqH12}, we are able to derive
	\begin{equation}\label{eqH1case1estimate} \delta^{-\frac{1}{2}}\big\|\langle u_-\rangle^{1+\sigma}\langle u_+\rangle^{\frac{1}{2}(1+\sigma)}\mathbf{L}_{\mathbf{1}}^{(\beta_h,0)}(\tau,x)\big\|_{L^2_tL^2_x}\lesssim
		C_1\varepsilon^2.
	\end{equation}
	
	\subsubsection*{\bf Case 2: $N\leqslant|\beta_h|\leqslant |\alpha_h|$}
	By virtue of $\operatorname{div}z_-=0$ and $ z^3_{-}|_{x_3=\pm\delta}=0$,  integration by parts gives
	\begin{align*}
		\mathbf{L}_{\mathbf{1}}^{(\beta_h,0)}(\tau,x)
		&=
		\underbrace{-\int_{\Omega_\delta}\partial_j G_\delta(x,y)\theta(|x_h-y_h|)\nabla\big(\partial_{h}^{\alpha_h-\beta_h}\partial_iz_+^j\partial_{h}^{\beta_h}z_-^i\big)(\tau,y)dy}_{\displaystyle \mathbf{L}_{\mathbf{13}}^{(\beta_h,0)}(\tau,x)}\\
		&\ \ \ \ \underbrace{-\int_{\Omega_\delta}\partial_j G_\delta(x,y)\partial_i\theta(|x_h-y_h|)\nabla\big(\partial_{h}^{\alpha_h-\beta_h}z_+^j\partial_{h}^{\beta_h}z_-^i\big)(\tau,y)dy}_{\displaystyle \mathbf{L}_{\mathbf{14}}^{(\beta_h,0)}(\tau,x)}.
	\end{align*}
	In the same manner, we can see that 
	\begin{align*}
		&\ \ \ \ \delta^{-\frac{1}{2}}\big\|\langle u_-\rangle^{1+\sigma}\langle u_+\rangle^{\frac{1}{2}(1+\sigma)}\mathbf{L}_{\mathbf{13}}^{(\beta_h,0)}(\tau,x)\big\|_{L^2_tL^2_x}\\
		&\lesssim \Big(\!\!\!\sum_{k\leqslant |\alpha_h|+1}\!\!\!\delta^{-\frac{1}{2}}\big(E_-^{(k,0)}(z_-)\big)^{\frac{1}{2}}+\!\!\!\sum_{k\leqslant |\alpha_h|}\!\!\!\delta^{-\frac{1}{2}}\big(E_-^{(k,0)}(\partial_3z_-)\big)^{\frac{1}{2}}\Big)\!\!
		\cdot\!\!\Big(\!\sum_{k\leqslant 5}\!\delta^{-\frac{1}{2}}\big(F_+^{(k,0)}(z_+)\big)^{\frac{1}{2}}+\!\sum_{k\leqslant 4}\!\delta^{-\frac{1}{2}}\big(F_+^{(k,0)}(\partial_3z_+)\big)^{\frac{1}{2}}\Big)\\
		&\lesssim C_1\varepsilon^2,\stepcounter{equation}\tag{\theequation}\label{eqH13}
	\end{align*}
	and
	\begin{align*}
		&\ \ \ \ 	\delta^{-\frac{1}{2}}\big\|\langle u_-\rangle^{1+\sigma}\langle u_+\rangle^{\frac{1}{2}(1+\sigma)}\mathbf{L}_{\mathbf{14}}^{(\beta_h,0)}(\tau,x)\big\|_{L^2_tL^2_x}\\
		&\lesssim \Big(\!\!\!\sum_{k\leqslant |\alpha_h|+1}\!\!\!\delta^{-\frac{1}{2}}\big(E_-^{(k,0)}(z_-)\big)^{\frac{1}{2}}+\!\!\!\sum_{k\leqslant |\alpha_h|}\delta^{-\frac{1}{2}}\big(E_-^{(k,0)}(\partial_3z_-)\big)^{\frac{1}{2}}\Big)\!\!\cdot\!\!\Big(\sum_{k\leqslant 4}\delta^{-\frac{1}{2}}\big(F_+^{(k,0)}(z_+)\big)^{\frac{1}{2}}+\sum_{k\leqslant 3}\delta^{-\frac{1}{2}}\big(F_+^{(k,0)}(\partial_3z_+)\big)^{\frac{1}{2}}\Big)\\
		&\lesssim C_1\varepsilon^2.\stepcounter{equation}\tag{\theequation}\label{eqH14}
	\end{align*}
	These estimates \eqref{eqH1-decomposition}, \eqref{eqH13} and \eqref{eqH14}  
	subsequently lead us to 
	\begin{equation}\label{eqH1case2estimate}
		\delta^{-\frac{1}{2}}\big\|\langle u_-\rangle^{1+\sigma}\langle u_+\rangle^{\frac{1}{2}(1+\sigma)}	\mathbf{L}_{\mathbf{1}}^{(\beta_h,0)}(\tau,x)\big\|_{L^2_tL^2_x}\lesssim C_1\varepsilon^2.
	\end{equation}
	
	Thus, for all $\beta_h\leqslant\alpha_h$, by noticing $N\geqslant 5$, we conclude from \textbf{Case 1} \eqref{eqH1case1estimate} and \textbf{Case 2} \eqref{eqH1case2estimate} that 
	\begin{equation}\label{eqH1estimate}
		\delta^{-\frac{1}{2}}\big\|\langle u_-\rangle^{1+\sigma}\langle u_+\rangle^{\frac{1}{2}(1+\sigma)}\mathbf{L}_{\mathbf{1}}^{(\beta_h,0)}(\tau,x)\big\|_{L^2_tL^2_x} 
		\lesssim C_1\varepsilon^2.
	\end{equation}
	Summing up \eqref{eqH1estimate} for all $\beta_h\leqslant\alpha_h$, we can summarize that 
	\begin{equation}\label{eqH1estimate-total}
		\delta^{-\frac{1}{2}}\big\|\langle u_-\rangle^{1+\sigma}\langle u_+\rangle^{\frac{1}{2}(1+\sigma)}\mathbf{L}_{\mathbf{1}}^{(\alpha_h,0)}(\tau,x)\big\|_{L^2_tL^2_x} 
		\lesssim C_1\varepsilon^2.
	\end{equation}
	
	\subsubsection*{\bf Estimate of $\mathbf{L}_{\mathbf{2}}^{(\alpha_h,0)}$}
	Based on the number of derivatives, we will distinguish two cases: 
	\[1\leqslant|\alpha_h|\leqslant N+2\ \ \text{ and }\ \ |\alpha_h|=0.\]
	
	\subsubsection*{\bf Case 1:  $1\leqslant |\alpha_h|\leqslant N+2$}

	We can use integration by parts to split $\mathbf{L}_{\mathbf{2}}^{(\alpha_h,0)}$ as follows:  for all $\gamma_h\leqslant\alpha_h$ with $|\gamma_h|=1$, there holds 
	\begin{align*}
		\mathbf{L}_{\mathbf{2}}^{(\alpha_h,0)}(\tau,x)
		&=-\int_{\Omega_\delta}\partial_i\partial_j G_\delta(x,y)\big(1-\theta(|x_h-y_h|)\big)\partial_{h}^{\gamma_h}\partial_{h}^{\alpha_h-\gamma_h} \nabla(z_+^jz_-^i)(\tau,y)dy\\
		&= \underbrace{-\int_{\Omega_\delta}\partial_{h}^{\gamma_h}\partial_i\partial_j G_\delta(x,y)\big(1-\theta(|x_h-y_h|)\big)\partial_{h}^{\alpha_h-\gamma_h}\nabla(z_+^jz_-^i)(\tau,y)dy}_{\mathbf{L}_{\mathbf{21}}^{(\alpha_h-\gamma_h,0)}(\tau,x)}\\
		&\ \ \ \ \underbrace{-\int_{\Omega_\delta}\partial_i\partial_j G_\delta(x,y)\partial_{h}^{\gamma_h}\theta(|x_h-y_h|)\partial_{h}^{\alpha_h-\gamma_h}\nabla(z_+^jz_-^i)(\tau,y)dy}_{\mathbf{L}_{\mathbf{22}}^{(\alpha_h-\gamma_h,0)}(\tau,x)}.\stepcounter{equation}\tag{\theequation}\label{eqH2-decomposition}
	\end{align*}
	On one hand, it is evident that the integral for $y$ in  $\mathbf{L}_{\mathbf{22}}^{(\alpha_h-\gamma_h,0)}$ exists when $y\in\{|x_h-y_h|\leqslant 2\}\times(-\delta,\delta)$. We can apply the argument for $\mathbf{L}_{\mathbf{1}}^{(\alpha_h,0)}$ as \eqref{eqH1-decomposition-total}-\eqref{eqH1estimate-total} 		 again, 
	with $\theta(|x_h-y_h|)$ replaced by $\partial_{y_h}^{\gamma_h}\theta(|x_h-y_h|)$ and $\alpha_h$ by $\alpha_h-\gamma_h$, to derive 
	\begin{equation}\label{eqH22}
		\delta^{-\frac{1}{2}}\big\|\langle u_-\rangle^{1+\sigma}\langle u_+\rangle^{\frac{1}{2}(1+\sigma)}\mathbf{L}_{\mathbf{22}}^{(\alpha_h-\gamma_h,0)}(\tau,x)\big\|_{L^2_tL^2_x} \lesssim C_1\varepsilon^2.
	\end{equation}
	On the other hand,  to evaluate  $\mathbf{L}_{\mathbf{21}}^{(\alpha_h-\gamma_h,0)}$,	we can  further decompose
	it as 
	\begin{equation*}
		\mathbf{L}_{\mathbf{21}}^{(\alpha_h-\gamma_h,0)}(\tau,x)
		=\sum_{\beta_h\leqslant\alpha_h-\gamma_h}C_{\alpha_h-\gamma_h,\beta_h}\underbrace{\int_{\Omega_\delta}\partial_h^{\gamma_h}\partial_i\partial_jG_\delta(x,y)\big(1-\theta(|x_h-y_h|)\big)\nabla\big(\partial_{h}^{\alpha_h-\gamma_h-\beta_h}z_+^j\partial_{h}^{\beta_h}z_-^i\big)(\tau,y)dy}_{\mathbf{L}_{\mathbf{21}}^{(\beta_h,0)}(\tau,x)}.
	\end{equation*}
	Every term $\mathbf{L}_{\mathbf{21}}^{(\beta_h,0)}$ therein can be handled in much the same way as $\mathbf{L}_{\mathbf{11}}^{(\beta_h,0)}$ and  $\mathbf{L}_{\mathbf{12}}^{(\beta_h,0)}$: 
	\begin{align*}
		&\ \ \ \ \delta^{-\frac{1}{2}}\big\|\langle u_-\rangle^{1+\sigma}\langle u_+\rangle^{\frac{1}{2}(1+\sigma)}\mathbf{L}_{\mathbf{21}}^{(\beta_h,0)}(\tau,x)\big\|_{L^2_tL^2_x}\\
		&\stackrel{\text{Lemma \ref{lemma:boundgreen}}}{\lesssim}\delta^{-\frac{1}{2}}\Big\|\langle u_-\rangle^{1+\sigma}\langle u_+\rangle^{\frac{1}{2}(1+\sigma)}\int_{-\delta}^{\delta}\int_{|x_h-y_h|\geqslant 1}\frac{1}{\delta}\frac{1}{|x_h-y_h|^3} \big|\nabla\big(\partial_{h}^{\alpha_h-\gamma_h-\beta_h}z_+^j\partial_{h}^{\beta_h}z_-^i\big)(\tau,y)\big|dy_hdy_3\Big\|_{L^2_tL^2_{x_h}L^2_{x_3}}\\
		&\stackrel{\eqref{eq:weight2}}{\lesssim}\delta^{-\frac{1}{2}}\Big\|\int_{-\delta}^{\delta}\int_{|x_h-y_h|\geqslant 1}\!\frac{1}{\delta}\frac{1}{|x_h-y_h|^{3-\frac{3}{2}(1+\sigma)}} \Big(\!\langle u_-\rangle^{1+\sigma}\langle u_+\rangle^{\frac{1}{2}(1+\sigma)}\big|\nabla\big(\partial_{h}^{\alpha_h-\gamma_h-\beta_h}z_+^j\partial_{h}^{\beta_h}z_-^i\big)\big|\Big)(\tau,y)dy_hdy_3\Big\|_{L^2_tL^2_{x_h}\!\! L^2_{x_3}}\\
		&\stackrel{\text{Minkowski}}{\lesssim}\!\delta^{-1}\!\bigg(\!\int_0^t\!\!\int_{-\delta}^{\delta}\Big\|\!\int_{|x_h-y_h|\geqslant 1}\!\!\frac{1}{|x_h-y_h|^{3-\frac{3}{2}(1+\sigma)}} \Big(\!\langle u_-\rangle^{1+\sigma}\langle u_+\rangle^{\frac{1}{2}(1+\sigma)}\big|\nabla\big(\partial_{h}^{\alpha_h-\gamma_h-\beta_h}\!z_+^j\partial_{h}^{\beta_h}\!z_-^i\big)\big|\Big)(\tau,y)dy_h\Big\|^2_{L^2_{x_h}}\! \! dy_3d\tau\!\!\bigg)^{\frac{1}{2}}\\
		&\stackrel{\text{Young}}{\lesssim}\delta^{-1}\bigg(\!\!\int_0^t\!\!\int_{-\delta}^{\delta}\!\Big\|\frac{1}{|x_h|^{3-\frac{3}{2}(1+\sigma)}}\chi_{|x_h|\geqslant 1}\Big\|^2_{L^2_{x_h}}\!\!\! \Big\|\Big(\langle u_-\rangle^{1+\sigma}\langle u_+\rangle^{\frac{1}{2}(1+\sigma)}\big|\nabla\big(\partial_{h}^{\alpha_h-\gamma_h-\beta_h}z_+^j\partial_{h}^{\beta_h}z_-^i\big)\big|\Big)(\tau,x_h,y_3)\Big\|^2_{L^1_{x_h}}\!\! dy_3d\tau\!\!\bigg)^{\frac{1}{2}}\\
		&\lesssim \delta^{-1} \Big\|\Big(\langle u_-\rangle^{1+\sigma}\langle u_+\rangle^{\frac{1}{2}(1+\sigma)}\big|\nabla\big(\partial_{h}^{\alpha_h-\gamma_h-\beta_h}z_+^j\partial_{h}^{\beta_h}z_-^i\big)\big|\Big)(\tau,x_h,y_3)\Big\|_{L^2_t L^1_{x_h}L^1_{y_3}}\\
		&\stackrel{\text{H\"older}}{\lesssim}\delta^{-1}\Big(\big\|\langle u_+\rangle^{1+\sigma}\partial_{h}^{\beta_h} \nabla z_-^i\big\|_{L^\infty_tL^2_{x_h}L^2_{x_3}}\cdot\Big\|\frac{\langle u_-\rangle^{1+\sigma}}{\langle u_+\rangle^{\frac{1}{2}(1+\sigma)}}\partial_{h}^{\alpha_h-\gamma_h-\beta_h}z_+^j\Big\|_{L^2_tL^2_{x_h}L^2_{x_3}}\\
		&\ \ \ \ +\big\|\langle u_+\rangle^{1+\sigma}\partial_{h}^{\beta_h} z_-^i\big\|_{L^\infty_tL^2_{x_h}L^2_{x_3}}\cdot\Big\|\frac{\langle u_-\rangle^{1+\sigma}}{\langle u_+\rangle^{\frac{1}{2}(1+\sigma)}}\partial_{h}^{\alpha_h-\gamma_h-\beta_h}\nabla z_+^j\Big\|_{L^2_tL^2_{x_h}L^2_{x_3}}\Big)\\
		&\lesssim \Big(\!\!\!\sum_{k\leqslant|\alpha_h|}\!\!\!\delta^{-\frac{1}{2}}\big(E_-^{(k,0)}(z_-)\big)^{\frac{1}{2}} +\!\!\!\sum_{k\leqslant|\alpha_h|-1}\!\!\!\delta^{-\frac{1}{2}}\big(E_-^{(k,0)}(\partial_3z_-)\big)^{\frac{1}{2}}\Big) \cdot\Big(\!\!\sum_{k\leqslant|\alpha_h|}\!\!\delta^{-\frac{1}{2}}\big(F_+^{(k,0)}(z_+)\big)^{\frac{1}{2}}+\!\!\sum_{k\leqslant|\alpha_h|-1}\!\!\delta^{-\frac{1}{2}}\big(F_+^{(k,0)}(\partial_3z_+)\big)^{\frac{1}{2}}\Big)\\
		&\stackrel{\eqref{improve2}}{\lesssim} C_1\varepsilon^2.
	\end{align*}
	Consequently, we obtain 
	\begin{equation}\label{eqH21}
		\delta^{-\frac{1}{2}}\big\|\langle u_-\rangle^{1+\sigma}\langle u_+\rangle^{\frac{1}{2}(1+\sigma)}\mathbf{L}_{\mathbf{21}}^{(\alpha_h-\gamma_h,0)}(\tau,x)\big\|_{L^2_tL^2_x}\lesssim C_1\varepsilon^2.
	\end{equation}
	Therefore, combining \eqref{eqH2-decomposition}, \eqref{eqH22} and \eqref{eqH21} gives rise to
	\begin{equation}\label{eqH2estimate} \delta^{-\frac{1}{2}}\big\|\langle u_-\rangle^{1+\sigma}\langle u_+\rangle^{\frac{1}{2}(1+\sigma)}\mathbf{L}_{\mathbf{2}}^{(\alpha_h,0)}(\tau,x)\big\|_{L^2_tL^2_x}\lesssim C_1\varepsilon^2.
	\end{equation}
	
	\subsubsection*{\bf Case 2:  $|\alpha_h|=0$}
	
	Using integration by parts, we see that
	\begin{align*}
		\mathbf{L}_{\mathbf{2}}^{(\alpha_h,0)}(\tau,x)
		&=
		\underbrace{-\int_{\Omega_\delta}\nabla\partial_i\partial_j G_\delta(x,y)\big(1-\theta(|x_h-y_h|)\big)(z_+^jz_-^i)(\tau,y)dy}_{\displaystyle \mathbf{L}_{\mathbf{23}}(\tau,x)}\\
		&\ \ \ \ \underbrace{-\int_{\Omega_\delta}\!\!\partial_i\partial_j G_\delta(x,y)\nabla\theta(|x_h-y_h|)(z_+^jz_-^i)(\tau,y)dy}_{\displaystyle \mathbf{L}_{\mathbf{24}}(\tau,x)}.
	\end{align*}
	On one hand, 	by the same method used to derive 
	\eqref{eqH21}, we can obtain the following estimate of $\mathbf{L}_{\mathbf{23}}$:
	\begin{equation*}
		\delta^{-\frac{1}{2}}\big\|\langle u_-\rangle^{1+\sigma}\langle u_+\rangle^{\frac{1}{2}(1+\sigma)}\mathbf{L}_{\mathbf{23}}(\tau,x)\big\|_{L^2_tL^2_x}\lesssim C_1\varepsilon^2.
	\end{equation*}
	On the other hand, similar arguments of $\mathbf{L}_{\mathbf{1}}^{(\alpha_h,0)}$ can be applied here to $\mathbf{L}_{\mathbf{24}}$ and hence 
	\begin{equation*} 
		\delta^{-\frac{1}{2}}\big\|\langle u_-\rangle^{1+\sigma}\langle u_+\rangle^{\frac{1}{2}(1+\sigma)}\mathbf{L}_{\mathbf{24}}(\tau,x)\big\|_{L^2_tL^2_x}\lesssim C_1\varepsilon^2.
	\end{equation*}
	In this case, gathering 
	the above estimates together shows
	\begin{equation}\label{eqH2estimate-case0} 
		\delta^{-\frac{1}{2}}\big\|\langle u_-\rangle^{1+\sigma}\langle u_+\rangle^{\frac{1}{2}(1+\sigma)}\mathbf{L}_{\mathbf{2}}^{(\alpha_h,0)}(\tau,x)\big\|_{L^2_tL^2_x}\lesssim C_1\varepsilon^2.
	\end{equation}
	
	In consequence, for all $\alpha_h$ with $0\leqslant|\alpha_h|\leqslant N+2$, we derive from \textbf{Case 1} \eqref{eqH2estimate} and \textbf{Case 2} \eqref{eqH2estimate-case0} that 
	\begin{equation}\label{eqH2estimate-total} 
		\delta^{-\frac{1}{2}}\big\|\langle u_-\rangle^{1+\sigma}\langle u_+\rangle^{\frac{1}{2}(1+\sigma)}\mathbf{L}_{\mathbf{2}}^{(\alpha_h,0)}(\tau,x)\big\|_{L^2_tL^2_x}\lesssim C_1\varepsilon^2.
	\end{equation}

	Up to now, according to \eqref{eq:H1H2}, \eqref{eqH1estimate-total} and \eqref{eqH2estimate-total}, we obtain for all $0\leqslant|\alpha_h|\leqslant N+2$ that 
	\begin{equation*}
		\delta^{-\frac{1}{2}}\big\|\langle u_-\rangle^{1+\sigma}\langle u_+\rangle^{\frac{1}{2}(1+\sigma)}\partial_{h}^{\alpha_h} \nabla p(\tau,x)\big\|_{L^2_tL^2_x}
		\lesssim C_1\varepsilon^2.
	\end{equation*}
	This completes the proof of the lemma. 
\end{proof}

\begin{lemma}\label{lemma:HH}
For any $\alpha_h\in(\mathbb{Z}_{\geqslant 0})^2$ and any $l\in\mathbb{Z}_{\geqslant 1}$ with 
	$1\leqslant
	|\alpha_h|+l\leqslant N+2$, there holds 
	\begin{equation*}
		\delta^{l-\frac{1}{2}}\big\|\langle u_-\rangle^{1+\sigma}\langle u_+\rangle^{\frac{1}{2}(1+\sigma)}\partial_{h}^{\alpha_h}\partial_3^l \nabla p\big\|_{L^2_tL^2_x}
		\lesssim C_1\varepsilon^2.
	\end{equation*}
\end{lemma}
\begin{proof}
Applying the weighted div-curl lemma to  the vector field $\partial_{h}^{\alpha_h}\nabla p$ with the weight $\langle u_-\rangle^{1+\sigma}\langle u_+\rangle^{\frac{1}{2}(1+\sigma)}$, we have 
	\begin{align*}
		&\ \ \ \ \big\|\langle u_-\rangle^{1+\sigma}\langle u_+\rangle^{\frac{1}{2}(1+\sigma)}\partial_{h}^{\alpha_h}\partial_3^l \nabla p\big\|_{L^2_x}
			\lesssim \big\|\langle u_-\rangle^{1+\sigma}\langle u_+\rangle^{\frac{1}{2}(1+\sigma)}\nabla^l \partial_h^{\alpha_h}\nabla p\big\|_{L^2_x}\\
			&\stackrel{\text{Lemma \ref{lemma:divcurl}},\  \eqref{eq:weight3}}{\lesssim}
		\sum_{l_1=0}^{l-1} \big\|\langle u_-\rangle^{1+\sigma}\langle u_+\rangle^{\frac{1}{2}(1+\sigma)}\operatorname{div}\partial_h^{\alpha_h}\nabla^{l_1} \nabla p\big\|_{L^2_x}+\sum_{l_1=0}^{l-1} \big\|\langle u_-\rangle^{1+\sigma}\langle u_+\rangle^{\frac{1}{2}(1+\sigma)}\operatorname{curl}\partial_h^{\alpha_h}\nabla^{l_1} \nabla p\big\|_{L^2_x}\\
		&\ \ \ \  
		+\big\|\langle u_-\rangle^{1+\sigma}\langle u_+\rangle^{\frac{1}{2}(1+\sigma)}\partial_h^{\alpha_h}\nabla p\big\|_{L^2_x}\\
		&\ \ \ \  
		 +\sum_{l_1=0}^{l-1}\Big|\int_{\partial\Omega_{\delta}}\langle u_-\rangle^{1+\sigma}\langle u_+\rangle^{\frac{1}{2}(1+\sigma)}\Big(\nabla^{l_1}\partial_h^{\alpha_h}(\nabla p)^h\cdot \nabla^{l_1}\nabla_h\partial_h^{\alpha_h}(\nabla p)^3-\nabla^{l_1}\partial_h^{\alpha_h}(\nabla p)^3\cdot \nabla^{l_1}\nabla_h\partial_h^{\alpha_h}(\nabla p)^h\Big)dx_h\Big|.
	\end{align*}
For the terms on the right hand side, we notice that 
\begin{align*}
&\operatorname{div}\partial_h^{\alpha_h}\nabla^{l_1} \nabla p=\partial_h^{\alpha_h}\nabla^{l_1}\underbrace{\operatorname{div}\nabla p}_{=\Delta p} 
\stackrel{\eqref{eqpressure-system}_1}{=}-\partial_h^{\alpha_h}\nabla^{l_1}(\nabla z_+\cdot \nabla z_-),\\
&\operatorname{curl}\partial_h^{\alpha_h}\nabla^{l_1} \nabla p=\partial_h^{\alpha_h}\nabla^{l_1}\underbrace{\operatorname{curl}\nabla p}_{=0}=0,\\
&(\nabla p)^3\big|_{\partial\Omega_{\delta}}=(\nabla p)^3\big|_{x_3=\pm\delta}\stackrel{\eqref{MHD equation}_1}{=}-\big(\partial_tz_++(z_--B_0)\cdot \nabla z_+\big)^3\big|_{x_3=\pm\delta}\stackrel{\eqref{MHD equation}_4}{=}0.
\end{align*}
Hence we obtain
\begin{align*}
	&\ \ \ \ \big\|\langle u_-\rangle^{1+\sigma}\langle u_+\rangle^{\frac{1}{2}(1+\sigma)}\partial_{h}^{\alpha_h}\partial_3^l \nabla p\big\|_{L^2_x}\\
	&\lesssim\sum_{l_1=0}^{l-1} \big\|\langle u_-\rangle^{1+\sigma}\langle u_+\rangle^{\frac{1}{2}(1+\sigma)}\partial_h^{\alpha_h}\nabla^{l_1} (\nabla z_+\cdot \nabla z_-)\big\|_{L^2_x}+\big\|\langle u_-\rangle^{1+\sigma}\langle u_+\rangle^{\frac{1}{2}(1+\sigma)}\partial_h^{\alpha_h}\nabla p\big\|_{L^2_x}.
\end{align*}
It then follows that 
		\begin{align*}
			&\ \ \ \ \delta^{l-\frac{1}{2}}\big\|\langle u_-\rangle^{1+\sigma}\langle u_+\rangle^{\frac{1}{2}(1+\sigma)}\partial_{h}^{\alpha_h}\partial_3^l \nabla p\big\|_{L^2_x}\\
			&\lesssim\sum_{|\alpha_h'|+l_1\leqslant l-1}\delta^{l_1+\frac{1}{2}} \big\|\langle u_-\rangle^{1+\sigma}\langle u_+\rangle^{\frac{1}{2}(1+\sigma)}\partial_h^{\alpha_h}\partial_h^{\alpha_h'}\partial_3^{l_1} (\nabla z_+\cdot \nabla z_-)\big\|_{L^2_x}+\delta^{-\frac{1}{2}}\big\|\langle u_-\rangle^{1+\sigma}\langle u_+\rangle^{\frac{1}{2}(1+\sigma)}\partial_h^{\alpha_h}\nabla p\big\|_{L^2_x}\\
			&\stackrel{\text{Lemmas  \ref{lemma1} \&  \ref{lemma:H}}}{\lesssim}C_1\varepsilon^2.
	\end{align*}
We have thus proved this lemma. 
\end{proof}

It is trivial to see that  combining the preceding two lemmas yields the pressure estimates for all $\partial_{h}^{\alpha_h}\partial_3^l \nabla p$  with the coefficient $\delta^{l-\frac{1}{2}}$.
\begin{lemma}\label{lemma:HH'}
	For any $\alpha_h\in(\mathbb{Z}_{\geqslant 0})^2$ and any $l\in\mathbb{Z}_{\geqslant 0}$ with
	$0\leqslant
	|\alpha_h|+l\leqslant N+2$, there holds 
	\begin{equation*}
		\delta^{l-\frac{1}{2}}\big\|\langle u_-\rangle^{1+\sigma}\langle u_+\rangle^{\frac{1}{2}(1+\sigma)}\partial_{h}^{\alpha_h}\partial_3^l \nabla p\big\|_{L^2_tL^2_x}
		\lesssim C_1\varepsilon^2.
	\end{equation*}
\end{lemma}

In particular, using the notations in Section \ref{sec:notation}, we have a direct renormalization corollary of  Lemma \ref{lemma:HH'}. We point out that  this corollary will be used in the approximation part of Section \ref{sec:approximation}.
\begin{corollary}\label{remarkp}
	For any $\alpha_h\in(\mathbb{Z}_{\geqslant 0})^2$ and any $l\in\mathbb{Z}_{\geqslant 0}$ with 
$0\leqslant
|\alpha_h|+l\leqslant N+2$, there holds 
\begin{equation*}
	\big\|\langle u_-\rangle^{1+\sigma}\langle u_+\rangle^{\frac{1}{2}(1+\sigma)}\partial_{h}^{\alpha_h}\partial_3^l \nabla_\delta p_{(\delta)}\big\|_{L^2_tL^2_x}
	\lesssim C_1\varepsilon^2.
\end{equation*}
\end{corollary}

To end this section, we give the pressure estimates for  $\partial_{h}^{\alpha_h}\partial_3^l\partial_3 \nabla p$ with the lower order coefficient $\delta^{l-\frac{1}{2}}$.
\begin{lemma}\label{lemma:HHH}
For any $\alpha_h\in(\mathbb{Z}_{\geqslant 0})^2$ and any $l\in\mathbb{Z}_{\geqslant 0}$ with 
$0\leqslant
|\alpha_h|+l\leqslant N+2$, there holds 
	\begin{equation*}
		\delta^{l-\frac{1}{2}}\big\|\langle u_-\rangle^{1+\sigma}\langle u_+\rangle^{\frac{1}{2}(1+\sigma)}\partial_{h}^{\alpha_h}\partial_3^l\partial_3 \nabla p\big\|_{L^2_tL^2_x}
		\lesssim C_1\varepsilon^2.
	\end{equation*}
\end{lemma}

\begin{proof}
The almost same reasoning used in Lemma \ref{lemma:HH} applies to this lower order coefficient case. Hence we have
\begin{align*}
	&\ \ \ \ \big\|\langle u_-\rangle^{1+\sigma}\langle u_+\rangle^{\frac{1}{2}(1+\sigma)}\partial_{h}^{\alpha_h}\partial_3^l\partial_3 \nabla p\big\|_{L^2_x}\\
	&\lesssim \sum_{l_1=0}^{l} \big\|\langle u_-\rangle^{1+\sigma}\langle u_+\rangle^{\frac{1}{2}(1+\sigma)}\partial_h^{\alpha_h}\nabla^{l_1} (\nabla z_+\cdot \nabla z_-)\big\|_{L^2_x}+\big\|\langle u_-\rangle^{1+\sigma}\langle u_+\rangle^{\frac{1}{2}(1+\sigma)}\partial_h^{\alpha_h}\nabla p\big\|_{L^2_x},
\end{align*}
which together with Lemma \ref{lemma1} and Lemma \ref{lemma:H}  implies that 
	\begin{align*}
	&\ \ \ \ \delta^{l-\frac{1}{2}}\big\|\langle u_-\rangle^{1+\sigma}\langle u_+\rangle^{\frac{1}{2}(1+\sigma)}\partial_{h}^{\alpha_h}\partial_3^l\partial_3 \nabla p\big\|_{L^2_x}\\
	&\lesssim\sum_{|\alpha_h'|+l_1\leqslant l}\delta^{l_1+\frac{1}{2}} \big\|\langle u_-\rangle^{1+\sigma}\langle u_+\rangle^{\frac{1}{2}(1+\sigma)}\partial_h^{\alpha_h}\partial_h^{\alpha_h'}\partial_3^{l_1} (\nabla z_+\cdot \nabla z_-)\big\|_{L^2_x}+\delta^{-\frac{1}{2}}\big\|\langle u_-\rangle^{1+\sigma}\langle u_+\rangle^{\frac{1}{2}(1+\sigma)}\partial_h^{\alpha_h}\nabla p\big\|_{L^2_x}
	\lesssim C_1\varepsilon^2.
\end{align*}
This proves the lemma. 
\end{proof}

\section{Construction of scattering fields in $\Omega_\delta$}\label{sec:scattering}

In this section, we will study the scattering fields associated to the solution  $\big(z_+(t,x),z_-(t,x)\big)$  constructed in Theorem \ref{lemma:global} for the system \eqref{MHD equation}.

Towards this goal, we first define the  infinities. 
Given a point $(0, x_1,x_2,x_3) \in \{t=0\}\times\Omega_\delta$, it determines uniquely a left-traveling straight line $\ell_-$ and a right-traveling straight line $\ell_+$: 
\[\ell_-: \mathbb{R}\rightarrow \mathbb{R}\times \Omega_\delta, \ \ t\mapsto (t,x_1+t,x_2,x_3),\]
\[\ell_+: \mathbb{R}\rightarrow \mathbb{R}\times \Omega_\delta, \ \ t\mapsto (t,x_1-t,x_2,x_3).\]
It is clear that  $u_-\big|_{\ell_-}\equiv x_1$ and $u_+\big|_{\ell_+}\equiv x_1$. We also denote the line $\ell_-$ by $\ell_-(u_-,x_2,x_3)$  where $u_-$, $x_2$ and $x_3$ are constants, and denote the line $\ell_+$ by $\ell_+(u_+,x_2,x_3)$  where $u_+$, $x_2$ and $x_3$ are constants. 
We use $\mathcal{C}_+$ to denote the collection of all the left-traveling lines and $\mathcal{C}_-$ to denote the collection of all the right-traveling lines:
\[\mathcal{C}_+=\big\{\ell_-(u_-,x_2,x_3)\big|\,t=\infty,\, (u_-,x_2,x_3)\in \Omega_\delta\big\},\] 
\[\mathcal{C}_-=\big\{\ell_+(u_+,x_2,x_3)\big|\,t=\infty,\, (u_+,x_2,x_3)\in \Omega_\delta\big\}.\] 
We call $\mathcal{C}_+$ \emph{the left infinity} and  $\mathcal{C}_-$ \emph{the right infinity} respectively, and in fact they are the spaces where the scattering fields live. 
These descriptions can be easily seen from the following Figure \ref{fig:scattering}.

	\begin{figure}[ht]
			\vspace{-0.1cm}
	\centering
	\includegraphics[width=4in]{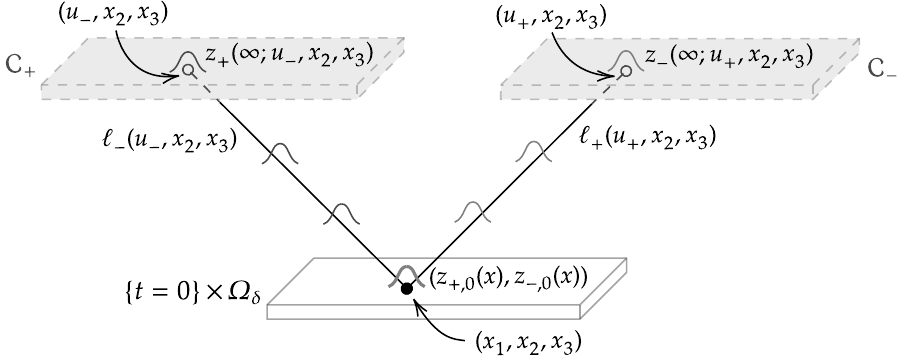}
	\vspace{-0.1cm}
	\caption{Infinities and scattering fields}
	\label{fig:scattering}
		\vspace{-0.1cm}
\end{figure}

We notice that the first two equations in \eqref{MHD equation} can be written as 
\begin{equation*}
	\begin{cases}
&\displaystyle
	\frac{d}{dt}\big(z_{+}(t,u_--t,x_2,x_3)\big) =-\nabla p(t,u_--t,x_2,x_3)-\big(z_{-}\cdot\nabla z_{+}\big)(t,u_--t,x_2,x_3),\\
&\displaystyle
	\frac{d}{dt}\big(z_{-}(t,u_++t,x_2,x_3)\big) =-\nabla p(t,u_++t,x_2,x_3)-\big(z_{+}\cdot\nabla z_{-}\big)(t,u_++t,x_2,x_3).
	\end{cases}
\end{equation*}
Integrating these two equations along  $\ell_-(u_+,x_2,x_3)$ and  $\ell_+(u_-,x_2,x_3)$ respectively
from time $0$ to time $t$ leads us to
\begin{equation*}
	\begin{cases}
&\displaystyle	z_+(t,u_--t,x_2,x_3)=z_+(0,u_-,x_2,x_3)-\int_0^t\big(\nabla p+z_{-}\cdot\nabla z_{+}\big)(\tau,u_--\tau,x_2,x_3)d\tau,\\
&\displaystyle	z_-(t,u_++t,x_2,x_3)=z_-(0,u_+,x_2,x_3)-\int_0^t\big(\nabla p+z_{+}\cdot\nabla z_{-}\big)(\tau,u_++\tau,x_2,x_3)d\tau.	\end{cases}
\end{equation*}
When $t\to \infty$, it is natural to expect that the following two formulas can  define the scattering fields $\delta^{-\frac{1}{2}}z_+(\infty;u_-,x_2,x_3)$ (on $\mathcal{C}_+$) and  $\delta^{-\frac{1}{2}}z_-(\infty;u_+,x_2,x_3)$ (on $\mathcal{C}_-$) respectively:
\begin{equation}\label{eq:def-sca}
\begin{cases}
	&\displaystyle
	\delta^{-\frac{1}{2}}z_{+}(\infty;u_-,x_2,x_3):= \delta^{-\frac{1}{2}}z_{+}(0,u_-,x_2,x_3)- \delta^{-\frac{1}{2}}\int_0^{\infty}(\nabla p+z_{-}\cdot\nabla z_{+})(\tau,u_--\tau,x_2,x_3)d\tau,\\
&\displaystyle 	\delta^{-\frac{1}{2}}z_{-}(\infty;u_+,x_2,x_3):= \delta^{-\frac{1}{2}}z_{-}(0,u_+,x_2,x_3)- \delta^{-\frac{1}{2}}\int_0^{\infty}(\nabla p+z_{+}\cdot\nabla z_{-})(\tau,u_++\tau,x_2,x_3)d\tau.
\end{cases}	
\end{equation}

In what follows, we turn to prove this expectation  valid and study the properties of scattering fields. The first main theorem of this paper is as follows:
\begin{theorem}[Scattering fields in $\Omega_{\delta}$]\label{thm1}  
All the integrals in  \eqref{eq:def-sca} converge. Therefore the vector fields $\delta^{-\frac{1}{2}}z_{+}(\infty;u_-,x_2,x_3)$ and  $\delta^{-\frac{1}{2}}z_{-}(\infty;u_+,x_2,x_3)$ are well-defined by \eqref{eq:def-sca}, and we call them the left scattering field and the right scattering field respectively. Moreover, for any $\alpha_h\in(\mathbb{Z}_{\geqslant 0})^2$ and $l\in\mathbb{Z}_{\geqslant 0}$ with 
$0\leqslant|\alpha_h|+l\leqslant N+2$,  there hold the following two properties of scattering fields:
\begin{enumerate}[(i)]
\item 
these scattering fields live in the following functional spaces in the weighted energy sense:
\begin{align*}
&\delta^{l-\frac{1}{2}}\partial_h^{\alpha_h}\partial_3^l z_{\pm}(\infty;u_\mp,x_2,x_3)\in L^2(\mathcal{C}_\pm,\langle u_\mp\rangle^{2(1+\sigma)}du_\mp dx_2dx_3),\\
&\delta^{-\frac{3}{2}}\partial_h^{\alpha_h} z^3_{\pm}(\infty;u_\mp,x_2,x_3)\in L^2(\mathcal{C}_\pm,\langle u_\mp\rangle^{2(1+\sigma)}du_\mp dx_2dx_3),\\
&\delta^{l-\frac{1}{2}}\partial_h^{\alpha_h}\partial_3^l(\partial_3 z_{\pm})(\infty;u_\mp,x_2,x_3)\in L^2(\mathcal{C}_\pm,\langle u_\mp\rangle^{2(1+\sigma)}du_\mp dx_2dx_3). 
\end{align*}
\item these scattering fields can be approximated by the large time solution in the weighted energy sense:
\begin{align*} 
&	\lim_{T\to\infty}\Big\|\delta^{l-\frac{1}{2}}\partial_h^{\alpha_h}\partial_3^lz_{\pm}(\infty;u_\mp,x_2,x_3)-\delta^{l-\frac{1}{2}}\partial_h^{\alpha_h}\partial_3^lz_{\pm}(T,u_\mp\mp T,x_2,x_3)\Big\|_{L^2(\mathcal{C}_\pm,\langle u_\mp\rangle^{2(1+\sigma)}du_\mp dx_2dx_3)}=0,\\
&	\lim_{T\to\infty}\Big\|\delta^{-\frac{3}{2}}\partial_h^{\alpha_h}z^3_{\pm}(\infty;u_\mp,x_2,x_3)-\delta^{-\frac{3}{2}}\partial_h^{\alpha_h}z^3_{\pm}(T,u_\mp\mp T,x_2,x_3)\Big\|_{L^2(\mathcal{C}_\pm,\langle u_\mp\rangle^{2(1+\sigma)}du_\mp dx_2dx_3)}=0,\\
&	\lim_{T\to\infty}\Big\|\delta^{l-\frac{1}{2}}\partial_h^{\alpha_h}\partial_3^l(\partial_3z_{\pm})(\infty;u_\mp,x_2,x_3)-\delta^{l-\frac{1}{2}}\partial_h^{\alpha_h}\partial_3^l(\partial_3z_{\pm})(T,u_\mp\mp T,x_2,x_3)\Big\|_{L^2(\mathcal{C}_\pm,\langle u_\mp\rangle^{2(1+\sigma)}du_\mp dx_2dx_3)}=0.
\end{align*}
\end{enumerate}
\end{theorem}

We remark here that 
based on Corollary \ref{coro1} and the notations before, Theorem \ref{thm1} also leads to an immediate consequence via renormalization:  
\begin{corollary}[Scattering fields in $\Omega_1$]\label{coro2}
Let  $\big(z_{+(\delta)}(t,x),z_{-(\delta)}(t,x)\big)$ be 	the solution  constructed in Corollary  \ref{coro1} for the rescaled system \eqref{eq:rescale}.
The vector fields $z_{+(\delta)}(\infty;u_-,x_2,x_3)$ (on the corresponding infinity  $\mathcal{C}_+=\big\{\ell_-(u_-,x_2,x_3)\big|\,t=\infty,\, (u_-,x_2,x_3)\in \Omega_1\big\}$) and  $z_{-(\delta)}(\infty;u_+,x_2,x_3)$ (on the corresponding infinity  $\mathcal{C}_-=\big\{\ell_+(u_+,x_2,x_3)\big|\,t=\infty,\, (u_+,x_2,x_3)\in \Omega_1\big\}$) are well-defined by the following \eqref{eq:def-sca2}:
\begin{equation}\label{eq:def-sca2}
	\begin{cases}
		&\displaystyle
		z_{+(\delta)}(\infty;u_-,x_2,x_3)= z_{+(\delta)}(0,u_-,x_2,x_3)- \int_0^{\infty}(\nabla_\delta p_{(\delta)}+z_{-(\delta)}\cdot\nabla z_{+(\delta)})(\tau,u_--\tau,x_2,x_3)d\tau,\\
		&\displaystyle 	z_{-(\delta)}(\infty;u_+,x_2,x_3)= z_{-(\delta)}(0,u_+,x_2,x_3)- \int_0^{\infty}(\nabla_\delta p_{(\delta)}+z_{+(\delta)}\cdot\nabla z_{-(\delta)})(\tau,u_++\tau,x_2,x_3)d\tau.
	\end{cases}	
\end{equation}
and we also call them the left scattering field and the right scattering field respectively. Moreover, for any $\alpha_h\in(\mathbb{Z}_{\geqslant 0})^2$ and $l\in\mathbb{Z}_{\geqslant 0}$ with 
$0\leqslant|\alpha_h|+l\leqslant N+2$,  there hold the following two properties of scattering fields:
\begin{enumerate}[(i)]
	\item 
	these scattering fields live in the following functional spaces in the weighted energy sense:
	\begin{align*}
		&\partial_h^{\alpha_h}\partial_3^l z_{\pm(\delta)}(\infty;u_\mp,x_2,x_3)\in L^2(\mathcal{C}_\pm,\langle u_\mp\rangle^{2(1+\sigma)}du_\mp dx_2dx_3),\\
		&\partial_h^{\alpha_h} z^3_{\pm(\delta)}(\infty;u_\mp,x_2,x_3)\in L^2(\mathcal{C}_\pm,\langle u_\mp\rangle^{2(1+\sigma)}du_\mp dx_2dx_3),\\
		&\partial_h^{\alpha_h}\partial_3^l(\partial_3 z_{\pm(\delta)})(\infty;u_\mp,x_2,x_3)\in L^2(\mathcal{C}_\pm,\langle u_\mp\rangle^{2(1+\sigma)}du_\mp dx_2dx_3).
	\end{align*}
	\item these scattering fields can be approximated by the large time solution in the weighted energy sense:
	\begin{align*} 
		&	\lim_{T\to\infty}\Big\|\partial_h^{\alpha_h}\partial_3^lz_{\pm(\delta)}(\infty;u_\mp,x_2,x_3)-\partial_h^{\alpha_h}\partial_3^lz_{\pm(\delta)}(T,u_\mp\mp T,x_2,x_3)\Big\|_{L^2(\mathcal{C}_\pm,\langle u_\mp\rangle^{2(1+\sigma)}du_\mp dx_2dx_3)}=0,\\
		&	\lim_{T\to\infty}\Big\|\partial_h^{\alpha_h}z^3_{\pm(\delta)}(\infty;u_\mp,x_2,x_3)- \partial_h^{\alpha_h}z^3_{\pm(\delta)}(T,u_\mp\mp T,x_2,x_3)\Big\|_{L^2(\mathcal{C}_\pm,\langle u_\mp\rangle^{2(1+\sigma)}du_\mp dx_2dx_3)}=0,\\
		&	\lim_{T\to\infty}\Big\| \partial_h^{\alpha_h}\partial_3^l(\partial_3z_{\pm(\delta)})(\infty;u_\mp,x_2,x_3)-\partial_h^{\alpha_h}\partial_3^l(\partial_3z_{\pm(\delta)})(T,u_\mp\mp T,x_2,x_3)\Big\|_{L^2(\mathcal{C}_\pm,\langle u_\mp\rangle^{2(1+\sigma)}du_\mp dx_2dx_3)}=0.
	\end{align*}
\end{enumerate}	
\end{corollary}

\begin{remark}\label{remark2D}
	When $\delta=0$, the domain $\Omega_{\delta}$ naturally becomes $\mathbb{R}^2$ and 
	we can rewrite \eqref{eq:def-sca2} as 
	\begin{equation}\label{eq:def-sca3}
		\begin{cases}
			&\displaystyle
			z_{+(0)}(\infty;u_-,x_2)= z_{+(0)}(0,u_-,x_2)- \int_0^{\infty}(\nabla p_{(0)}+z_{-(0)}\cdot\nabla z_{+(0)})(\tau,u_--\tau,x_2)d\tau,\\
			&\displaystyle 	z_{-(0)}(\infty;u_+,x_2)= z_{-(0)}(0,u_+,x_2)- \int_0^{\infty}(\nabla p_{(0)}+z_{+(0)}\cdot\nabla z_{-(0)})(\tau,u_++\tau,x_2)d\tau.
		\end{cases}	
	\end{equation}
	It is worth noting that in $\mathbb{R}^2$, $z_{+(0)}^h(\infty;u_-,x_2)$ and $z_{-(0)}^h(\infty;u_+,x_2)$, 
	i.e. the 2D version of the scattering fields $z_{+(0)}(\infty;u_-,x_2)$ and $z_{-(0)}(\infty;u_+,x_2)$  constructed in \eqref{eq:def-sca3}, 
	coincide with the scattering fields $z_{+}(\infty;u_-,x_2)$ and $z_{-}(\infty;u_+,x_2)$ in \eqref{eq:def-sca4} respectively. 
\end{remark}

Our task of this section reduces to showing  Theorem \ref{thm1}. 
The proof of Theorem \ref{thm1} naturally 
consists of 
the following six lemmas.
 
\begin{lemma}\label{welldefined}
The scattering fields $\delta^{-\frac{1}{2}}z_{\pm}(\infty;u_\mp,x_2,x_3)$ in \eqref{eq:def-sca} are well-defined.
\end{lemma}
\begin{proof}
By symmetry, it suffices to show that $\delta^{-\frac{1}{2}}z_{+}(\infty;u_-,x_2,x_3)$ is well-defined. 

Using the fact $\langle u_-\rangle\geqslant 1$ and the standard Sobolev inequality in $\mathbb{R}^4$, we have 
\begin{align*}
&\ \ \ \ \delta^{-\frac{1}{2}}\langle u_-\rangle^{\frac{1}{2}(1+\sigma)}\langle u_+\rangle^{\frac{1}{2}(1+\sigma)} |\nabla p|\\
&\lesssim\delta^{-\frac{1}{2}}\big\|\langle u_-\rangle^{1+\sigma}\langle u_+\rangle^{\frac{1}{2}(1+\sigma)} \nabla p\big\|_{L^\infty_{t,x}}\\
&\lesssim\delta^{-\frac{1}{2}}\big\|\langle u_-\rangle^{1+\sigma}\langle u_+\rangle^{\frac{1}{2}(1+\sigma)} \nabla p\big\|_{L^2_{t,x}}+\delta^{-\frac{1}{2}}\big\|\nabla^3(\langle u_-\rangle^{1+\sigma}\langle u_+\rangle^{\frac{1}{2}(1+\sigma)} \nabla p)\big\|_{L^2_{t,x}}\\
&\stackrel{\eqref{eq:weight3}}{\lesssim}\delta^{-\frac{1}{2}}\big\|\langle u_-\rangle^{1+\sigma}\langle u_+\rangle^{\frac{1}{2}(1+\sigma)} \nabla p\big\|_{L^2_{t,x}}+\delta^{-\frac{1}{2}}\big\|\langle u_-\rangle^{1+\sigma}\langle u_+\rangle^{\frac{1}{2}(1+\sigma)} \nabla^3 \nabla p\big\|_{L^2_{t,x}}\\
&\lesssim
\sum_{0\leqslant|\alpha_h|\leqslant 3\atop 0\leqslant l\leqslant 3}\delta^{l-\frac{1}{2}}\big\|\langle u_-\rangle^{1+\sigma}\langle u_+\rangle^{\frac{1}{2}(1+\sigma)}\partial_{h}^{\alpha_h}\partial_3^l \nabla p\big\|_{L^2_tL^2_x}
\stackrel{\text{Lemma \ref{lemma:HH'}}}{\lesssim}C_1\varepsilon^2,\\
	&\ \ \ \ \delta^{-\frac{1}{2}}\langle u_-\rangle^{\frac{1}{2}(1+\sigma)}\langle u_+\rangle^{\frac{1}{2}(1+\sigma)} |z_-\cdot\nabla z_+|\\
	&\lesssim \delta^{-\frac{1}{2}}\big\|\langle u_-\rangle^{1+\sigma}\langle u_+\rangle^{\frac{1}{2}(1+\sigma)} (z_-\cdot\nabla z_+)\big\|_{L^\infty_{t,x}}\\
	&\lesssim\delta^{-\frac{1}{2}}\big\|\langle u_-\rangle^{1+\sigma}\langle u_+\rangle^{\frac{1}{2}(1+\sigma)} (z_-\cdot\nabla z_+)\big\|_{L^2_{t,x}}+\delta^{-\frac{1}{2}}\big\|\nabla^3(\langle u_-\rangle^{1+\sigma}\langle u_+\rangle^{\frac{1}{2}(1+\sigma)} (z_-\cdot\nabla z_+))\big\|_{L^2_{t,x}}\\
	&\stackrel{\eqref{eq:weight3}}{\lesssim}\delta^{-\frac{1}{2}}\big\|\langle u_-\rangle^{1+\sigma}\langle u_+\rangle^{\frac{1}{2}(1+\sigma)} (z_-\cdot\nabla z_+)\big\|_{L^2_{t,x}}+\delta^{-\frac{1}{2}}\big\|\langle u_-\rangle^{1+\sigma}\langle u_+\rangle^{\frac{1}{2}(1+\sigma)} \nabla^3 (z_-\cdot\nabla z_+)\big\|_{L^2_{t,x}}\\
	&\lesssim
	\sum_{0\leqslant|\alpha_h|\leqslant 3\atop 0\leqslant l\leqslant 3}\delta^{l-\frac{1}{2}}\big\|\langle u_-\rangle^{1+\sigma}\langle u_+\rangle^{\frac{1}{2}(1+\sigma)}\partial_{h}^{\alpha_h}\partial_3^l (z_-\cdot\nabla z_+)\big\|_{L^2_tL^2_x}
	\stackrel{\text{Lemma \ref{lemma2}}}{\lesssim}C_1\varepsilon^2.
\end{align*}
Combined with \eqref{eq:product}, these two estimates lead us to 
\begin{equation*}
\delta^{-\frac{1}{2}}|\nabla p+z_-\cdot \nabla z_+|\lesssim\frac{C_1\varepsilon^2}{\langle u_-\rangle^{\frac{1}{2}(1+\sigma)}\langle u_+\rangle^{\frac{1}{2}(1+\sigma)}}\lesssim\frac{C_1\varepsilon^2}{(1+|t+a|)^{1+\sigma}}\in L_t^1(\mathbb{R}).
\end{equation*}
Therefore the integral in \eqref{eq:def-sca} converges and hence $\delta^{-\frac{1}{2}}z_{+}(\infty;u_-,x_2,x_3)$ is well-defined. 
\end{proof}
 
\begin{lemma}\label{lemma4}
There hold
	\begin{equation}\label{eq1}
		 \delta^{-\frac{1}{2}}z_{\pm}(\infty;u_\mp,x_2,x_3)\in L^2(\mathcal{C}_\pm,\langle u_\mp\rangle^{2(1+\sigma)}du_\mp dx_2dx_3),
	\end{equation}
and 
\begin{equation}\label{eq2}
\lim_{T\to\infty}\Big\|\delta^{-\frac{1}{2}}z_{\pm}(\infty;u_\mp,x_2,x_3)-\delta^{-\frac{1}{2}}z_{\pm}(T,u_\mp\mp T,x_2,x_3)\Big\|_{L^2(\mathcal{C}_\pm,\langle u_\mp\rangle^{2(1+\sigma)}du_\mp dx_2dx_3)}=0.
\end{equation}
\end{lemma}
\begin{proof}
	Based on the symmetry, it suffices to derive the estimates on $z_{+}(\infty;u_-,x_2,x_3)$.
	
Firstly, by \eqref{improve2}, we note that 
\begin{align*}
	&\ \ \ \ \big\|\delta^{-\frac{1}{2}}z_+(0,u_-,x_2,x_3)\big\|_{L^2(\mathcal{C}_+,\langle u_-\rangle^{2(1+\sigma)}du_-dx_2dx_3)}\\
	&	= \delta^{-\frac{1}{2}}\big\|z_+(0,u_-,x_2,x_3)\big\|_{L^2(\Omega_\delta,\langle u_-\rangle^{2(1+\sigma)}du_-dx_2dx_3)}
	=\delta^{-\frac{1}{2}}E_+^{(0,0)}(z_{+,0})
	\lesssim  \varepsilon,
\end{align*}
which shows 
\begin{equation}\label{eq:lemma1-14}  \delta^{-\frac{1}{2}}z_{+}(0,u_-,x_2,x_3)\in L^2(\mathcal{C}_+,\langle u_-\rangle^{2(1+\sigma)}du_-dx_2dx_3).
\end{equation}
We also note that for any large $T>0$, \eqref{improve2} gives
\begin{equation}\label{eq:def-large}
	\delta^{-\frac{1}{2}}z_{+}(\infty;u_-,x_2,x_3)-\delta^{-\frac{1}{2}}z_{+}(T,u_--T,x_2,x_3)=-\delta^{-\frac{1}{2}}\int_T^{\infty}(\nabla p+z_{-}\cdot\nabla z_{+})(\tau,u_--\tau,x_2,x_3)d\tau.
\end{equation}

Secondly, using coordinate transformations (which allow us to perform the analysis on spacetimes) and H\"older inequality, we obtain
\begin{align*}
	&\ \ \ \ \Big\|\delta^{-\frac{1}{2}}\int_0^{\infty}  (\nabla p+z_{-}\cdot\nabla z_{+}) 
	(\tau,u_--\tau,x_2,x_3)d\tau\Big\|_{L^2(\mathcal{C}_+,\langle u_-\rangle^{2(1+\sigma)}du_-dx_2dx_3)}\\
	&
	\lesssim\delta^{-\frac{1}{2}}\Big\|\Big(\int_{\mathbb{R}}\frac{1}{\langle u_+\rangle^{1+\sigma}}du_+\Big)^{\frac{1}{2}}\Big(\int_{\mathbb{R}}\langle u_+\rangle^{1+\sigma} |(\nabla p+z_{-}\cdot\nabla z_{+}) 
	(u_+,u_-,x_2,x_3)|^2du_+\Big)^{\frac{1}{2}}\Big\|_{L^2(\mathcal{C}_+,\langle u_-\rangle^{2(1+\sigma)}du_-dx_2dx_3)}\\
	&\lesssim\delta^{-\frac{1}{2}}\big\|\langle u_-\rangle^{1+\sigma}\langle u_+\rangle^{\frac{1}{2}(1+\sigma)} (\nabla p+z_{-}\cdot\nabla z_{+}) 
	(u_+,u_-,x_2,x_3)\big\|_{L^2(\mathbb{R}\times\Omega_\delta,du_+du_-dx_2dx_3)}\\
	&\lesssim\delta^{-\frac{1}{2}}\big\|\langle u_-\rangle^{1+\sigma}\langle u_+\rangle^{\frac{1}{2}(1+\sigma)} \nabla p\big\|_{L^2_tL^2_x}+\delta^{-\frac{1}{2}}\big\|\langle u_-\rangle^{1+\sigma}\langle u_+\rangle^{\frac{1}{2}(1+\sigma)}(z_-\cdot \nabla z_+)\big\|_{L^2_tL^2_x}
	\stackrel{\text{Lemmas \ref{lemma:H} \& \ref{lemma2}}}{\lesssim}C_1\varepsilon^2.
\end{align*}	
This implies 
\begin{equation}\label{eq:lemma1}
	\delta^{-\frac{1}{2}}\int_0^{\infty} (\nabla p+z_{-}\cdot\nabla z_{+}) 
	(\tau,u_--\tau,x_2,x_3)d\tau\in L^2(\mathcal{C}_+,\langle u_-\rangle^{2(1+\sigma)}du_-dx_2dx_3),
\end{equation}
and
\begin{equation}\label{eq:lemma1'}
\lim_{T\to \infty}\Big\|\delta^{-\frac{1}{2}}\int_T^{\infty}  (\nabla p+z_{-}\cdot\nabla z_{+}) 
(\tau,u_--\tau,x_2,x_3)d\tau\Big\|_{L^2(\mathcal{C}_+,\langle u_-\rangle^{2(1+\sigma)}du_-dx_2dx_3)}=0.
\end{equation}

Finally, putting \eqref{eq:def-sca}, \eqref{eq:lemma1-14} and \eqref{eq:lemma1} together yields \eqref{eq1}, while \eqref{eq:def-large} and \eqref{eq:lemma1'} give rise to  \eqref{eq2}.
\end{proof}

We are now in a position to derive the uniform estimates concerning $\partial_h^{\alpha_h}z_\pm(\infty;u_\mp,x_2,x_3)$ with the coefficient $\delta^{-\frac{1}{2}}$.
 
\begin{lemma}\label{lemma5}
For any $\alpha_h\in(\mathbb{Z}_{\geqslant 0})^2$ with 
$1\leqslant|\alpha_h|\leqslant N+2$, there hold
\begin{equation*} 
\delta^{-\frac{1}{2}}\partial_h^{\alpha_h} z_{\pm}(\infty;u_\mp,x_2,x_3)\in L^2(\mathcal{C}_\pm,\langle u_\mp\rangle^{2(1+\sigma)}du_\mp dx_2dx_3),
\end{equation*}
and 
\begin{equation*} 
	\lim_{T\to\infty}\Big\|\delta^{-\frac{1}{2}}\partial_h^{\alpha_h}z_{\pm}(\infty;u_\mp,x_2,x_3)-\delta^{-\frac{1}{2}}\partial_h^{\alpha_h}z_{\pm}(T,u_\mp\mp T,x_2,x_3)\Big\|_{L^2(\mathcal{C}_\pm,\langle u_\mp\rangle^{2(1+\sigma)}du_\mp dx_2dx_3)}=0.
\end{equation*}
\end{lemma}
\begin{proof}
By the symmetry considerations, we only give details for the estimates on $z_{+}(\infty;u_-,x_2,x_3)$.
		
Applying the derivative $	\partial_h^{\alpha_h}$ to 	\eqref{eq:def-sca}-\eqref{eq:def-large}, we have 
	\begin{equation*}
		\delta^{-\frac{1}{2}}\partial_h^{\alpha_h} z_{+}(\infty;u_-,x_2,x_3)=\delta^{-\frac{1}{2}}\partial_h^{\alpha_h} z_{+}(0,u_-,x_2,x_3)-\delta^{-\frac{1}{2}}\partial_h^{\alpha_h}\int_0^{\infty} (\nabla p+z_{-}\cdot\nabla z_{+}) 
			(\tau,u_--\tau,x_2,x_3)d\tau,
	\end{equation*}
and
\begin{equation}\label{eq:lemma1-11'}
	\delta^{-\frac{1}{2}}\partial_h^{\alpha_h} z_{+}(\infty;u_-,x_2,x_3)-\delta^{-\frac{1}{2}}\partial_h^{\alpha_h} z_{+}(T,u_--T,x_2,x_3)=-\delta^{-\frac{1}{2}}\partial_h^{\alpha_h}\int_T^{\infty} (\nabla p+z_{-}\cdot\nabla z_{+}) 
	(\tau,u_--\tau,x_2,x_3)d\tau.
\end{equation}
It is clear from \eqref{improve2} that 
\begin{align*}
&\ \ \ \ 	\big\|\delta^{-\frac{1}{2}}\partial_h^{\alpha_h}z_+(0,u_-,x_2,x_3)\big\|_{L^2(\mathcal{C}_+,\langle u_-\rangle^{2(1+\sigma)}du_-dx_2dx_3)}\\
&	= \delta^{-\frac{1}{2}}\big\|\partial_h^{\alpha_h}z_+(0,u_-,x_2,x_3)\big\|_{L^2(\Omega_\delta,\langle u_-\rangle^{2(1+\sigma)}du_-dx_2dx_3)}
=\delta^{-\frac{1}{2}}E_+^{(|\alpha_h|,0)}(z_{+,0})
\lesssim  	\varepsilon,
\end{align*}
which means 
\begin{equation*}
\delta^{-\frac{1}{2}}\partial_h^{\alpha_h} z_{+}(0,u_-,x_2,x_3)\in L^2(\mathcal{C}_+,\langle u_-\rangle^{2(1+\sigma)}du_-dx_2dx_3).
\end{equation*}
Therefore it suffices to show that 
\begin{equation}\label{eq:lemma1-5}
\delta^{-\frac{1}{2}}\partial_h^{\alpha_h}\int_0^{\infty} (\nabla p+z_{-}\cdot\nabla z_{+}) 
(\tau,u_--\tau,x_2,x_3)d\tau\in L^2(\mathcal{C}_+,\langle u_-\rangle^{2(1+\sigma)}du_-dx_2dx_3),
\end{equation}
and 
\begin{equation}\label{eq:lemma1-5'}
	\lim_{T\to \infty}\Big\|\delta^{-\frac{1}{2}}\partial_h^{\alpha_h}\int_T^{\infty}  (\nabla p+z_{-}\cdot\nabla z_{+}) 
	(\tau,u_--\tau,x_2,x_3)d\tau\Big\|_{L^2(\mathcal{C}_+,\langle u_-\rangle^{2(1+\sigma)}du_-dx_2dx_3)}=0.
\end{equation}

The rest of this proof is divided into four steps. 

\subsubsection*{\bf Step 1:} We first prove that 
\begin{equation}\label{eq:lemma1-6}
\delta^{-\frac{1}{2}}\int_0^{\infty} \partial_h^{\alpha_h}(\nabla p+z_{-}\cdot\nabla z_{+}) 
(\tau,u_--\tau,x_2,x_3)d\tau\in L^2(\mathcal{C}_+,\langle u_-\rangle^{2(1+\sigma)}du_-dx_2dx_3).
\end{equation}

In fact, 
this can also be proved via coordinate transformations and H\"older inequality:
\begin{align*}
&\ \ \ \ \Big\|\delta^{-\frac{1}{2}}\int_0^{\infty} \partial_h^{\alpha_h}(\nabla p+z_{-}\cdot\nabla z_{+}) 
(\tau,u_--\tau,x_2,x_3)d\tau\Big\|_{L^2(\mathcal{C}_+,\langle u_-\rangle^{2(1+\sigma)}du_-dx_2dx_3)}\\
&
\lesssim\delta^{-\frac{1}{2}}\Big\|\Big(\!\!\int_{\mathbb{R}}\frac{1}{\langle u_+\rangle^{1+\sigma}}du_+\!\Big)^{\!\frac{1}{2}}\Big(\!\!\int_{\mathbb{R}}\langle u_+\rangle^{1+\sigma} |\partial_h^{\alpha_h}(\nabla p+z_{-}\cdot\nabla z_{+}) 
(u_+,u_-,x_2,x_3)|^2du_+\!\Big)^{\!\frac{1}{2}}\Big\|_{L^2(\mathcal{C}_+,\langle u_-\rangle^{2(1+\sigma)}du_-dx_2dx_3)}\\
&\lesssim\delta^{-\frac{1}{2}}\big\|\langle u_-\rangle^{1+\sigma}\langle u_+\rangle^{\frac{1}{2}(1+\sigma)} \partial_h^{\alpha_h}(\nabla p+z_{-}\cdot\nabla z_{+}) 
(u_+,u_-,x_2,x_3)\big\|_{L^2(\mathbb{R}\times\Omega_\delta,du_+du_-dx_2dx_3)}\\
&\lesssim\delta^{-\frac{1}{2}}\big\|\langle u_-\rangle^{1+\sigma}\langle u_+\rangle^{\frac{1}{2}(1+\sigma)}\partial_{h}^{\alpha_h} \nabla p\big\|_{L^2_tL^2_x}+\delta^{-\frac{1}{2}}\big\|\langle u_-\rangle^{1+\sigma}\langle u_+\rangle^{\frac{1}{2}(1+\sigma)}\mathbf{J}_+^{(\alpha_h,0)}\big\|_{L^2_tL^2_x}\!\!\!\!\!\!\!
\stackrel{\text{Lemmas \ref{lemma:H} \&  \ref{lemma2}}}{\lesssim}\!\!\!\!\! C_1\varepsilon^2.\stepcounter{equation}\tag{\theequation}\label{eq:lemma1-7}
\end{align*}

\subsubsection*{\bf Step 2:} We next prove that for any $\gamma_h\leqslant\alpha_h$ with $|\gamma_h|=1$, 
there holds 
\begin{equation}\label{eq:lemma1-8}
\delta^{-\frac{1}{2}}\partial_h^{\gamma_h}\int_0^{\infty} \partial_h^{\alpha_h-\gamma_h}(\nabla p+z_{-}\cdot\nabla z_{+}) 
(\tau,u_--\tau,x_2,x_3)d\tau\in L^2(\mathcal{C}_+,\langle u_-\rangle^{2(1+\sigma)}du_-dx_2dx_3).
\end{equation}

In this step, 
	we only consider the case where 
	the outermost derivative of the above term 
	is taken as  $\partial_h^{\gamma_h}=\partial_{1}$, and the $\partial_{2}$ case can be treated in the same way.  
	Now it suffices to show that 
	\begin{equation}\label{eq:lemma1-1}
	\delta^{-\frac{1}{2}}\partial_{1}\int_0^{\infty} \partial_h^{\alpha_h-\gamma_h}(\nabla p+z_{-}\cdot\nabla z_{+}) 
	(\tau,u_--\tau,x_2,x_3)d\tau\in L^2(\mathcal{C}_+,\langle u_-\rangle^{2(1+\sigma)}du_-dx_2dx_3).
	\end{equation}

In fact, by definition, we  deduce that 
	\begin{align*}
	&\ \ \ \ \Big\|\delta^{-\frac{1}{2}}\partial_{1} \int_0^{\infty}\partial_h^{\alpha_h-\gamma_h}(\nabla p+z_{-}\cdot\nabla z_{+}) (\tau,u_--\tau,x_2,x_3)d\tau \Big\|_{L^2(\mathcal{C}_+,\langle u_{-}\rangle^{2(1+\sigma)}du_-dx_2dx_3)}\\
	&=\delta^{-\frac{1}{2}}\Big\|\lim_{h\to 0} \int_0^{\infty}\frac{1}{h}\big(\partial_h^{\alpha_h-\gamma_h}(\nabla p+z_{-}\cdot\nabla z_{+}) (\tau,u_-+h-\tau,x_2,x_3)\\
	&\ \ \ \ \ \ \ \ \ \ \ \ \ \ \ \ \ \ \ \ \ \ \ \  -\partial_h^{\alpha_h-\gamma_h}(\nabla p+z_{-}\cdot\nabla z_{+}) (\tau,u_--\tau,x_2,x_3)\big)
	d\tau \Big\|_{L^2(\mathcal{C}_+,\langle u_{-}\rangle^{2(1+\sigma)}du_-dx_2dx_3)}\\
	&\stackrel{\text{Fatou}}{\leqslant}\delta^{-\frac{1}{2}}\liminf_{h\to 0}\Big\| \int_0^{\infty}\frac{1}{h}\big(\partial_h^{\alpha_h-\gamma_h}(\nabla p+z_{-}\cdot\nabla z_{+}) (\tau,u_-+h-\tau,x_2,x_3)\\
	&\ \ \ \ \ \ \ \ \ \ \ \ \ \ \ \ \ \ \ \ \ \ \ \ \ \ \ \ \ \ \  -\partial_h^{\alpha_h-\gamma_h}(\nabla p+z_{-}\cdot\nabla z_{+}) (\tau,u_--\tau,x_2,x_3)\big)
	d\tau \Big\|_{L^2(\mathcal{C}_+,\langle u_{-}\rangle^{2(1+\sigma)}du_-dx_2dx_3)}\\
	&\stackrel{\text{Newton-Leibniz}}{\leqslant}\!\!\!\!\!\delta^{-\frac{1}{2}}\liminf_{h\to 0}\Big\|\!\! \int_0^{\infty}\!\!\!\int_0^1\!\!\partial_{1}\partial_h^{\alpha_h-\gamma_h}(\nabla p+z_{-}\cdot\nabla z_{+}) (\tau,u_-+\theta h-\tau,x_2,x_3)d\theta d\tau \Big\|_{L^2(\mathcal{C}_+,\langle u_{-}\rangle^{2(1+\sigma)}du_-dx_2dx_3)}\\
	&\leqslant\delta^{-\frac{1}{2}}\liminf_{h\to 0}\Big\|\int_0^1 \int_0^{\infty}\partial_{1}\partial_h^{\alpha_h-\gamma_h}(\nabla p+z_{-}\cdot\nabla z_{+}) (\tau,u_-+\theta h-\tau,x_2,x_3)d\tau d\theta  \Big\|_{L^2(\mathcal{C}_+,\langle u_{-}\rangle^{2(1+\sigma)}du_-dx_2dx_3)}\\
	&\leqslant\delta^{-\frac{1}{2}}\liminf_{h\to 0}\int_0^1\Big\| \int_0^{\infty}\partial_{1}\partial_h^{\alpha_h-\gamma_h}(\nabla p+z_{-}\cdot\nabla z_{+}) (\tau,u_-+\theta h-\tau,x_2,x_3)d\tau  \Big\|_{L^2(\mathcal{C}_+,\langle u_{-}\rangle^{2(1+\sigma)}du_-dx_2dx_3)}d\theta \\
	&\stackrel{\text{set }U_-=u_-+\theta h}{\leqslant}\!\!\delta^{-\frac{1}{2}}\liminf_{h\to 0}\int_0^1\Big\|\! \int_0^{\infty}\!\!\partial_{1}\partial_h^{\alpha_h-\gamma_h}(\nabla p+z_{-}\cdot\nabla z_{+}) (\tau,U_--\tau,x_2,x_3)d\tau  \Big\|_{L^2(\mathcal{C}_+,\langle u_{-}\rangle^{2(1+\sigma)}dU_-dx_2dx_3)}d\theta\\
	&\leqslant \Big\|\delta^{-\frac{1}{2}}\int_0^{\infty} \partial_h^{\alpha_h}(\nabla p+z_{-}\cdot\nabla z_{+}) 
	(\tau,u_--\tau,x_2,x_3)d\tau\Big\|_{L^2(\mathcal{C}_+,\langle u_-\rangle^{2(1+\sigma)}du_-dx_2dx_3)}
	\stackrel{\eqref{eq:lemma1-7}}{\lesssim} C_1\varepsilon^2.
	\end{align*}
 Hence \eqref{eq:lemma1-1} holds and therefore \eqref{eq:lemma1-8} follows.

\subsubsection*{\bf Step 3:} We turn to prove that for any $\gamma_h\leqslant\alpha_h$ with $|\gamma_h|=1$, as vector fields in $L^2(\mathcal{C}_+,\langle u_{-}\rangle^{2(1+\sigma)} du_-dx_2dx_3)$, 
there holds 
	\begin{align*}
	& 
	\delta^{-\frac{1}{2}}\partial_h^{\gamma_h}\int_0^{\infty} \partial_h^{\alpha_h-\gamma_h}(\nabla p+z_{-}\cdot\nabla z_{+}) 
	(\tau,u_--\tau,x_2,x_3)d\tau\\
	&\stackrel{ L^2(\mathcal{C}_+,\langle u_-\rangle^{2(1+\sigma)}du_-dx_2dx_3)}{=\joinrel=\joinrel=\joinrel=\joinrel=\joinrel=\joinrel=\joinrel=\joinrel=\joinrel=\joinrel=\joinrel=\joinrel=\joinrel=\joinrel=\joinrel=\joinrel=}\delta^{-\frac{1}{2}}\int_0^{\infty} \partial_h^{\alpha_h}(\nabla p+z_{-}\cdot\nabla z_{+}) 
	(\tau,u_--\tau,x_2,x_3)d\tau.\stepcounter{equation}\tag{\theequation}\label{eq:lemma1-9}
\end{align*}

In view of \eqref{eq:lemma1-6} and \eqref{eq:lemma1-8}, it suffices to show the following equation in the sense of distributions:
\begin{equation}\label{eq:lemma1-2}
	\delta^{-\frac{1}{2}}\partial_h^{\gamma_h}\int_0^{\infty} \partial_h^{\alpha_h-\gamma_h}(\nabla p+z_{-}\cdot\nabla z_{+}) 
	(\tau,u_--\tau,x_2,x_3)d\tau
	\stackrel{ \mathcal{D}'(\mathcal{C}_+)}{=\joinrel=\joinrel=\joinrel=}\delta^{-\frac{1}{2}}\int_0^{\infty}  \partial_h^{\alpha_h}(\nabla p+z_{-}\cdot\nabla z_{+}) 
	(\tau,u_--\tau,x_2,x_3)d\tau. 
\end{equation}

Based on \eqref{eq:lemma1} and \eqref{eq:lemma1-6},  
both the two time integrals in \eqref{eq:lemma1-2} are locally integrable functions. 
Let us take an arbitrary vector field $\varphi\in \mathcal{D}(\mathcal{C}_+)$. It is clear that $\varphi\in \mathcal{D}(\mathcal{C}_+)\subset L^2(\mathcal{C}_+)$, which also gives  $\partial_{h}^{\gamma_h}\varphi\in \mathcal{D}(\mathcal{C}_+)\subset L^2(\mathcal{C}_+)$. 
These facts enable us to infer that the following two spacetime integrals are finite:
\begin{align*}
	&\ \ \ \
	\delta^{-\frac{1}{2}}\int_{[0,\infty)\times \mathcal{C}_+}\big| \partial_h^{\alpha_h-\gamma_h}(\nabla p+z_{-}\cdot\nabla z_{+}) 
	(\tau,u_--\tau,x_2,x_3)\cdot\partial_{h}^{\gamma_h}\varphi(u_-,x_2,x_3)\big|d\tau du_-dx_2dx_3\\
	&\lesssim \delta^{-\frac{1}{2}}\int_{\mathbb{R}\times\mathcal{C}_+}\big|\partial_h^{\alpha_h-\gamma_h}(\nabla p+z_{-}\cdot\nabla z_{+}) 
	(u_+,u_-,x_2,x_3)\big|\cdot|\partial_{h}^{\gamma_h}\varphi(u_-,x_2,x_3)|du_+du_-dx_2dx_3\\
	&\lesssim
	\delta^{-\frac{1}{2}}\Big(\int_{\mathbb{R}\times \mathcal{C}_+}\langle  u_+\rangle^{1+\sigma}\big|\partial_h^{\alpha_h-\gamma_h}(\nabla p+z_{-}\cdot\nabla z_{+}) 
	(u_+,u_-,x_2,x_3)\big|^2 \Big)^{\frac{1}{2}}
	\Big(\int_{\mathbb{R}\times \mathcal{C}_+}\frac{|\partial_{h}^{\gamma_h}\varphi(u_-,x_2,x_3)|^2}{\langle u_+\rangle^{1+\sigma}}\Big)^{\frac{1}{2}}\\
	&\lesssim\delta^{-\frac{1}{2}}\Big(\int_{\mathbb{R}\times \mathcal{C}_+}\langle u_-\rangle^{2(1+\sigma)}\langle u_+\rangle^{1+\sigma}\big|\partial_h^{\alpha_h-\gamma_h}(\nabla p+z_{-}\cdot\nabla z_{+}) (u_+,u_-,x_2,x_3)\big|^2\Big)^{\frac{1}{2}}\\
	&\ \ \ \ \times \Big(\int_{\mathbb{R}}\frac{1}{\langle u_+\rangle^{1+\sigma}}\Big(\int_{\mathcal{C}_+} |\partial_{h}^{\gamma_h} \varphi(u_-,x_2,x_3)|^2  du_-dx_2dx_3 \Big)du_+\Big)^{\frac{1}{2}}\\
	&\lesssim\varepsilon^2\|\partial_{h}^{\gamma_h}\varphi\|_{L^2}<\infty,\stepcounter{equation}\tag{\theequation}\label{eq:lemma1-3}\\
	&\ \ \ \
	\delta^{-\frac{1}{2}}\int_{[0,\infty)\times \mathcal{C}_+}\big| \partial_h^{\alpha_h}(\nabla p+z_{-}\cdot\nabla z_{+}) 
	(\tau,u_--\tau,x_2,x_3)\cdot\varphi(u_-,x_2,x_3)\big|d\tau du_-dx_2dx_3\\
	&\lesssim \delta^{-\frac{1}{2}}\int_{\mathbb{R}\times\mathcal{C}_+}\big|\partial_h^{\alpha_h}(\nabla p+z_{-}\cdot\nabla z_{+}) 
	(u_+,u_-,x_2,x_3)\big|\cdot|\varphi(u_-,x_2,x_3)|du_+du_-dx_2dx_3\\
	&\lesssim
	\delta^{-\frac{1}{2}}\Big(\int_{\mathbb{R}\times \mathcal{C}_+}\langle  u_+\rangle^{1+\sigma}\big|\partial_h^{\alpha_h}(\nabla p+z_{-}\cdot\nabla z_{+}) 
	(u_+,u_-,x_2,x_3)\big|^2 \Big)^{\frac{1}{2}}
	\Big(\int_{\mathbb{R}\times \mathcal{C}_+}\frac{|\varphi(u_-,x_2,x_3)|^2}{\langle u_+\rangle^{1+\sigma}}\Big)^{\frac{1}{2}}\\
	&\lesssim\delta^{-\frac{1}{2}}\Big(\int_{\mathbb{R}\times \mathcal{C}_+}\langle u_-\rangle^{2(1+\sigma)}\langle u_+\rangle^{1+\sigma}\big|\partial_h^{\alpha_h}(\nabla p+z_{-}\cdot\nabla z_{+}) (u_+,u_-,x_2,x_3)\big|^2\Big)^{\frac{1}{2}}\\
	&\ \ \ \ \times \Big(\int_{\mathbb{R}}\frac{1}{\langle u_+\rangle^{1+\sigma}}\Big(\int_{\mathcal{C}_+} \varphi(u_-,x_2,x_3)|^2  du_-dx_2dx_3 \Big)du_+\Big)^{\frac{1}{2}}\\
	&\lesssim\varepsilon^2\|\varphi\|_{L^2}<\infty.\stepcounter{equation}\tag{\theequation}\label{eq:lemma1-4}
\end{align*}
We remark here that the estimates \eqref{eq:lemma1-3}-\eqref{eq:lemma1-4} will ensure the subsequent applications of Fubini's theorem to commute the order of integrals.

Using the above estimates, integration by parts and Fubini's theorem repeatedly, we then derive in the sense of distributions that 
	\begin{align*}
	&\ \ \ \  \Big\langle \delta^{-\frac{1}{2}}\partial_{h}^{\gamma_h}\int_0^{\infty}\partial_h^{\alpha_h-\gamma_h}(\nabla p+z_{-}\cdot\nabla z_{+}) 
		(\tau,u_--\tau,x_2,x_3)d\tau,\ \varphi(u_-,x_2,x_3)\Big\rangle\\
		&=-\delta^{-\frac{1}{2}}\Big\langle\int_0^{\infty}\partial_h^{\alpha_h-\gamma_h}(\nabla p+z_{-}\cdot\nabla z_{+}) 
		(\tau,u_--\tau,x_2,x_3)d\tau,\  \partial_{h}^{\gamma_h} \varphi(u_-,x_2,x_3)\Big\rangle\\
		&=-\delta^{-\frac{1}{2}}\int_{\mathcal{C}_+}\Big(\int_0^{\infty}\partial_h^{\alpha_h-\gamma_h}(\nabla p+z_{-}\cdot\nabla z_{+}) 
		(\tau,u_--\tau,x_2,x_3)d\tau\Big)\cdot \partial_{h}^{\gamma_h}\varphi(u_-,x_2,x_3)du_-dx_2dx_3\\
		&\stackrel{ 
		\eqref{eq:lemma1-3}}{=}-\delta^{-\frac{1}{2}}\int_{[0,\infty) \times \mathcal{C}_+}\partial_h^{\alpha_h-\gamma_h}(\nabla p+z_{-}\cdot\nabla z_{+}) (\tau,u_--\tau,x_2,x_3)\cdot \partial_{h}^{\gamma_h}\varphi(u_-,x_2,x_3)d\tau du_-dx_2dx_3\\
		&=-\delta^{-\frac{1}{2}}\int_{[0,\infty)} \Big(\int_{\mathcal{C}_+}\partial_h^{\alpha_h-\gamma_h}(\nabla p+z_{-}\cdot\nabla z_{+}) (\tau,u_--\tau,x_2,x_3)\cdot \partial_{h}^{\gamma_h}\varphi(u_-,x_2,x_3)du_-dx_2dx_3\Big)d\tau \\
		&=\delta^{-\frac{1}{2}}\int_{[0,\infty)} \Big(\int_{\mathcal{C}_+}\partial_h^{\alpha_h}(\nabla p+z_{-}\cdot\nabla z_{+}) (\tau,u_--\tau,x_2,x_3)\cdot \varphi(u_-,x_2,x_3)du_-dx_2dx_3\Big)d\tau \\
		&=\delta^{-\frac{1}{2}}\int_{[0,\infty)\times \mathcal{C}_+}\partial_h^{\alpha_h}(\nabla p+z_{-}\cdot\nabla z_{+}) (\tau,u_--\tau,x_2,x_3)\cdot \varphi(u_-,x_2,x_3)du_-dx_2dx_3d\tau\\
		&\stackrel{ 
		\eqref{eq:lemma1-4}}{=}\delta^{-\frac{1}{2}}\int_{\mathcal{C}_+}\Big(\int_{0}^{\infty} \partial_h^{\alpha_h}(\nabla p+z_{-}\cdot\nabla z_{+})(\tau,u_--\tau,x_2,x_3)d\tau\Big)\cdot \varphi(u_-,x_2,x_3)du_-dx_2dx_3\\
		&=\Big\langle\delta^{-\frac{1}{2}}\int_0^{\infty}  \partial_h^{\alpha_h}(\nabla p+z_{-}\cdot\nabla z_{+})(\tau,u_--\tau,x_2,x_3)d\tau,\ \varphi(u_-,x_2,x_3)\Big\rangle.
	\end{align*}
This implies \eqref{eq:lemma1-2} immediately. Thus we have proved \eqref{eq:lemma1-9}.

\subsubsection*{\bf Step 4:} We are now ready to show \eqref{eq:lemma1-5} and \eqref{eq:lemma1-5'}.

By induction on $\alpha_h$,  we obtain that the following equation holds in the sense of  weighted $L^2$
space $L^2(\mathcal{C}_+,\langle u_{-}\rangle^{2(1+\sigma)} du_-dx_2dx_3)$ as an immediate consequence of  \eqref{eq:lemma1-9}:
\begin{align*}
	& 
	\delta^{-\frac{1}{2}}\partial_h^{\alpha_h}\int_0^{\infty}(\nabla p+z_{-}\cdot\nabla z_{+}) 
	(\tau,u_--\tau,x_2,x_3)d\tau\\
	&\stackrel{ L^2(\mathcal{C}_+,\langle u_-\rangle^{2(1+\sigma)}du_-dx_2dx_3)}{=\joinrel=\joinrel=\joinrel=\joinrel=\joinrel=\joinrel=\joinrel=\joinrel=\joinrel=\joinrel=\joinrel=\joinrel=\joinrel=\joinrel=\joinrel=\joinrel=}\delta^{-\frac{1}{2}}\int_0^{\infty} \partial_h^{\alpha_h}(\nabla p+z_{-}\cdot\nabla z_{+}) 
	(\tau,u_--\tau,x_2,x_3)d\tau.\stepcounter{equation}\tag{\theequation}\label{eq:lemma1-10}
\end{align*}
This proves \eqref{eq:lemma1-5}. 

Moreover, by \eqref{eq:lemma1-10} and \eqref{eq:lemma1-7}, there holds 
\begin{align*}
&\ \ \ \ \Big\|\delta^{-\frac{1}{2}}\partial_h^{\alpha_h}\int_0^{\infty} (\nabla p+z_{-}\cdot\nabla z_{+}) 
(\tau,u_--\tau,x_2,x_3)d\tau\Big\|_{L^2(\mathcal{C}_+,\langle u_-\rangle^{2(1+\sigma)}du_-dx_2dx_3)}\\
&=\Big\|\delta^{-\frac{1}{2}}\int_0^{\infty} \partial_h^{\alpha_h}(\nabla p+z_{-}\cdot\nabla z_{+}) 
(\tau,u_--\tau,x_2,x_3)d\tau\Big\|_{L^2(\mathcal{C}_+,\langle u_-\rangle^{2(1+\sigma)}du_-dx_2dx_3)}
\lesssim C_1\varepsilon^2.\stepcounter{equation}\tag{\theequation}\label{eq:lemma1-12}
\end{align*}
As a direct consequence of   \eqref{eq:lemma1-10} and  \eqref{eq:lemma1-12}, we obtain 
\begin{align*}
&\ \ \ \ \lim_{T\to \infty}\Big\|\delta^{-\frac{1}{2}}\partial_h^{\alpha_h}\int_T^{\infty} (\nabla p+z_{-}\cdot\nabla z_{+}) 
(\tau,u_--\tau,x_2,x_3)d\tau\Big\|_{L^2(\mathcal{C}_+,\langle u_-\rangle^{2(1+\sigma)}du_-dx_2dx_3)}\\
&=\lim_{T\to \infty}\Big\|\delta^{-\frac{1}{2}}\int_T^{\infty}\partial_h^{\alpha_h} (\nabla p+z_{-}\cdot\nabla z_{+}) 
(\tau,u_--\tau,x_2,x_3)d\tau\Big\|_{L^2(\mathcal{C}_+,\langle u_-\rangle^{2(1+\sigma)}du_-dx_2dx_3)}=0.
\end{align*}
Together with \eqref{eq:lemma1-11'}, this gives rise to \eqref{eq:lemma1-5'}.

The proof of this lemma is now complete.
\end{proof}

Clearly, the above two lemmas also establish the following lemma as a direct consequence. We remark here that we do not need to consider other terms (such as  
 $\partial_h^{\alpha_h}\partial_3^lz^3_{\pm}$) in this lemma since they can be covered by the subsequent lemmas together with the fact that $\operatorname{div}z_\pm=0$ gives  $\partial_3z_\pm^h=-\partial_hz_\pm^h$.

\begin{lemma}\label{lemma5'}
	For any $\alpha_h\in(\mathbb{Z}_{\geqslant 0})^2$  with 
	$0\leqslant|\alpha_h|\leqslant N+2$, there hold
	\begin{equation*}
		\delta^{-\frac{3}{2}}\partial_h^{\alpha_h}z^3_{\pm}(\infty;u_\mp,x_2,x_3)\in L^2(\mathcal{C}_\pm,\langle u_\mp\rangle^{2(1+\sigma)}du_\mp dx_2dx_3),
	\end{equation*}
	and 
	\begin{equation*} 
		\lim_{T\to\infty}\Big\|\delta^{-\frac{3}{2}}\partial_h^{\alpha_h}z^3_{\pm}(\infty;u_\mp,x_2,x_3)-\delta^{-\frac{3}{2}}\partial_h^{\alpha_h}z^3_{\pm}(T,u_\mp\mp T,x_2,x_3)\Big\|_{L^2(\mathcal{C}_\pm,\langle u_\mp\rangle^{2(1+\sigma)}du_\mp dx_2dx_3)}=0.
	\end{equation*}
\end{lemma}

We turn to derive the uniform estimates concerning $\partial_h^{\alpha_h}\partial_3^lz_\pm(\infty;u_\mp,x_2,x_3)$ with the coefficient $\delta^{l-\frac{1}{2}}$.
 
\begin{lemma}\label{lemma6}
	For any $\alpha_h\in(\mathbb{Z}_{\geqslant 0})^2$ and $l\in\mathbb{Z}_{\geqslant 1}$ with 
	$1\leqslant|\alpha_h|+l\leqslant N+2$, there hold
	\begin{equation*}
\delta^{l-\frac{1}{2}}
	\partial_h^{\alpha_h}\partial_3^l z_{\pm}(\infty;u_\mp,x_2,x_3)\in L^2(\mathcal{C}_\pm,\langle u_\mp\rangle^{2(1+\sigma)}du_\mp dx_2dx_3),\end{equation*}
and 
\begin{equation*} 
	\lim_{T\to\infty}\Big\|\delta^{l-\frac{1}{2}}\partial_h^{\alpha_h}\partial_3^lz_{\pm}(\infty;u_\mp,x_2,x_3)-\delta^{l-\frac{1}{2}}\partial_h^{\alpha_h}\partial_3^lz_{\pm}(T,u_\mp\mp T,x_2,x_3)\Big\|_{L^2(\mathcal{C}_\pm,\langle u_\mp\rangle^{2(1+\sigma)}du_\mp dx_2dx_3)}=0.
\end{equation*}
\end{lemma}

\begin{proof}
		By the symmetry considerations, it suffices to give details for the estimate on $z_{+}(\infty;u_-,x_2,x_3)$.
		
Applying the derivative $	\partial_h^{\alpha_h}\partial_3^l$ to 	\eqref{eq:def-sca} and \eqref{eq:def-large} gives rise to 
	\begin{align*} 
	&\ \ \ \ 	\delta^{l-\frac{1}{2}}\partial_h^{\alpha_h}\partial_3^l z_{+}(\infty;u_-,x_2,x_3)\\
	&=\delta^{l-\frac{1}{2}}\partial_h^{\alpha_h}\partial_3^l z_{+}(0,u_-,x_2,x_3)-\delta^{l-\frac{1}{2}}\partial_h^{\alpha_h}\partial_3^l\int_0^{\infty} (\nabla p+z_{-}\cdot\nabla z_{+}) 
		(\tau,u_--\tau,x_2,x_3)d\tau,\stepcounter{equation}\tag{\theequation}\label{eq:lemma1-15}
	\end{align*}
and 
	\begin{align*} 
&\ \ \ \ 	\delta^{l-\frac{1}{2}}\partial_h^{\alpha_h}\partial_3^l z_{+}(\infty;u_-,x_2,x_3)-\delta^{l-\frac{1}{2}}\partial_h^{\alpha_h}\partial_3^l z_{+}(T,u_--T,x_2,x_3)
\\&
=-\delta^{l-\frac{1}{2}}\partial_h^{\alpha_h}\partial_3^l\int_T^{\infty} (\nabla p+z_{-}\cdot\nabla z_{+}) 
	(\tau,u_--\tau,x_2,x_3)d\tau.\stepcounter{equation}\tag{\theequation}\label{eq:lemma1-15'}
\end{align*}
According to \eqref{improve2}, we have
	\begin{align*}
		&\ \ \ \ 	\big\|\delta^{l-\frac{1}{2}}\partial_h^{\alpha_h}\partial_3^lz_+(0,u_-,x_2,x_3)\big\|_{L^2(\mathcal{C}_+,\langle u_-\rangle^{2(1+\sigma)}du_-dx_2dx_3)}\\
		&	=\delta^{l-\frac{1}{2}} \big\|\partial_h^{\alpha_h}\partial_3^lz_+(0,u_-,x_2,x_3)\big\|_{L^2(\Omega_\delta,\langle u_-\rangle^{2(1+\sigma)}du_-dx_2dx_3)}		=\delta^{l-\frac{1}{2}}E_+^{(|\alpha_h|,l)}(z_{+,0})
		\lesssim 
		\varepsilon,
	\end{align*}
and hence 
	\begin{equation}\label{eq:lemma1-16}
		\delta^{l-\frac{1}{2}}\partial_h^{\alpha_h}\partial_3^l z_{+}(0,u_-,x_2,x_3)\in L^2(\mathcal{C}_+,\langle u_-\rangle^{2(1+\sigma)}du_-dx_2dx_3).
	\end{equation}
By virtue of \eqref{eq:lemma1-15}, \eqref{eq:lemma1-15'} and \eqref{eq:lemma1-16}, our task is now reduced to showing that
	\begin{equation}\label{eq:lemma1-26}
	\delta^{l-\frac{1}{2}}	\partial_h^{\alpha_h}\partial_3^l\int_0^{\infty} (\nabla p+z_{-}\cdot\nabla z_{+}) 
		(\tau,u_--\tau,x_2,x_3)d\tau\in L^2(\mathcal{C}_+,\langle u_-\rangle^{2(1+\sigma)}du_-dx_2dx_3),
	\end{equation}
and 
\begin{equation}\label{eq:lemma1-26'}
	\lim_{T\to \infty}\Big\|\delta^{-\frac{1}{2}}\partial_h^{\alpha_h}\partial_3^l\int_T^{\infty}  (\nabla p+z_{-}\cdot\nabla z_{+}) 
	(\tau,u_--\tau,x_2,x_3)d\tau\Big\|_{L^2(\mathcal{C}_+,\langle u_-\rangle^{2(1+\sigma)}du_-dx_2dx_3)}=0.
\end{equation}
	
	The rest of this proof is divided into four steps. 
	
	\subsubsection*{\bf Step 1:} We first prove that 
	\begin{equation}\label{eq:lemma1-17}
		\delta^{l-\frac{1}{2}}\int_0^{\infty} \partial_h^{\alpha_h}\partial_3^l(\nabla p+z_{-}\cdot\nabla z_{+}) 
		(\tau,u_--\tau,x_2,x_3)d\tau\in L^2(\mathcal{C}_+,\langle u_-\rangle^{2(1+\sigma)}du_-dx_2dx_3).
	\end{equation}

	In fact, 
	this can be proved via coordinate transformations and H\"older inequality:
\begin{align*}
	&\ \ \ \ \Big\|\delta^{l-\frac{1}{2}}\int_0^{\infty} \partial_h^{\alpha_h}\partial_3^l(\nabla p+z_{-}\cdot\nabla z_{+}) 
	(\tau,u_--\tau,x_2,x_3)d\tau\Big\|_{L^2(\mathcal{C}_+,\langle u_-\rangle^{2(1+\sigma)}du_-dx_2dx_3)}\\
	&
	\lesssim\delta^{l-\frac{1}{2}}\Big\|\Big(\!\!\int_{\mathbb{R}}\!\frac{1}{\langle u_+\rangle^{1+\sigma}}du_+\!\Big)^{\!\!\frac{1}{2}}\!\Big(\!\!\int_{\mathbb{R}}\!\!\langle u_+\rangle^{1+\sigma} |\partial_h^{\alpha_h}\partial_3^l(\nabla p+z_{-}\cdot\nabla z_{+}) 
	(u_+,u_-,x_2,x_3)|^2du_+\!\Big)^{\!\!\frac{1}{2}}\Big\|_{L^2(\mathcal{C}_+,\langle u_-\rangle^{2(1+\sigma)}du_-dx_2dx_3)}\\
	&\lesssim\delta^{l-\frac{1}{2}}\big\|\langle u_-\rangle^{1+\sigma}\langle u_+\rangle^{\frac{1}{2}(1+\sigma)} \partial_h^{\alpha_h}\partial_3^l(\nabla p+z_{-}\cdot\nabla z_{+}) 
	(u_+,u_-,x_2,x_3)\big\|_{L^2(\mathbb{R}\times\Omega_\delta,du_+du_-dx_2dx_3)}\\
	&\lesssim\delta^{l-\frac{1}{2}}\big\|\langle u_-\rangle^{1+\sigma}\langle u_+\rangle^{\frac{1}{2}(1+\sigma)}\partial_{h}^{\alpha_h}\partial_3^l \nabla p\big\|_{L^2_tL^2_x}+\delta^{l-\frac{1}{2}}\big\|\langle u_-\rangle^{1+\sigma}\langle u_+\rangle^{\frac{1}{2}(1+\sigma)}{\mathbf{J}_+^{(\alpha_h,l)}}\big\|_{L^2_tL^2_x}\!\!\!\!\!\!\!
	\stackrel{\text{Lemmas \ref{lemma:HH} \& \ref{lemma2}}}{\lesssim}\!\!\!\!\! C_1\varepsilon^2.\stepcounter{equation}\tag{\theequation}\label{eq:lemma1-18}
\end{align*}

	\subsubsection*{\bf Step 2:} We show that  
	\begin{equation}\label{eq:lemma1-19}
	\delta^{l-\frac{1}{2}}\partial_h^{\gamma_h}\int_0^{\infty} \partial_h^{\alpha_h-\gamma_h}\partial_3^l(\nabla p+z_{-}\cdot\nabla z_{+}) (\tau,u_--\tau,x_2,x_3)d\tau\in L^2(\mathcal{C}_+,\langle u_-\rangle^{2(1+\sigma)}du_-dx_2dx_3)
	\end{equation}
for any $\gamma_h\leqslant\alpha_h$ with $|\gamma_h|=1$, and 
\begin{equation}\label{eq:lemma1-20}
	\delta^{l-\frac{1}{2}}\partial_3\int_0^{\infty} \partial_h^{\alpha_h}\partial_3^{l-1}(\nabla p+z_{-}\cdot\nabla z_{+}) 
	(\tau,u_--\tau,x_2,x_3)d\tau\in L^2(\mathcal{C}_+,\langle u_-\rangle^{2(1+\sigma)}du_-dx_2dx_3).
\end{equation}

 We note that \eqref{eq:lemma1-19} can be proved by similar methods used for  \eqref{eq:lemma1-8} and \eqref{eq:lemma1-20}. Thus it is sufficient to show \eqref{eq:lemma1-20} now. 
In fact,  
there holds  
\begin{align*}
	&\ \ \ \ \Big\|\delta^{l-\frac{1}{2}}\partial_3\int_0^{\infty} \partial_h^{\alpha_h}\partial_3^{l-1}(\nabla p+z_{-}\cdot\nabla z_{+}) (\tau,u_--\tau,x_2,x_3)d\tau \Big\|_{L^2(\mathcal{C}_+,\langle u_{-}\rangle^{2(1+\sigma)}du_-dx_2dx_3)}\\
	&=\delta^{l-\frac{1}{2}}\Big\|\lim_{h\to 0} \int_0^{\infty}\frac{1}{h}\big(\partial_h^{\alpha_h}\partial_3^{l-1}(\nabla p+z_{-}\cdot\nabla z_{+}) (\tau,u_--\tau,x_2,x_3+h)\\
	&\ \ \ \ \ \ \ \ \ \ \ \ \ \ \ \ \ \ \ \ \ \ \ \  -\partial_h^{\alpha_h}\partial_3^{l-1}(\nabla p+z_{-}\cdot\nabla z_{+}) (\tau,u_--\tau,x_2,x_3)\big)
	d\tau \Big\|_{L^2(\mathcal{C}_+,\langle u_{-}\rangle^{2(1+\sigma)}du_-dx_2dx_3)}\\
	&\stackrel{\text{Fatou}}{\leqslant}\delta^{l-\frac{1}{2}}\liminf_{h\to 0}\Big\| \int_0^{\infty}\frac{1}{h}\big(\partial_h^{\alpha_h}\partial_3^{l-1}(\nabla p+z_{-}\cdot\nabla z_{+}) (\tau,u_--\tau,x_2,x_3+h)\\
	&\ \ \ \ \ \ \ \ \ \ \ \ \ \ \ \ \ \ \ \ \ \ \ \ \ \ \ \ \ \ \  -\partial_h^{\alpha_h}\partial_3^{l-1}(\nabla p+z_{-}\cdot\nabla z_{+}) (\tau,u_--\tau,x_2,x_3)\big)
	d\tau \Big\|_{L^2(\mathcal{C}_+,\langle u_{-}\rangle^{2(1+\sigma)}du_-dx_2dx_3)}\\
	&\stackrel{\text{Newton-Leibniz}}{\leqslant}\!\!\!\!\!\delta^{l-\frac{1}{2}}\liminf_{h\to 0}\Big\| \int_0^{\infty}\!\!\!\int_0^1\partial_h^{\alpha_h}\partial_3^{l}(\nabla p+z_{-}\cdot\nabla z_{+}) (\tau,u_--\tau,x_2,x_3+\theta h)d\theta d\tau \Big\|_{L^2(\mathcal{C}_+,\langle u_{-}\rangle^{2(1+\sigma)}du_-dx_2dx_3)}\\
	&\leqslant\delta^{l-\frac{1}{2}}\liminf_{h\to 0}\Big\|\int_0^1 \int_0^{\infty}\partial_h^{\alpha_h}\partial_3^{l}(\nabla p+z_{-}\cdot\nabla z_{+}) (\tau,u_--\tau,x_2,x_3+\theta h)d\tau d\theta  \Big\|_{L^2(\mathcal{C}_+,\langle u_{-}\rangle^{2(1+\sigma)}du_-dx_2dx_3)}\\
	&\leqslant\delta^{l-\frac{1}{2}}\liminf_{h\to 0}\int_0^1\Big\| \int_0^{\infty}\partial_h^{\alpha_h}\partial_3^{l}(\nabla p+z_{-}\cdot\nabla z_{+}) (\tau,u_--\tau,x_2,x_3+\theta h)d\tau  \Big\|_{L^2(\mathcal{C}_+,\langle u_{-}\rangle^{2(1+\sigma)}du_-dx_2dx_3)}d\theta \\
	&\stackrel{\text{set }X_3=x_3+\theta h}{\leqslant}\delta^{l-\frac{1}{2}}\liminf_{h\to 0}\int_0^1\Big\| \int_0^{\infty}\partial_h^{\alpha_h}\partial_3^{l}(\nabla p+z_{-}\cdot\nabla z_{+}) (\tau,u_--\tau,x_2,X_3)d\tau  \Big\|_{L^2(\mathcal{C}_+,\langle u_{-}\rangle^{2(1+\sigma)}du_-dx_2dX_3)}d\theta\\
	&\leqslant \Big\|\delta^{l-\frac{1}{2}}\int_0^{\infty} \partial_h^{\alpha_h}\partial_3^l(\nabla p+z_{-}\cdot\nabla z_{+}) 
	(\tau,u_--\tau,x_2,x_3)d\tau\Big\|_{L^2(\mathcal{C}_+,\langle u_-\rangle^{2(1+\sigma)}du_-dx_2dx_3)}
	\stackrel{\eqref{eq:lemma1-18}}{\lesssim} C_1\varepsilon^2.
\end{align*}
The proof of this step is complete.

	\subsubsection*{\bf Step 3:} We prove that  as vector fields in $L^2(\mathcal{C}_+,\langle u_{-}\rangle^{2(1+\sigma)} du_-dx_2dx_3)$, 
	there hold
	\begin{align*}
		& 
		\delta^{l-\frac{1}{2}}\partial_h^{\gamma_h}\int_0^{\infty} \partial_h^{\alpha_h-\gamma_h}\partial_3^l(\nabla p+z_{-}\cdot\nabla z_{+}) 
		(\tau,u_--\tau,x_2,x_3)d\tau\\
		&\stackrel{ L^2(\mathcal{C}_+,\langle u_-\rangle^{2(1+\sigma)}du_-dx_2dx_3)}{=\joinrel=\joinrel=\joinrel=\joinrel=\joinrel=\joinrel=\joinrel=\joinrel=\joinrel=\joinrel=\joinrel=\joinrel=\joinrel=\joinrel=\joinrel=\joinrel=}\delta^{l-\frac{1}{2}}\int_0^{\infty} \partial_h^{\alpha_h}\partial_3^l(\nabla p+z_{-}\cdot\nabla z_{+}) 
		(\tau,u_--\tau,x_2,x_3)d\tau\stepcounter{equation}\tag{\theequation}\label{eq:lemma1-21}
	\end{align*}
	for any $\gamma_h\leqslant\alpha_h$ with $|\gamma_h|=1$, and 
		\begin{align*}
		& 
		\delta^{l-\frac{1}{2}}\partial_3\int_0^{\infty} \partial_h^{\alpha_h}\partial_3^{l-1}(\nabla p+z_{-}\cdot\nabla z_{+}) 
		(\tau,u_--\tau,x_2,x_3)d\tau\\
		&\stackrel{ L^2(\mathcal{C}_+,\langle u_-\rangle^{2(1+\sigma)}du_-dx_2dx_3)}{=\joinrel=\joinrel=\joinrel=\joinrel=\joinrel=\joinrel=\joinrel=\joinrel=\joinrel=\joinrel=\joinrel=\joinrel=\joinrel=\joinrel=\joinrel=\joinrel=}\delta^{l-\frac{1}{2}}\int_0^{\infty} \partial_h^{\alpha_h}\partial_3^l(\nabla p+z_{-}\cdot\nabla z_{+}) 
		(\tau,u_--\tau,x_2,x_3)d\tau.\stepcounter{equation}\tag{\theequation}\label{eq:lemma1-22}
	\end{align*}

Since \eqref{eq:lemma1-21} can be proved by similar methods used for  \eqref{eq:lemma1-9} and \eqref{eq:lemma1-22}, it is now suffices to show \eqref{eq:lemma1-22}. 
By virtue of \eqref{eq:lemma1-17} and  \eqref{eq:lemma1-20}, we only need to prove in the sense of distributions that 
	\begin{equation}\label{eq:lemma1-23}
		\delta^{l-\frac{1}{2}}\partial_3\int_0^{\infty}\!\! \partial_h^{\alpha_h}\partial_3^{l-1}(\nabla p+z_{-}\cdot\nabla z_{+}) 
		(\tau,u_--\tau,x_2,x_3)d\tau
		\stackrel{ \mathcal{D}'(\mathcal{C}_+)}{=\joinrel=\joinrel=\joinrel=}\delta^{l-\frac{1}{2}}\!\int_0^{\infty}\!\! \partial_h^{\alpha_h}\partial_3^l(\nabla p+z_{-}\cdot\nabla z_{+}) 
		(\tau,u_--\tau,x_2,x_3)d\tau. 
	\end{equation}

	Based on \eqref{eq:lemma1} and \eqref{eq:lemma1-17}, 
	both the two time integrals in \eqref{eq:lemma1-23} are locally integrable functions. 
For any vector field  $\varphi\in \mathcal{D}(\mathcal{C}_+)\subset L^2(\mathcal{C}_+)$, we have $\partial_3\varphi\in \mathcal{D}(\mathcal{C}_+)\subset L^2(\mathcal{C}_+)$. Hence, by the estimates in proofs of \eqref{eq:lemma1} and \eqref{eq:lemma1-17}, we can show   the following two spacetime integrals  finite to make the use of Fubini's theorem legitimate:
	\begin{align*}
		&\ \ \ \
		\delta^{l-\frac{1}{2}}\int_{[0,\infty)\times \mathcal{C}_+}\big|\partial_h^{\alpha_h}\partial_3^{l-1}(\nabla p+z_{-}\cdot\nabla z_{+}) 
		(\tau,u_--\tau,x_2,x_3)\cdot\partial_{h}^{\gamma_h}\varphi(u_-,x_2,x_3)\big|d\tau du_-dx_2dx_3\\
		&\lesssim \delta^{l-\frac{1}{2}}\int_{\mathbb{R}\times\mathcal{C}_+}\big|\partial_h^{\alpha_h}\partial_3^{l-1}(\nabla p+z_{-}\cdot\nabla z_{+}) 
		(u_+,u_-,x_2,x_3)\big|\cdot|\partial_{h}^{\gamma_h}\varphi(u_-,x_2,x_3)|du_+du_-dx_2dx_3\\
		&\lesssim
		\delta^{l-\frac{1}{2}}\Big(\int_{\mathbb{R}\times \mathcal{C}_+}\langle  u_+\rangle^{1+\sigma}\big|\partial_h^{\alpha_h}\partial_3^{l-1}(\nabla p+z_{-}\cdot\nabla z_{+})(u_+,u_-,x_2,x_3)\big|^2 \Big)^{\frac{1}{2}}
		\Big(\int_{\mathbb{R}\times \mathcal{C}_+}\frac{|\partial_3\varphi(u_-,x_2,x_3)|^2}{\langle u_+\rangle^{1+\sigma}}\Big)^{\frac{1}{2}}\\
		&\lesssim\delta^{l-\frac{1}{2}}\Big(\int_{\mathbb{R}\times \mathcal{C}_+}\langle u_-\rangle^{2(1+\sigma)}\langle u_+\rangle^{1+\sigma}\big|\partial_h^{\alpha_h}\partial_3^{l-1}(\nabla p+z_{-}\cdot\nabla z_{+}) (u_+,u_-,x_2,x_3)\big|^2\Big)^{\frac{1}{2}}\\
		&\ \ \ \ \times \Big(\int_{\mathbb{R}}\frac{1}{\langle u_+\rangle^{1+\sigma}}\Big(\int_{\mathcal{C}_+} |\partial_3 \varphi(u_-,x_2,x_3)|^2  du_-dx_2dx_3 \Big)du_+\Big)^{\frac{1}{2}}\\
		&\lesssim\varepsilon^2\|\partial_3\varphi\|_{L^2}<\infty,\stepcounter{equation}\tag{\theequation}\label{eq:lemma1-24}\\
		&\ \ \ \
		\delta^{l-\frac{1}{2}}\int_{[0,\infty)\times \mathcal{C}_+}\big|\partial_h^{\alpha_h}\partial_3^l(\nabla p+z_{-}\cdot\nabla z_{+}) 
		(\tau,u_--\tau,x_2,x_3)\cdot\varphi(u_-,x_2,x_3)\big|d\tau du_-dx_2dx_3\\
		&\lesssim \delta^{l-\frac{1}{2}}\int_{\mathbb{R}\times\mathcal{C}_+}\big|\partial_h^{\alpha_h}\partial_3^l(\nabla p+z_{-}\cdot\nabla z_{+}) 
		(u_+,u_-,x_2,x_3)\big|\cdot|\varphi(u_-,x_2,x_3)|du_+du_-dx_2dx_3\\
		&\lesssim
		\delta^{l-\frac{1}{2}}\Big(\int_{\mathbb{R}\times \mathcal{C}_+}\langle  u_+\rangle^{1+\sigma}\big|\partial_h^{\alpha_h}\partial_3^l(\nabla p+z_{-}\cdot\nabla z_{+}) 
		(u_+,u_-,x_2,x_3)\big|^2 \Big)^{\frac{1}{2}}
		\Big(\int_{\mathbb{R}\times \mathcal{C}_+}\frac{|\varphi(u_-,x_2,x_3)|^2}{\langle u_+\rangle^{1+\sigma}}\Big)^{\frac{1}{2}}\\
		&\lesssim\delta^{l-\frac{1}{2}}\Big(\int_{\mathbb{R}\times \mathcal{C}_+}\langle u_-\rangle^{2(1+\sigma)}\langle u_+\rangle^{1+\sigma}\big|\partial_h^{\alpha_h}\partial_3^l(\nabla p+z_{-}\cdot\nabla z_{+}) (u_+,u_-,x_2,x_3)\big|^2\Big)^{\frac{1}{2}}\\
		&\ \ \ \ \times \Big(\int_{\mathbb{R}}\frac{1}{\langle u_+\rangle^{1+\sigma}}\Big(\int_{\mathcal{C}_+} \varphi(u_-,x_2,x_3)|^2  du_-dx_2dx_3 \Big)du_+\Big)^{\frac{1}{2}}\\
		&\lesssim\varepsilon^2\|\varphi\|_{L^2}<\infty.\stepcounter{equation}\tag{\theequation}\label{eq:lemma1-25}
	\end{align*}
    
Applying \eqref{eq:lemma1-24}-\eqref{eq:lemma1-25}, integration by parts and Fubini's theorem repeatedly then gives in the sense of distributions that 
	\begin{align*}
		&\ \ \ \  \Big\langle\delta^{l-\frac{1}{2}} \partial_3\int_0^{\infty} \partial_h^{\alpha_h}\partial_3^{l-1}(\nabla p+z_{-}\cdot\nabla z_{+}) 
		(\tau,u_--\tau,x_2,x_3)d\tau,\ \varphi(u_-,x_2,x_3)\Big\rangle\\
		&=-\delta^{l-\frac{1}{2}}\Big\langle\int_0^{\infty}\partial_h^{\alpha_h}\partial_3^{l-1}(\nabla p+z_{-}\cdot\nabla z_{+}) 
		(\tau,u_--\tau,x_2,x_3)d\tau,\  \partial_3 \varphi(u_-,x_2,x_3)\Big\rangle\\
		&=-\delta^{l-\frac{1}{2}}\int_{\mathcal{C}_+}\Big(\int_0^{\infty}\partial_h^{\alpha_h}\partial_3^{l-1}(\nabla p+z_{-}\cdot\nabla z_{+}) 
		(\tau,u_--\tau,x_2,x_3)d\tau\Big)\cdot \partial_3\varphi(u_-,x_2,x_3)du_-dx_2dx_3\\
		&\stackrel{ 
			\eqref{eq:lemma1-24}}{=}-\delta^{l-\frac{1}{2}}\int_{[0,\infty) \times \mathcal{C}_+}\partial_h^{\alpha_h}\partial_3^{l-1}(\nabla p+z_{-}\cdot\nabla z_{+}) (\tau,u_--\tau,x_2,x_3)\cdot \partial_3\varphi(u_-,x_2,x_3)d\tau du_-dx_2dx_3\\
		&=-\delta^{l-\frac{1}{2}}\int_{[0,\infty)} \Big(\int_{\mathcal{C}_+}\partial_h^{\alpha_h}\partial_3^{l-1}(\nabla p+z_{-}\cdot\nabla z_{+}) (\tau,u_--\tau,x_2,x_3)\cdot \partial_3\varphi(u_-,x_2,x_3)du_-dx_2dx_3\Big)d\tau \\
		&=\delta^{l-\frac{1}{2}}\int_{[0,\infty)} \Big(\int_{\mathcal{C}_+}\partial_h^{\alpha_h}\partial_3^l(\nabla p+z_{-}\cdot\nabla z_{+}) (\tau,u_--\tau,x_2,x_3)\cdot \varphi(u_-,x_2,x_3)du_-dx_2dx_3\Big)d\tau \\
		&=\delta^{l-\frac{1}{2}}\int_{[0,\infty)\times \mathcal{C}_+}\partial_h^{\alpha_h}\partial_3^l(\nabla p+z_{-}\cdot\nabla z_{+}) (\tau,u_--\tau,x_2,x_3)\cdot \varphi(u_-,x_2,x_3)du_-dx_2dx_3d\tau\\
		&\stackrel{  
		\eqref{eq:lemma1-25}}{=}\delta^{l-\frac{1}{2}}\int_{\mathcal{C}_+}\Big(\int_{0}^{\infty} \partial_h^{\alpha_h}\partial_3^l(\nabla p+z_{-}\cdot\nabla z_{+})(\tau,u_--\tau,x_2,x_3)d\tau\Big)\cdot \varphi(u_-,x_2,x_3)du_-dx_2dx_3\\
		&=\Big\langle\delta^{l-\frac{1}{2}}\int_0^{\infty}  \partial_h^{\alpha_h}\partial_3^l(\nabla p+z_{-}\cdot\nabla z_{+})(\tau,u_--\tau,x_2,x_3)d\tau,\ \varphi(u_-,x_2,x_3)\Big\rangle.
	\end{align*}
which yields \eqref{eq:lemma1-23}. Thus we have finished this step.

	\subsubsection*{\bf Step 4:} Finally, we  prove \eqref{eq:lemma1-26} and \eqref{eq:lemma1-26'}. 
	
	By induction on $\alpha_h$ and $l$, we can apply \eqref{eq:lemma1-21} and \eqref{eq:lemma1-22} repeatedly to get  the following equation  in the sense of weighted $L^2$ space $L^2(\mathcal{C}_+,\langle u_{-}\rangle^{2(1+\sigma)} du_-dx_2dx_3)$:
	\begin{align*}
		& 
		\delta^{l-\frac{1}{2}}\partial_h^{\alpha_h}\partial_3^l\int_0^{\infty}(\nabla p+z_{-}\cdot\nabla z_{+}) 
		(\tau,u_--\tau,x_2,x_3)d\tau\\
		&\stackrel{ L^2(\mathcal{C}_+,\langle u_-\rangle^{2(1+\sigma)}du_-dx_2dx_3)}{=\joinrel=\joinrel=\joinrel=\joinrel=\joinrel=\joinrel=\joinrel=\joinrel=\joinrel=\joinrel=\joinrel=\joinrel=\joinrel=\joinrel=\joinrel=\joinrel=}\delta^{l-\frac{1}{2}}\int_0^{\infty} \partial_h^{\alpha_h}\partial_3^l(\nabla p+z_{-}\cdot\nabla z_{+}) 
		(\tau,u_--\tau,x_2,x_3)d\tau,\stepcounter{equation}\tag{\theequation}\label{eq:lemma1-27}
	\end{align*}
which together with \eqref{eq:lemma1-17} leads us to  \eqref{eq:lemma1-26}. 
Moreover, we can derive from \eqref{eq:lemma1-27} and  \eqref{eq:lemma1-18}  that 
	\begin{align*}
		&\ \ \ \ \Big\|\delta^{l-\frac{1}{2}}\partial_h^{\alpha_h}\partial_3^l\int_0^{\infty} (\nabla p+z_{-}\cdot\nabla z_{+}) 
		(\tau,u_--\tau,x_2,x_3)d\tau\Big\|_{L^2(\mathcal{C}_+,\langle u_-\rangle^{2(1+\sigma)}du_-dx_2dx_3)}\\
		&=\Big\|\delta^{l-\frac{1}{2}}\int_0^{\infty} \partial_h^{\alpha_h}\partial_3^l(\nabla p+z_{-}\cdot\nabla z_{+}) 
		(\tau,u_--\tau,x_2,x_3)d\tau\Big\|_{L^2(\mathcal{C}_+,\langle u_-\rangle^{2(1+\sigma)}du_-dx_2dx_3)}
		\lesssim C_1\varepsilon^2, 
	\end{align*}
which leads us to \eqref{eq:lemma1-26'} immediately.

Therefore we have ended the proof of this lemma.
\end{proof}

Let us also derive the uniform estimates concerning $\partial_h^{\alpha_h}\partial_3^l(\partial_3z_\pm)(\infty;u_\mp,x_2,x_3)$ with the lower order coefficient $\delta^{l-\frac{1}{2}}$. 
\begin{lemma}\label{lemma7}
	For any $\alpha_h\in(\mathbb{Z}_{\geqslant 0})^2$ and $l\in\mathbb{Z}_{\geqslant 0}$ with 
	$0\leqslant|\alpha_h|+l\leqslant N+2$, there holds 
	\begin{equation*}
		\delta^{l-\frac{1}{2}}
		\partial_h^{\alpha_h}\partial_3^l (\partial_3z_{\pm})(\infty;u_\mp,x_2,x_3)\in L^2(\mathcal{C}_\pm,\langle u_\mp\rangle^{2(1+\sigma)}du_\mp dx_2dx_3),
	\end{equation*}
and 
\begin{equation*} 
	\lim_{T\to\infty}\Big\|\delta^{l-\frac{1}{2}}\partial_h^{\alpha_h}\partial_3^l(\partial_3z_{\pm})(\infty;u_\mp,x_2,x_3)-\delta^{l-\frac{1}{2}}\partial_h^{\alpha_h}\partial_3^l(\partial_3z_{\pm})(T,u_\mp\mp T,x_2,x_3)\Big\|_{L^2(\mathcal{C}_\pm,\langle u_\mp\rangle^{2(1+\sigma)}du_\mp dx_2dx_3)}=0.
\end{equation*}
\end{lemma}
\begin{proof}
By the symmetry considerations, we only need to consider  
the scattering field  $z_{+}(\infty;u_-,x_2,x_3)$.
Similar arguments  
 in Lemma \ref{lemma6} remain valid for this lemma. The only difference lies in the  \textbf{Step 1}, which can be modified as follows:
 \begin{align*}
 	&\ \ \ \ \Big\|\delta^{l-\frac{1}{2}}\int_0^{\infty} \partial_h^{\alpha_h}\partial_3^l\partial_3(\nabla p+z_{-}\cdot\nabla z_{+}) 
 	(\tau,u_--\tau,x_2,x_3)d\tau\Big\|_{L^2(\mathcal{C}_+,\langle u_-\rangle^{2(1+\sigma)}du_-dx_2dx_3)}\\
 	& 
 	\lesssim\!\delta^{l\!-\!\frac{1}{2}\!}\Big\|\Big(\!\!\int_{\mathbb{R}}\!\!\frac{1}{\langle u_+\rangle^{1+\sigma}}du_+\!\!\Big)^{\!\!\frac{1}{2}}\Big(\!\!\int_{\mathbb{R}}\!\!\langle u_+\rangle^{1+\sigma}\! |\partial_h^{\alpha_h}\partial_3^l\partial_3(\nabla p+z_{-}\cdot\nabla z_{+}) 
 	(u_+,u_-,x_2,x_3)|^2du_+\!\!\Big)^{\!\!\frac{1}{2}}\!\Big\|_{L^2(\mathcal{C}_+,\langle u_-\rangle^{2(1+\sigma)}du_-dx_2dx_3)}\\
 	&\lesssim\delta^{l-\frac{1}{2}}\big\|\langle u_-\rangle^{1+\sigma}\langle u_+\rangle^{\frac{1}{2}(1+\sigma)} \partial_h^{\alpha_h}\partial_3^l\partial_3(\nabla p+z_{-}\cdot\nabla z_{+}) 
 	(u_+,u_-,x_2,x_3)\big\|_{L^2(\mathbb{R}\times\Omega_\delta,du_+du_-dx_2dx_3)}\\
 	&\lesssim\delta^{l-\frac{1}{2}}\big\|\langle u_-\rangle^{1+\sigma}\langle u_+\rangle^{\frac{1}{2}(1+\sigma)}\partial_{h}^{\alpha_h}\partial_3^l\partial_3 \nabla p\big\|_{L^2_tL^2_x}+\delta^{l-\frac{1}{2}}\big\|\langle u_-\rangle^{1+\sigma}\langle u_+\rangle^{\frac{1}{2}(1+\sigma)}{\mathbf{K}_+^{(\alpha_h,l)}}\big\|_{L^2_tL^2_x}
 	\stackrel{\text{Lemmas \ref{lemma:HHH} \& \ref{lemma3}}}{\lesssim} C_1\varepsilon^2. 
 \end{align*}	
We have thus proved this lemma.
\end{proof}

Gathering the above six lemmas  (from Lemma \ref{welldefined} to Lemma \ref{lemma7}) gives Theorem \ref{thm1} immediately.

\section{Rigidity from infinity in $\Omega_{\delta}$}  
\label{sec:rigidity}

Based on the properties of scattering fields shown in the last section, we are  ready to establish the rigidity from infinity for 3D Alfv\'en waves in thin domains $\Omega_{\delta}$. The second main theorem of this paper is as follows:
\begin{theorem}[Rigidity theorem in  $\Omega_{\delta}$]\label{thm2}
If the scattering fields constructed in Theorem \ref{thm1} 
vanish on the infinities, i.e.
\[\begin{cases}
	&\delta^{-\frac{1}{2}}z_+(\infty;u_-,x_2,x_3)\equiv 0 \ \ \text{on} \ \ \mathcal{C}_+,\\
	&\delta^{-\frac{1}{2}}z_-(\infty;u_+,x_2,x_3)\equiv 0 \ \ \text{on} \ \ \mathcal{C}_-,
\end{cases}
\]
then the initial Alfv\'en waves governed by the system  \eqref{MHD equation} vanish identically, i.e.  \[\big(z_{+,0}(x),z_{-,0}(x)\big)\equiv(0,0)\ \text{ for all }x\in\Omega_\delta,\] and hence
the Alfv\'en waves  governed by the system  \eqref{MHD equation} vanish identically, i.e. \[\big(z_+(t,x),z_-(t,x)\big)\equiv(0,0)\  \text{ for all }(t,x)\in \mathbb{R}\times \Omega_\delta.\] 
\end{theorem}

To construct this rigidity is mainly motivated by the previous work \cite{Li-Yu}. 
In a similar manner, the rigidity in Theorem \ref{thm2} has the physical interpretation that the 3D Alfv\'en waves  produced from the plasma in thin domains  $\Omega_{\delta}$ are characterized by their scattering fields detected by faraway observers, and hence  there are no  Alfv\'en waves  
 in thin domains  $\Omega_{\delta}$ if no waves are detected by faraway observers.

 We remark here that 
 based on Corollary \ref{coro2} and the notations before, Theorem \ref{thm2}  can  be immediately rephrased as the following renormalization result:
 
 \begin{corollary}[Rigidity theorem in $\Omega_1$]\label{coro3}
 	If the scattering fields constructed in Corollary \ref{coro2} 
 	vanish on the infinities, i.e.
 	\[\begin{cases}
 		&z_{+(\delta)}(\infty;u_-,x_2,x_3)\equiv 0 \ \ \text{on} \ \ \mathcal{C}_+,\\
 		&z_{-(\delta)}(\infty;u_+,x_2,x_3)\equiv 0 \ \ \text{on} \ \ \mathcal{C}_-,
 	\end{cases}
 	\]
 	then the initial Alfv\'en waves governed by the rescaled system  \eqref{eq:rescale} vanish identically, i.e.  \[\big(z_{+(\delta),0}(x),z_{-(\delta),0}(x)\big)\equiv(0,0)\ \text{ for all }x\in\Omega_1,\] and hence
 	the Alfv\'en waves governed by the rescaled system  \eqref{eq:rescale}  vanish identically, i.e.,  \[\big(z_{+(\delta)}(t,x),z_{-(\delta)}(t,x)\big)\equiv(0,0)\  \text{ for all }(t,x)\in \mathbb{R}\times \Omega_1.\] 
 \end{corollary}
 
The rest of this section is devoted to proving Theorem \ref{thm2}.
Now we suppose that the scattering fields 
	vanish identically at infinities, that is,
	\[\delta^{-\frac{1}{2}}z_\pm(\infty;x_1,x_2,u_\mp)\equiv 0\  \text{on $\mathcal{C}_\pm$.}\]
	
	Let $\epsilon<\varepsilon_0$ be an arbitrarily given small positive constant. By virtue of Theorem \ref{thm1} (ii), we can rephrase the vanishing property of scattering fields to the large time behavior of the solution, i.e.  there exists a large time $T_\epsilon>0$ such that we have the following smallness condition in the weighted energy sense:
	\begin{align*}
		\sum_{+,-}\bigg(&\sum_{0\leqslant|\alpha_h|+l\leqslant N+2}\Big\|\delta^{l-\frac{1}{2}}\partial_h^{\alpha_h}\partial_3^lz_{\pm}(T_\epsilon,u_\mp\mp T_\epsilon,x_2,x_3)\Big\|_{L^2(\Omega_{\delta},\langle u_\mp\rangle^{2(1+\sigma)}du_\mp dx_2dx_3)}^2\\
		&+\sum_{0\leqslant|\alpha_h|\leqslant N+2}\Big\|\delta^{-\frac{3}{2}}\partial_h^{\alpha_h}z^3_{\pm}(T_\epsilon,u_\mp\mp T_\epsilon,x_2,x_3)\Big\|_{L^2(\Omega_{\delta},\langle u_\mp\rangle^{2(1+\sigma)}du_\mp dx_2dx_3)}^2\\
		&+\sum_{0\leqslant|\alpha_h|+l\leqslant N+2}\Big\|\delta^{l-\frac{1}{2}}\partial_h^{\alpha_h}\partial_3^l(\partial_3z_{\pm})(T_\epsilon,u_\mp\mp T_\epsilon,x_2,x_3)\Big\|_{L^2(\Omega_{\delta},\langle u_\mp\rangle^{2(1+\sigma)}du_\mp dx_2dx_3)}^2\bigg)<\epsilon^2.
	\end{align*}

	We are in a position to study the position parameter $a$. As depicted in the following  Figure \ref{fig:rigidity},
	\begin{figure}[ht]
			\vspace{-0.1cm}
		\centering
		\includegraphics[width=3.8in]{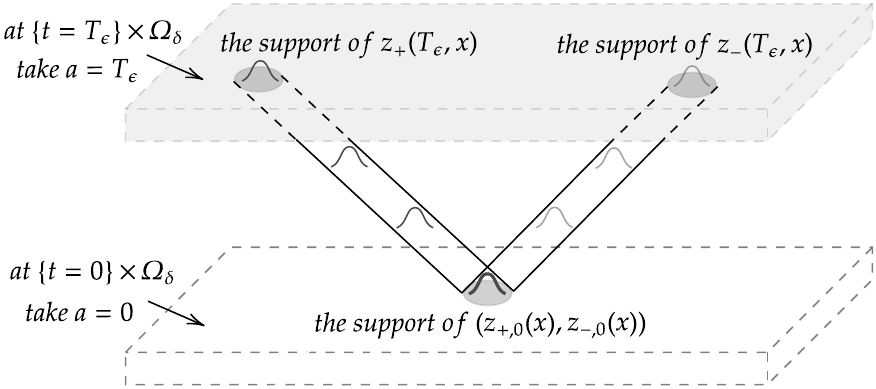}
		\vspace{-0.1cm}
		\caption{The position parameter in rigidity}
		\label{fig:rigidity}
			\vspace{-0.1cm}
	\end{figure}
	 at the initial time slice $\{t=0\}\times\Omega_{\delta}$, the position parameter $a=a_0=0$ is given so that we can construct the solution $z_\pm(t,x)$. At the time slice $\{t=T_\epsilon\}\times\Omega_{\delta}$, within the Cartesian coordinates, we can translate the above smallness condition as
	\begin{align*}
	\sum_{+,-}\bigg(&\sum_{0\leqslant|\alpha_h|+l\leqslant N+2}\Big\|\big(1+|x_1\pm T_\epsilon|^2\big)^{\frac{1+\sigma}{2}}\delta^{l-\frac{1}{2}}\partial_h^{\alpha_h}\partial_3^lz_{\pm}(T_\epsilon,x)\Big\|_{L^2(\Omega_{\delta},dx)}^2\\
	&+\sum_{0\leqslant|\alpha_h|\leqslant N+2}\Big\|\big(1+|x_1\pm T_\epsilon|^2\big)^{\frac{1+\sigma}{2}}\delta^{-\frac{3}{2}}\partial_h^{\alpha_h}z^3_{\pm}(T_\epsilon,x)\Big\|_{L^2(\Omega_{\delta},dx)}^2\\
	&+\sum_{0\leqslant|\alpha_h|+l\leqslant N+2}\Big\|\big(1+|x_1\pm T_\epsilon|^2\big)^{\frac{1+\sigma}{2}}\delta^{l-\frac{1}{2}}\partial_h^{\alpha_h}\partial_3^l(\partial_3z_{\pm})(T_\epsilon,x)\Big\|_{L^2(\Omega_{\delta},dx)}^2\bigg)<\epsilon^2.
	\end{align*}

Let us take $a=T_\epsilon$ as the new position parameter and regard $\big(z_{+}(T_\epsilon,x),z_{-}(T_\epsilon,x)\big)$ as the new initial data for the MHD system \eqref{MHD equation}. In this way, the smallness condition can be rewritten as 
		\begin{align*}
		\sum_{+,-}\bigg(&\sum_{0\leqslant|\alpha_h|+l\leqslant N+2}\Big\|\big(1+|x_1\pm a|^2\big)^{\frac{1+\sigma}{2}}\delta^{l-\frac{1}{2}}\partial_h^{\alpha_h}\partial_3^lz_{\pm}(T_\epsilon,x)\Big\|_{L^2(\Omega_{\delta},dx)}^2\\
		&+\sum_{0\leqslant|\alpha_h|\leqslant N+2}\Big\|\big(1+|x_1\pm a|^2\big)^{\frac{1+\sigma}{2}}\delta^{-\frac{3}{2}}\partial_h^{\alpha_h}z^3_{\pm}(T_\epsilon,x)\Big\|_{L^2(\Omega_{\delta},dx)}^2\\
		&+\sum_{0\leqslant|\alpha_h|+l\leqslant N+2}\Big\|\big(1+|x_1\pm a|^2\big)^{\frac{1+\sigma}{2}}\delta^{l-\frac{1}{2}}\partial_h^{\alpha_h}\partial_3^l(\partial_3z_{\pm})(T_\epsilon,x)\Big\|_{L^2(\Omega_{\delta},dx)}^2\bigg)<\epsilon^2.
	\end{align*}

We turn to solve the system backwards in time.
Denote $\widetilde{t}:=T_\epsilon-t$. The initial setting now turns into $\{t=T_\epsilon\}\times\Omega_{\delta}=\{\widetilde{t}=0\}\times\Omega_{\delta}$. Therefore we have
\begin{align*} 
	\sum_{+,-}\bigg(&\sum_{0\leqslant|\alpha_h|+l\leqslant N+2}\Big\|\big(1+|x_1\pm a|^2\big)^{\frac{1+\sigma}{2}}\delta^{l-\frac{1}{2}}\partial_h^{\alpha_h}\partial_3^lz_{\pm}(\widetilde{t}=0,x)\Big\|_{L^2(\Omega_{\delta},dx)}^2\\
	&+\sum_{0\leqslant|\alpha_h|\leqslant N+2}\Big\|\big(1+|x_1\pm a|^2\big)^{\frac{1+\sigma}{2}}\delta^{-\frac{3}{2}}\partial_h^{\alpha_h}z^3_{\pm}(\widetilde{t}=0,x)\Big\|_{L^2(\Omega_{\delta},dx)}^2\\
	&+\sum_{0\leqslant|\alpha_h|+l\leqslant N+2}\Big\|\big(1+|x_1\pm a|^2\big)^{\frac{1+\sigma}{2}}\delta^{l-\frac{1}{2}}\partial_h^{\alpha_h}\partial_3^l(\partial_3z_{\pm})(\widetilde{t}=0,x)\Big\|_{L^2(\Omega_{\delta},dx)}^2\bigg)<\epsilon^2.
\end{align*}

We are now ready to apply the  uniform (with respect to $\delta$) weighted energy estimates  which are independent of the position parameter. It should be noted that  the behaviors of $z_\pm$ in the process of solving the system backwards in time are the same as those of $z_\mp$ in the standard version of  uniform (with respect to $\delta$) weighted energy estimates due to the symmetry of time and the symmetry of space.
Consequently, according to Theorem \ref{lemma:global}, we infer that 
		\begin{align*} 
	\sum_{+,-}\bigg(&\sum_{0\leqslant|\alpha_h|+l\leqslant N+2}\Big\|\big(1+|u_\pm\pm a|^2\big)^{\frac{1+\sigma}{2}}\delta^{l-\frac{1}{2}}\partial_h^{\alpha_h}\partial_3^lz_{\mp}(\widetilde{t},x)\Big\|_{L^2(\Omega_{\delta},dx)}^2\\
		&+\sum_{0\leqslant|\alpha_h|\leqslant N+2}\Big\|\big(1+|u_\pm\pm a|^2\big)^{\frac{1+\sigma}{2}}\delta^{-\frac{3}{2}}\partial_h^{\alpha_h}z^3_{\mp}(\widetilde{t},x)\Big\|_{L^2(\Omega_{\delta},dx)}^2\\
		&+\sum_{0\leqslant|\alpha_h|+l\leqslant N+2}\Big\|\big(1+|u_\pm\pm a|^2\big)^{\frac{1+\sigma}{2}}\delta^{l-\frac{1}{2}}\partial_h^{\alpha_h}\partial_3^l(\partial_3z_{\mp})(\widetilde{t},x)\Big\|_{L^2(\Omega_{\delta},dx)}^2\bigg)<C\epsilon^2,
	\end{align*}
	where $C$ is a universal constant. This gives 
\begin{align*} 
	\sum_{+,-}\bigg(&\sum_{0\leqslant|\alpha_h|+l\leqslant N+2}\Big\|\big(1+|u_\mp\mp a|^2\big)^{\frac{1+\sigma}{2}}\delta^{l-\frac{1}{2}}\partial_h^{\alpha_h}\partial_3^lz_{\pm}(\widetilde{t},x)\Big\|_{L^2(\Omega_{\delta},dx)}^2\\
	&+\sum_{0\leqslant|\alpha_h|\leqslant N+2}\Big\|\big(1+|u_\mp\mp a|^2\big)^{\frac{1+\sigma}{2}}\delta^{-\frac{3}{2}}\partial_h^{\alpha_h}z^3_{\pm}(\widetilde{t},x)\Big\|_{L^2(\Omega_{\delta},dx)}^2\\
	&+\sum_{0\leqslant|\alpha_h|+l\leqslant N+2}\Big\|\big(1+|u_\mp\mp a|^2\big)^{\frac{1+\sigma}{2}}\delta^{l-\frac{1}{2}}\partial_h^{\alpha_h}\partial_3^l(\partial_3z_{\pm})(\widetilde{t},x)\Big\|_{L^2(\Omega_{\delta},dx)}^2\bigg)<C\epsilon^2.
\end{align*}

By definition of $u_\pm$ in \eqref{def:u}, it then follows that 
\begin{align*} 
	\sum_{+,-}\bigg(&\sum_{0\leqslant|\alpha_h|+l\leqslant N+2}\Big\|\big(1+|x_1\pm \widetilde{t}\mp a|^2\big)^{\frac{1+\sigma}{2}}\delta^{l-\frac{1}{2}}\partial_h^{\alpha_h}\partial_3^lz_{\pm}(\widetilde{t},x)\Big\|_{L^2(\Omega_{\delta},dx)}^2\\
	&+\sum_{0\leqslant|\alpha_h|\leqslant N+2}\Big\|\big(1+|x_1\pm \widetilde{t}\mp  a|^2\big)^{\frac{1+\sigma}{2}}\delta^{-\frac{3}{2}}\partial_h^{\alpha_h}z^3_{\pm}(\widetilde{t},x)\Big\|_{L^2(\Omega_{\delta},dx)}^2\\
	&+\sum_{0\leqslant|\alpha_h|+l\leqslant N+2}\Big\|\big(1+|x_1\pm \widetilde{t}\mp a|^2\big)^{\frac{1+\sigma}{2}}\delta^{l-\frac{1}{2}}\partial_h^{\alpha_h}\partial_3^l(\partial_3z_{\pm})(\widetilde{t},x)\Big\|_{L^2(\Omega_{\delta},dx)}^2\bigg)<C\epsilon^2.
\end{align*}
This estimate indeed holds for all $\widetilde{t}$. In particular, we take $\widetilde{t}=T_\epsilon$ to get 
\begin{align*} 
	\sum_{+,-}\bigg(&\sum_{0\leqslant|\alpha_h|+l\leqslant N+2}\Big\|\big(1+|x_1|^2\big)^{\frac{1+\sigma}{2}}\delta^{l-\frac{1}{2}}\partial_h^{\alpha_h}\partial_3^lz_{\pm}(\widetilde{t}=T_\epsilon,x)\Big\|_{L^2(\Omega_{\delta},dx)}^2\\
	&+\sum_{0\leqslant|\alpha_h|\leqslant N+2}\Big\|\big(1+|x_1|^2\big)^{\frac{1+\sigma}{2}}\delta^{-\frac{3}{2}}\partial_h^{\alpha_h}z^3_{\pm}(\widetilde{t}=T_\epsilon,x)\Big\|_{L^2(\Omega_{\delta},dx)}^2\\
	&+\sum_{0\leqslant|\alpha_h|+l\leqslant N+2}\Big\|\big(1+|x_1|^2\big)^{\frac{1+\sigma}{2}}\delta^{l-\frac{1}{2}}\partial_h^{\alpha_h}\partial_3^l(\partial_3z_{\pm})(\widetilde{t}=T_\epsilon,x)\Big\|_{L^2(\Omega_{\delta},dx)}^2\bigg)<C\epsilon^2.
\end{align*}

We now return to replace $\widetilde{t}$ by $t$.  Hence,  at the time slice $\{t=0\}\times\Omega_{\delta}$, it holds that 
\begin{align*} 
	\sum_{+,-}\bigg(&\sum_{0\leqslant|\alpha_h|+l\leqslant N+2}\Big\|\big(1+|x_1|^2\big)^{\frac{1+\sigma}{2}}\delta^{l-\frac{1}{2}}\partial_h^{\alpha_h}\partial_3^lz_{\pm}(0,x)\Big\|_{L^2(\Omega_{\delta},dx)}^2\\
	&+\sum_{0\leqslant|\alpha_h|\leqslant N+2}\Big\|\big(1+|x_1|^2\big)^{\frac{1+\sigma}{2}}\delta^{-\frac{3}{2}}\partial_h^{\alpha_h}z^3_{\pm}(0,x)\Big\|_{L^2(\Omega_{\delta},dx)}^2\\
	&+\sum_{0\leqslant|\alpha_h|+l\leqslant N+2}\Big\|\big(1+|x_1|^2\big)^{\frac{1+\sigma}{2}}\delta^{l-\frac{1}{2}}\partial_h^{\alpha_h}\partial_3^l(\partial_3z_{\pm})(0,x)\Big\|_{L^2(\Omega_{\delta},dx)}^2\bigg)<C\epsilon^2,
\end{align*}
which means
\begin{align*} 
	\sum_{+,-}\bigg(&\sum_{0\leqslant|\alpha_h|+l\leqslant N+2}\Big\|\big(1+|x_1|^2\big)^{\frac{1+\sigma}{2}}\delta^{l-\frac{1}{2}}\partial_h^{\alpha_h}\partial_3^lz_{\pm,0}\Big\|_{L^2(\Omega_{\delta},dx)}^2
	+\!\!\sum_{0\leqslant|\alpha_h|\leqslant N+2}\!\!\Big\|\big(1+|x_1|^2\big)^{\frac{1+\sigma}{2}}\delta^{-\frac{3}{2}}\partial_h^{\alpha_h}z^3_{\pm,0}\Big\|_{L^2(\Omega_{\delta},dx)}^2\\
	&+\sum_{0\leqslant|\alpha_h|+l\leqslant N+2}\Big\|\big(1+|x_1|^2\big)^{\frac{1+\sigma}{2}}\delta^{l-\frac{1}{2}}\partial_h^{\alpha_h}\partial_3^l(\partial_3z_{\pm,0})\Big\|_{L^2(\Omega_{\delta},dx)}^2\bigg)<C\epsilon^2.
\end{align*}
Here, at the time slice $\{t=0\}\times\Omega_{\delta}$, we notice that  the new weights associated to $a$ indeed coincide with the original weights, i.e. $\big(1+|x_1|^2\big)^{\omega}$.
	
	Since $\epsilon$ is arbitrary, we conclude that 
	\begin{align*} 
		\sum_{+,-}\bigg(&\sum_{0\leqslant|\alpha_h|+l\leqslant N+2}\Big\|\big(1+|x_1|^2\big)^{\frac{1+\sigma}{2}}\delta^{l-\frac{1}{2}}\partial_h^{\alpha_h}\partial_3^lz_{\pm,0}\Big\|_{L^2(\Omega_{\delta},dx)}^2
		+\!\!\sum_{0\leqslant|\alpha_h|\leqslant N+2}\!\!\Big\|\big(1+|x_1|^2\big)^{\frac{1+\sigma}{2}}\delta^{-\frac{3}{2}}\partial_h^{\alpha_h}z^3_{\pm,0}\Big\|_{L^2(\Omega_{\delta},dx)}^2\\
		&+\sum_{0\leqslant|\alpha_h|+l\leqslant N+2}\Big\|\big(1+|x_1|^2\big)^{\frac{1+\sigma}{2}}\delta^{l-\frac{1}{2}}\partial_h^{\alpha_h}\partial_3^l(\partial_3z_{\pm,0})\Big\|_{L^2(\Omega_{\delta},dx)}^2\bigg)=0,
	\end{align*}
which implies that the initial Alfv\'en waves $z_{\pm,0}$ indeed vanish. 
Therefore the Alfv\'en waves $z_{\pm}(t,x)$  vanish identically. This establishes  Theorem \ref{thm2} as desired.

\section{Asymptotics of rigidity theorem from $\Omega_{\delta}$ to $\mathbb{R}^2$}\label{sec:approximation}

In view of the above theorems, we are now in a position to investigate the approximation of the rigidity from infinity theorem from thin domains $\Omega_{\delta}$ to $\mathbb{R}^2$. Indeed, an immediate consequence is as follows:

\begin{corollary}[Asymptotics of rigidity theorem from $\Omega_{\delta}$ to $\mathbb{R}^2$ as   $\delta$ goes to zero]\label{thm3}
Under the assumptions of Theorem \ref{lemma:approx}, if  $z_{+(\delta)}(\infty;u_-,x_2,x_3)$ and $z_{-(\delta)}(\infty;u_+,x_2,x_3)$ are the scattering fields constructed in Corollary \ref{coro2} for the rescaled system \eqref{eq:rescale}, then there exist scattering fields  $z_{+(0)}^h(\infty;u_-,x_2)$ and $z_{-(0)}^h(\infty;u_+,x_2)$ such that for any $x_3\in(-1,1)$, there hold
\begin{equation}\label{eq:approx}
	\begin{split}
		\lim_{\delta\to 0}z_{\pm(\delta)}^h(\infty;u_-,x_2,x_3)&=z_{\pm(0)}^h(\infty;u_+,x_2)\ \ \text{ in }H^N(\mathbb{R}^2),\\
		\lim_{\delta\to 0}z_{\pm(\delta)}^3(\infty;u_-,x_2,x_3)&=0\ \ \text{ in }H^{N-1}(\mathbb{R}^2).
\end{split}\end{equation}
Moreover,	if the scattering fields  $z_{+(0)}^h(\infty;u_-,x_2)$ and $z_{-(0)}^h(\infty;u_+,x_2)$
vanish on the infinities, i.e.
\[\begin{cases}
	&z^h_{+(0)}(\infty;u_-,x_2)\equiv 0 \ \ \text{on} \ \ \mathcal{C}_+ \text{ (the 2D version of the $\mathcal{C}_+$ above)},\\
	&z^h_{-(0)}(\infty;u_+,x_2)\equiv 0 \ \ \text{on} \ \ \mathcal{C}_- \text{ (the 2D version of the $\mathcal{C}_-$ above)},
\end{cases}
\]
then the initial Alfv\'en waves governed by the 2D version of the rescaled system  \eqref{eq:rescale} vanish identically, i.e.  \[\big(z^h_{+(0),0}(x),z^h_{-(0),0}(x)\big)\equiv(0,0)\ \text{ for all }x\in\mathbb{R}^2,\] and hence
the Alfv\'en waves governed by the 2D version of the rescaled system  \eqref{eq:rescale}  vanish identically, i.e.,  \[\big(z^h_{+(0)}(t,x),z^h_{-(0)}(t,x)\big)\equiv(0,0)\  \text{ for all }(t,x)\in \mathbb{R}\times \mathbb{R}^2.\] 
\end{corollary}
\begin{proof}

Similar arguments used in the proof of Lemma \ref{welldefined} 
enable us to derive that 
\begin{align*}
&	 \langle u_\mp\rangle^{\frac{1}{2}(1+\sigma)}\langle u_\pm\rangle^{\frac{1}{2}(1+\sigma)} |\nabla_\delta p_{(\delta)}|
	\lesssim 
	\sum_{0\leqslant|\alpha_h|\leqslant 3\atop 0\leqslant l\leqslant 3}\big\|\langle u_\mp\rangle^{1+\sigma}\langle u_\pm\rangle^{\frac{1}{2}(1+\sigma)}\partial_{h}^{\alpha_h}\partial_3^l \nabla_\delta p_{(\delta)}\big\|_{L^2_tL^2_x}
	\stackrel{\text{Corollary \ref{remarkp}}}{\lesssim}C_1\varepsilon^2,\\
&	\langle u_\mp\rangle^{\frac{1}{2}(1+\sigma)}\langle u_\pm\rangle^{\frac{1}{2}(1+\sigma)} |z_{\mp(\delta)}\cdot\nabla z_{\pm(\delta)}|
	\lesssim\!\!\! 
	\sum_{0\leqslant|\alpha_h|\leqslant 3\atop 0\leqslant l\leqslant 3}\big\|\langle u_\mp\rangle^{1+\sigma}\langle u_\pm\rangle^{\frac{1}{2}(1+\sigma)}\partial_{h}^{\alpha_h}\partial_3^l (z_{\mp(\delta)}\cdot\nabla z_{\pm(\delta)})\big\|_{L^2_tL^2_x}
\!\!\!	\stackrel{\text{Remark \ref{remarkJ}}}{\lesssim}\!\!\! C_1\varepsilon^2.
\end{align*}
In view of \eqref{eq:product}, these two estimates then yield
\begin{equation*}
	|\nabla_\delta p_{(\delta)}+z_{\mp(\delta)}\cdot \nabla z_{\pm(\delta)}|\lesssim\frac{C_1\varepsilon^2}{\langle u_\mp\rangle^{\frac{1}{2}(1+\sigma)}\langle u_\pm\rangle^{\frac{1}{2}(1+\sigma)}}\lesssim\frac{C_1\varepsilon^2}{(1+|t+a|)^{1+\sigma}}\in L_t^1(\mathbb{R}).
\end{equation*}
This means that the two integrands in \eqref{eq:def-sca2} are uniformly integrable. 

 Therefore,  by the Lebesgue's dominated convergence theorem and Theorem \ref{lemma:approx}, we infer that 
\begin{align*}
\lim_{\delta\to 0}\int_0^\infty(\nabla_\delta p_{(\delta)}+z_{\mp(\delta)}\cdot\nabla z_{\pm(\delta)})(\tau,u_\mp\mp\tau,x_2,x_3)d\tau
&=\int_0^\infty\lim_{\delta\to 0}(\nabla_\delta p_{(\delta)}+z_{\mp(\delta)}\cdot\nabla z_{\pm(\delta)})(\tau,u_\mp\mp\tau,x_2,x_3)d\tau\\
&=\int_0^\infty(\nabla p_{(0)}+z_{\mp(0)}\cdot\nabla z_{\pm(0)})(\tau,u_\mp\mp\tau,x_2)d\tau,
\end{align*}
where the limit holds in $H^N(\mathbb{R}^2)$ for 
the horizontal component 
and in $H^{N-1}(\mathbb{R}^2)$ for 
the vertical component. 
Together with \eqref{eq:def-sca2}, Theorem \ref{lemma:approx} and Remark \ref{remark2D}, this gives rise to 
\begin{align*}
\lim_{\delta\to 0}z_{\pm(\delta)}(\infty;u_\mp,x_2,x_3)
&= \lim_{\delta\to 0}z_{\pm(\delta)}(0,u_\mp,x_2,x_3)- \lim_{\delta\to 0}\int_0^{\infty}(\nabla_\delta p_{(\delta)}+z_{\mp(\delta)}\cdot\nabla z_{\pm(\delta)})(\tau,u_\mp\mp\tau,x_2,x_3)d\tau\\
&=z_{\pm(0)}(0,u_\mp,x_2)-\int_0^\infty(\nabla p_{(0)}+z_{\mp(0)}\cdot\nabla z_{\pm(0)})(\tau,u_\mp\mp\tau,x_2)d\tau
=z_{\pm(0)}(\infty;u_\mp,x_2),
\end{align*}
where the limit holds in $H^N(\mathbb{R}^2)$ 
for the horizontal  
component and in $H^{N-1}(\mathbb{R}^2)$ for  
the vertical component. 
Precisely, we obtain \eqref{eq:approx} as asserted. 

By virtue of Corollary \ref{coro3}, it now follows that 
 	if the scattering fields constructed in Remark \ref{remark2D} 
vanish on the infinities, i.e.
\[\begin{cases}
	&z_{+(0)}(\infty;u_-,x_2)\equiv 0 \ \ \text{on} \ \ \mathcal{C}_+,\\
	&z_{-(0)}(\infty;u_+,x_2)\equiv 0 \ \ \text{on} \ \ \mathcal{C}_-,
\end{cases}
\]
then the initial Alfv\'en waves governed by the rescaled system  \eqref{eq:rescale} vanish identically, i.e.  \[\big(z_{+(0),0}(x),z_{-(0),0}(x)\big)\equiv(0,0)\ \text{ for all }x\in\mathbb{R}^2,\] and hence
the Alfv\'en waves governed by the rescaled system  \eqref{eq:rescale}  vanish identically, i.e.,  \[\big(z_{+(0)}(t,x),z_{-(0)}(t,x)\big)\equiv(0,0)\  \text{ for all }(t,x)\in \mathbb{R}\times \mathbb{R}^2.\] 
We note that the scattering fields 
$z_{+(0)}^h(\infty;u_-,x_2)$ and $z_{-(0)}^h(\infty;u_+,x_2)$ are the 2D version of the scattering fields $z_{+(0)}(\infty;u_-,x_2)$ and $z_{-(0)}(\infty;u_+,x_2)$; the initial data $\big(z^h_{+(0),0}(x),z^h_{-(0),0}(x)\big)$ are the 2D version of the initial data $\big(z_{+(0),0}(x),z_{-(0),0}(x)\big)$; and the Alfv\'en waves $\big(z^h_{+(0)}(t,x),z^h_{-(0)}(t,x)\big)$ are the 2D version of the Alfv\'en waves $\big(z_{+(0)}(t,x),z_{-(0)}(t,x)\big)$. Consequently, the rigidity part of Corollary \ref{thm3} follows immediately. 

Up to now, we have finished the proof of Corollary \ref{thm3}. 
\end{proof}

\begin{remark}\label{remark1}
	In particular, the scattering fields 
	$z_{+(0)}^h(\infty;u_-,x_2)$ and $z_{-(0)}^h(\infty;u_+,x_2)$ are the 2D version of the scattering fields $z_{+(0)}(\infty;u_-,x_2)$ and $z_{-(0)}(\infty;u_+,x_2)$ constructed in Remark \ref{remark2D} for the rescaled system \eqref{eq:rescale}. 
	We also remark that the scattering fields  $z_{+(0)}^h(\infty;u_-,x_2)$ and $z_{-(0)}^h(\infty;u_+,x_2)$ coincide with the scattering fields $z_{+}(\infty;u_-,x_2)$ and $z_{-}(\infty;u_+,x_2)$  in Theorem \ref{futurescatteringfields MHD 2d} as \eqref{eq:def-sca4} respectively. Furthermore, the rigidity part of Corollary \ref{thm3} indeed coincides with the 2D version of the rigidity from infinity theorem constructed in \cite{Li-Yu}. For the readers' convenience, this 2D version of the rigidity theorem in \cite{Li-Yu} is provided as Theorem \ref{rigidity theorem 1 2d} in the appendix. For clarity of comparison,  
	we illustrate the relations among these results in the following Figure \ref{fig:relations}.
	\begin{figure}[ht]
		\vspace{-0.1cm}
		\centering
		\includegraphics[width=6in]{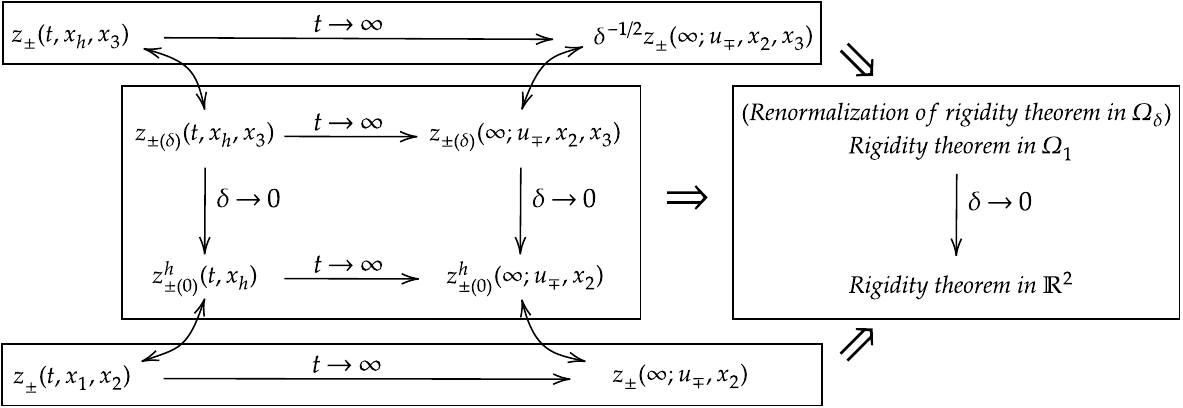}
		\vspace{-0.1cm}
		\caption{Relations among rigidity theorems}
		\label{fig:relations}
		\vspace{-0.1cm}
	\end{figure}
\end{remark}

\begin{remark}
	Though the  approximation of global existence result from $\Omega_{\delta}$ to $\mathbb{R}^2$  in \cite{Xu} is nontrivial, 	we see that the  approximation of rigidity result from $\Omega_{\delta}$ to $\mathbb{R}^2$ in this paper  can be regarded a direct consequence.  
	This is trivial based on  the geometric  fact that $\Omega_{\delta}$ can be viewed as $\mathbb{R}^2$ when $\delta\to 0$,    the 
	approximation result in Theorem 1.3 from \cite{Xu} that  3D Alfv\'en waves in $\Omega_{\delta}$ converge to  2D Alfv\'en waves in $\mathbb{R}^2$ 
	when $\delta\to 0$, and in particular 
	the rigidity from infinity  conditions that the scattering fields for Alfv\'en waves 
	are vanishing at their corresponding infinities. However, 
	things will become completely different if we study the  approximation of  inverse scattering result from $\Omega_{\delta}$ to $\mathbb{R}^2$ when given arbitrarily small scattering fields (instead of vanishing scattering fields).     This will be treated in a forthcoming paper since the problem is more complicated and requires involved techniques. 
\end{remark}

\begin{appendix}
	\section{Rigidity from infinity in $\mathbb{R}^2$}
	
	In this appendix, we extend the rigidity results in \cite{Li-Yu} for 3D Alfv\'en waves in $\mathbb{R}^3$  to the case of 2D Alfv\'en waves in $\mathbb{R}^2$ for the readers' convenience. 
	
	Precisely, we study the most interesting physical situation in 2D MHD where a strong background magnetic field generates Alfv\'en waves. For consistency in 
	this paper, we follow the notations of initial data $(z_{+,0},z_{-,0})$, weight functions as in \eqref{def:weight}, and so forth (only differences lie  in replacing $\Omega_\delta$ by $\mathbb{R}^2$ and all the 3D coordinates by the corresponding 2D coordinates); and we  
	assume that the strong background  magnetic field is taken as the unit vector field along the $x_1$-axis in $\mathbb{R}^2$: $B_0=(1,0)$.  
	
	Since there never exist boundary conditions in $\mathbb{R}^2$, the system for $(z_+,z_-)$ now can be written as the first four equations in \eqref{MHD equation}, i.e. 
	\begin{equation}\label{MHD equation2}\begin{cases}
			&\partial_{t}z_{+}+(z_--B_0)\cdot \nabla z_{+} =-\nabla p,\\
			&\partial_{t}z_{-}+(z_++B_0)\cdot \nabla z_{-} =-\nabla p,\\
			&\operatorname{div} z_{+}=0,\ \
			\operatorname{div} z_{-}=0.
	\end{cases}\end{equation}
	
We now list the three main theorems for the 2D case to give a more precise comparison with the 3D case in \cite{Li-Yu} and the thin-domain case in the previous sections.
	\begin{theorem}[Weighted energy estimates in $\mathbb{R}^2$]\label{Main Energy Estimates MHD 2d}  Let $N_* \in \mathbb{Z}_{\geqslant 5}$ and $\sigma \in\big(0,\frac{1}{3}\big)$. There exists a universal constant $\varepsilon_0\in(0,1)$ such that if the initial data $\big(z_{+,0}(x),z_{-,0}(x)\big)$ of \eqref{MHD equation2} satisfy
		\begin{equation*}
			\mathcal{E}^{N_*}(0): =\sum_{+,-}\sum_{k=0}^{N_*+1}\Big\|\left(1+|x_1\pm a|^2\right)^{\frac{1+\sigma}{2}}\nabla^{k} z_{\pm}(0,x)\Big\|_{L^2(\mathbb{R}^2)}^2\leqslant\varepsilon_0^2,
		\end{equation*}
		then the system \eqref{MHD equation2} admits a unique global solution  $\big(z_+(t,x),z_-(t,x)\big)$. Moreover, there exists a universal constant $C$ such that the following weighted energy estimates hold:
		\begin{equation*}
			\sum_{+,-}\sum_{k=0}^{N_*+1}\sup_{t\geqslant 0}\Big\|\left(1+|u_\mp \pm a|^2\right)^{\frac{1+\sigma}{2}}\nabla^{k}z_{\pm}(t,x)\Big\|_{L^2(\mathbb{R}^2)}^2 \leqslant C \mathcal{E}^{N_*}(0).
		\end{equation*}
	\end{theorem}

	\begin{theorem}[Scattering fields in $\mathbb{R}^2$]\label{futurescatteringfields MHD 2d}
		For the solution $\left(z_+(t,x),z_-(t,x)\right)$ constructed in Theorem \ref{Main Energy Estimates MHD 2d}, the following two vector fields
		\begin{equation}\label{eq:def-sca4}
			\begin{cases}
				&\displaystyle z_+(\infty;u_-,x_2):=z_+(0,u_-,x_2)-\int_0^{\infty} \left(\nabla p+z_{-}\cdot\nabla z_{+}\right)(\tau,u_--\tau,x_2)d\tau\\
				&\displaystyle z_-(\infty;u_+,x_2):=z_-(0,u_+,x_2)-\int_0^{\infty} \left(\nabla p+z_{+}\cdot\nabla z_{-}\right)(\tau,u_++\tau,x_2)d\tau
			\end{cases}
		\end{equation}
		are well-defined on the infinities   $\mathcal{C}_+$  (the 2D version of the $\mathcal{C}_+$ above) and $\mathcal{C}_-$ (the 2D version of the  $\mathcal{C}_-$ above) respectively. We call  $z_+(\infty;u_-,x_2)$ as the left scattering field and $z_-(\infty;u_+,x_2)$ as the right  scattering field. Moreover, for any $\beta\in(\mathbb{Z}_{\geqslant 0})^2$ with $0\leqslant |\beta|\leqslant N_*$ there hold the following two properties of scattering fields:
		\begin{enumerate}[(i)]
	\item these scattering fields live in the following functional spaces in the weighted energy sense:
	\[\nabla^\beta z_\pm(\infty;u_\mp,x_2)\in L^2(\mathcal{C}_\pm,\langle u_\mp\rangle^{2(1+\sigma)}du_\mp dx_2).\]
	\item these scattering fields can be approximated by the large time solution in the weighted energy sense:
	\[\lim_{T\to\infty}\Big\|\nabla^{\beta} z_{\pm}(\infty;u_\mp,x_2)-\nabla^\beta z_{\pm}(T,u_\mp\mp T,x_2)\Big\|_{L^2(\mathcal{C}_\pm,\langle u_\mp\rangle^{2(1+\sigma)}du_\mp dx_2)}=0\]
		\end{enumerate}
	\end{theorem}

	\begin{theorem}[Rigidity theorem in $\mathbb{R}^2$]\label{rigidity theorem 1 2d}
		If the scattering fields  constructed in Theorem \ref{futurescatteringfields MHD 2d} vanish on the  infinities, i.e. 
		\begin{equation*}
			\begin{cases}
				& z_+(\infty;u_-,x_2)\equiv 0\ \ \text{on} \ \ \mathcal{C}_+\text{ (the 2D version of the  $\mathcal{C}_+$ above)},\\
				& z_-(\infty;u_+,x_2)\equiv 0\ \ \text{on} \ \ \mathcal{C}_-\text{ (the 2D version of the  $\mathcal{C}_-$ above)},
			\end{cases}
		\end{equation*}	
	then the initial Alfv\'en waves vanish identically, i.e.  \[\big(z_{+,0}(x),z_{-,0}(x)\big)\equiv(0,0)\ \text{ for all }x\in\mathbb{R}^2,\] and hence
	the Alfv\'en waves  vanish identically, i.e.,  \[\big(z_+(t,x),z_-(t,x)\big)\equiv(0,0)\  \text{ for all }(t,x)\in \mathbb{R}\times \mathbb{R}^2.\] 
	\end{theorem}
	
\begin{remark}
	In fact, the rigidity from infinity for Alfv\'en waves in $\mathbb{R}^2$  constructed in Theorem \ref{rigidity theorem 1 2d} coincides with the rigidity part of Corollary \ref{thm3}, and hence  coincides with the approximation of the rigidity from infinity for Alfv\'en waves in $\Omega_{\delta}$ propagating along the horizontal direction as $\delta$ goes to zero. Combined  with Remark \ref{remark1}, this means that   the two perspectives  on the rigidity for Alfv\'en waves in $\mathbb{R}^2$ from Corollary \ref{thm3} and Theorem \ref{rigidity theorem 1 2d} coincide with each other and are perfectly unified. We also remark here that these relations have been nicely depicted in Figure \ref{fig:relations}.
\end{remark}

The proofs of these results are indeed almost the same as that used in \cite{Li-Yu}, and hence we only  sketch the differences between them in the rest of this paper. 

In fact, it suffices to make suitable modifications on \cite{Li-Yu}: On one hand, we need to replace $\delta\in\big(0,\frac{2}{3}\big)$ therein by $\sigma\in\big(0,\frac{1}{3}\big)$, and replace $\mathbb{R}^3$ by $\mathbb{R}^2$ and replace all the 3D coordinates (such as $(u_+,x_2,x_3)$) by the corresponding 2D coordinates (such as $(u_+,x_2)$), which will be carried out repeatedly in what follows without further comment. On the other hand, since the pressure term therein  involves the Newtonian potential, we need to modify the related estimates on the pressure (we point out that the results are the same while the details are different), especially including the bound on $\mathbf{I'}$ (i.e. from (2.23) to (2.30)) and the proof of Corollary 2.12 therein. In the rest of this paper, we will use two subsections to modify these two pressure estimates respectively.

\subsection{Modification on the estimate of $\mathbf{I'}$}
By taking divergence on both sides of the first equation of \eqref{MHD equation2} and using $\operatorname{div} z_{\pm}=0$, we can obtain the first equation in  \eqref{eqpressure-system}, i.e. $-\Delta p=\partial_iz_-^j\partial_jz_+^i$. According to the Newtonian potential in the 2D case, we can infer on each time slice $\Sigma_{\tau}=\{(t,x)\in\mathbb{R}\times\mathbb{R}^2\,|\,t=\tau\}$ that
\begin{equation*} 
	p(\tau,x)=-\frac{1}{2\pi} \int_{\mathbb{R}^2}\log{|x-y|}\cdot \big(\partial_{i}z^{j}_{-}\partial_{j}z^{i}_{+}\big)(\tau,y)dy.
\end{equation*}
Therefore, the following decomposition of $\nabla p$ holds on any time slice $\Sigma_{\tau}$:
\begin{align*}
	\nabla p(\tau,x)=\ &
	-\frac{1}{2\pi}\nabla\int_{\mathbb{R}^2}\log{|x-y|}\cdot\big(\partial_{i}z^{j}_{-}\partial_{j}z^{i}_{+}\big)(\tau,y)dy=
	-\frac{1}{2\pi}\int_{\mathbb{R}^2}\nabla\log{|x-y|}\cdot\big(\partial_{i}z^{j}_{-}\partial_{j}z^{i}_{+}\big)(\tau,y)dy\\
	=\ &-\frac{1}{2\pi}\int_{\mathbb{R}^2}\nabla\log{|x-y|}\cdot \theta(|x-y|)\cdot \big(\partial_{i}z^{j}_{-}\partial_{j}z^{i}_{+}\big)(\tau,y)dy\\
	& -\frac{1}{2\pi}\int_{\mathbb{R}^2}\nabla\log{|x-y|}\cdot \big(1-\theta(|x-y|)\big)\cdot\big(\partial_{i}z^{j}_{-}\partial_{j}z^{i}_{+}\big)(\tau,y)dy,
\end{align*}
where the smooth cutoff function $\theta(r)$ is still taken as \eqref{cutoff}. 
Since $\operatorname{div} z_\pm=0$, we can integrate by parts to derive 
\begin{align*}
	\!\!\!	\nabla p(\tau,x)
	=\ &
	-\frac{1}{2\pi}\int_{\mathbb{R}^2}\nabla\log{|x-y|}\cdot	\theta(|x-y|)\cdot\big(\partial_{i}z^{j}_{-}\partial_{j}z^{i}_{+}\big)(\tau,y)dy\\
	\ &  +\frac{1}{2\pi}\int_{\mathbb{R}^2}\partial_{i}\left(\nabla\log{|x-y|} \cdot\left(1-	\theta(|x-y|)\right)\right)\cdot\big(z^{j}_{-}\partial_{j}z^{i}_{+}\big)(\tau,y)dy\\
	=\ &
	-\frac{1}{2\pi}\int_{\mathbb{R}^2}\nabla\log{|x-y|}\cdot	\theta(|x-y|)\cdot\big(\partial_{i}z^{j}_{-}\partial_{j}z^{i}_{+}\big)(\tau,y)dy\\
	\ &  -\frac{1}{2\pi}\int_{\mathbb{R}^2}\partial_{j}\partial_{i}\left(\nabla\log{|x-y|}\! \cdot\!\left(1-	\theta(|x-y|)\right)\right)\!\cdot\!\big(z^{j}_{-}z^{i}_{+}\big)(\tau,y)dy.\stepcounter{equation}\tag{\theequation}\label{eq:nabla p 2d}
\end{align*}
Using the property of the cutoff function $\theta(r)$, we can bound $\nabla p$ as follows:
\begin{align*}
	\left|\nabla p(\tau ,x)\right|
	\lesssim &
	\int_{|x-y|\leqslant 2}
	\left|\nabla\log{|x-y|}\right|\cdot\left|	\theta(|x-y|)\right|\cdot\big|\big(\partial_{i}z^{j}_{-}\partial_{j}z^{i}_{+}\big)(\tau,y)\big|dy\\
	&+\int_{|x-y|\geqslant 1}
	\left|\partial_{j}\partial_{i}\left(\nabla\log{|x-y|}\right)\right|\cdot \left|1-\theta(|x-y|)\right|\cdot 
	\big|\big(z^{j}_{-}z^{i}_{+}\big)(\tau,y)\big|dy\\
	&
	+\int_{1\leqslant |x-y|\leqslant 2}\Big(\left|\partial_{i}\left(\nabla\log{|x-y|}\right)\right|\cdot \left|\theta'(|x-y|)\right|	+\left|\nabla\log{|x-y|}\right|\cdot  \left|\theta''(|x-y|)\right|\Big) \cdot \big| \big(z^{j}_{-}z^{i}_{+}\big)(\tau,y)\big|dy\\
	\lesssim &\underbrace{\int_{|x-y|\leqslant 2}\frac{\big|\big(\nabla z_{-}\cdot \nabla z_{+}\big)(\tau ,y)\big|}{|x-y|^{}}dy}_{\mathbf{A_{1}}}+
	\underbrace{\int _{|x-y|\geqslant 1}\frac{\big|\big(z_{-}\cdot z_{+}\big)(\tau ,y)\big|}{|x-y|^{3}}dy}_{\mathbf{A_{2}}}+\underbrace{\int _{1\leqslant |x-y|\leqslant 2} \left|\left(z_{-}\cdot z_{+}\right)(\tau ,y)\right|dy}_{\mathbf{A_{3}}}.\stepcounter{equation}\tag{\theequation}\label{decomposition of nabla p 2d}
\end{align*}
Hence we obtain 
\begin{align*}
	\mathbf{I'} 
	&\lesssim 
	\underbrace{\int_{[0,t]\times\mathbb{R}^2}\langle u_{-}\rangle^{2(1+\sigma)}\langle u_{+}\rangle^{1+\sigma} \left|\mathbf{A_1}\right|^2dxd\tau}_{\mathbf{I_1}}
	+\underbrace{\int_{[0,t]\times\mathbb{R}^2}\langle u_{-}\rangle^{2(1+\sigma)}\langle u_{+}\rangle^{1+\sigma} \left|\mathbf{A_2}\right|^2dxd\tau}_{\mathbf{I_2}}\\
	&+\underbrace{\int_{[0,t]\times\mathbb{R}^2}\langle u_{-}\rangle^{2(1+\sigma)}\langle u_{+}\rangle^{1+\sigma} \left|\mathbf{A_3}\right|^2dxd\tau}_{\mathbf{I_3}}.\stepcounter{equation}\tag{\theequation}\label{eq:I'}
\end{align*}

Before proceeding further, we collect some preliminary results in \cite{Li-Yu}. 
For example,  the properties \eqref{eq:weight1}-\eqref{eq:weight2} now can be improved as follows: 
\begin{equation}\label{eq:xgeqleq}
	(\langle u_{\mp}\rangle^{1+\sigma}\langle u_{\pm}\rangle^{\frac{1+\sigma}{2}})(\tau,x)\lesssim\begin{cases}
		(\langle u_{\mp}\rangle^{1+\sigma}\langle u_{\pm}\rangle^{\frac{1+\sigma}{2}})(\tau,y)+|x-y|^{\frac{3(1+\sigma)}{2}}&\ \ \ \ \text{if}\ \  |x-y|\geqslant 1,\\
		(\langle u_{\mp}\rangle^{1+\sigma}\langle u_{\pm}\rangle^{\frac{1+\sigma}{2}})(\tau,y)&\ \ \ \ \text{if}\ \  |x-y|\leqslant 2.
	\end{cases}
\end{equation}
Moreover, 
there hold the following pointwise estimate and spacetime estimate for $z_\pm$:
\begin{enumerate}
	\item[(i)] (Weighted Sobolev inequality)
	For all $k\leqslant N_*-2$ and multi-indices $\alpha$ with $\left|\alpha\right|=k$, we have	
	\begin{equation}\label{eq:Sobolev}
		\langle u_{\mp}\rangle^\sigma|z_{\pm}|+\langle u_{\mp}\rangle^\sigma|\nabla z_{\pm}^{(\alpha)}|\lesssim C_1\varepsilon.
	\end{equation}
	\item[(ii)] (Weighted spacetime estimate)
	For all $0\leqslant k\leqslant N_*$ and multi-indices $\alpha$ with $\left|\alpha\right|=k$, we have
	\begin{equation}\label{eq:flux}
		\int_{[0,t]\times\mathbb{R}^2}\frac{\langle u_{\mp}\rangle^{2(1+\sigma)}}{\langle u_{\pm}\rangle^{\sigma}}|z_{\pm}|^2dxd\tau+\int_{[0,t]\times\mathbb{R}^2}\frac{\langle u_{\mp}\rangle^{2(1+\sigma)}}{\langle u_{\pm}\rangle^{\sigma}}\big|\nabla z^{(\alpha)}_{\pm}\big|^2dxd\tau\lesssim \big(C_1\big)^2\varepsilon^2.
	\end{equation}
\end{enumerate}  
The proofs of these results are omitted since the 2D case can also be treated in the same manner. We refer the readers to Lemma 2.4, Lemma 2.5 and Lemma 2.7 in \cite{Li-Yu} for more details. 

We now turn to bound $\mathbf{I_1}$, $\mathbf{I_2}$ and $\mathbf{I_3}$ one by one:

For $\mathbf{I_1}$, using the definition of $\mathbf{A_1}$, we derive
\begin{align*}
	\langle u_-\rangle^{1+\sigma}\langle u_+\rangle^{\frac{{1+\sigma}}{2}}\left|\mathbf{A_1}\right|
	&=
	\int_{|x-y|\leqslant 2}\frac{\big(\langle u_-\rangle^{1+\sigma}\langle u_+\rangle^{\frac{{1+\sigma}}{2}}\big)(\tau,x)\left|\left(\nabla z_{-}\cdot\nabla z_{+}\right)(\tau,y)\right|}{|x-y|}dy\\
	&\stackrel{\eqref{eq:xgeqleq}}{\lesssim}
	\int_{|x-y|\leqslant 2}\frac{\left(\langle u_-\rangle^{1+\sigma}\langle u_+\rangle^{\frac{{1+\sigma}}{2}}\right)(\tau,y)\left|\left(\nabla z_{-}\cdot\nabla z_{+}\right)(\tau,y)\right|}{|x-y|}dy\\
	&\leqslant \left\|\langle u_{+}\rangle^{{1+\sigma}}\,\nabla z_-\right\|_{L^\infty_x}\int_{|x-y|\leqslant 2}\frac{\langle u_{-}\rangle^{{{1+\sigma}}}(\tau,y)\left|\nabla z_+(\tau,y)\right|}{\langle u_{+}\rangle^{\frac{{1+\sigma}}{2}}(\tau,y)|x-y|}dy\\
	&\stackrel{\eqref{eq:Sobolev}}{\lesssim}
	C_1\varepsilon\int_{|x-y|\leqslant 2}\frac{1}{|x-y|}\frac{\langle u_{-}\rangle^{{1+\sigma}}(\tau,y)}{\langle u_{+}\rangle^{\frac{{1+\sigma}}{2}}(\tau,y)} \left|\nabla z_+(\tau,y)\right|dy.
\end{align*}
By \eqref{observe integration},  we notice that $\dfrac{1}{|x|}\text{\Large$\text{\Large$\chi$}$}_{|x|\leqslant 2}\in L^1(\mathbb{R}^2)$. Using Young's inequality, we obtain
\begin{equation*}
	\big\|\langle u_-\rangle^{1+\sigma}\langle u_+\rangle^{\frac{{1+\sigma}}{2}}\mathbf{A_1}\big\|^2_{L^2(\mathbb{R}^2)}\!\!
	\lesssim\left(C_1\right)^2\varepsilon^2\,\Big\|\frac{1}{|x|}\text{\Large$\text{\Large$\chi$}$}_{|x|\leqslant 2}\Big\|^2_{L^1(\mathbb{R}^2)} \Big\|\frac{\langle u_{-}\rangle^{{1+\sigma}}}{\langle u_{+}\rangle^{\frac{{1+\sigma}}{2}}}\nabla z_{+}\Big\|^2_{L^2(\mathbb{R}^2)}\!\!
	\lesssim\left(C_1\right)^2\varepsilon^2\,\Big\|\frac{\langle u_{-}\rangle^{{1+\sigma}}}{\langle u_{+}\rangle^{\frac{{1+\sigma}}{2}}}\nabla z_{+}\Big\|^2_{L^2(\mathbb{R}^2)}.
\end{equation*}
We then apply \eqref{eq:flux} to infer that 
\begin{equation}\label{eq:I1 2d}
	\mathbf{I_{1}} \lesssim \left(C_1\right)^4\varepsilon^4.
\end{equation}

Here we notice that $\mathbf{I_3}$ can be treated in exactly the same manner. Consequently, we get 
\begin{equation}\label{eq:I3 2d}
	\mathbf{I_3}\lesssim \left(C_1\right)^4\varepsilon^4.
\end{equation}

For $\mathbf{I_2}$, we obtain
\begin{align*}
	\mathbf{I_2}
	&=\int_{0}^{t}\Bigg\|\int_{|x-y|\geqslant 1}\langle u_{-}\rangle^{{1+\sigma}}(\tau,x)\langle u_{+}\rangle^{\frac{{1+\sigma}}{2}}(\tau,x)\frac{\left|(z_{-}\cdot z_{+})(\tau,y)\right|}{|x-y|^3}dy\Bigg\|_{L^2(\mathbb{R}^2)}^2d\tau\\
	&\stackrel{\eqref{eq:xgeqleq}}{\lesssim} \int_{0}^{t}\Bigg\|\int_{|x-y|\geqslant 1}\big(\big(\langle u_{-}\rangle^{{1+\sigma}}\langle u_{+}\rangle^{\frac{{1+\sigma}}{2}}\big)(\tau,y)+|x-y|^{\frac{3{1+\sigma}}{2}}\big)\frac{\left|\left(z_{-}\cdot z_{+}\right)(\tau,y)\right|}{|x-y|^3}dy\Bigg\|_{L^2(\mathbb{R}^2)}^2d\tau\\
	&\lesssim
	\underbrace{\int_{0}^{t}\bigg\|\int_{|x-y|\geqslant 1}\!\!\!\big(\langle u_{-}\rangle^{{1+\sigma}}\langle u_{+}\rangle^{\frac{{1+\sigma}}{2}}\big)(\tau,y)\frac{\left|\left(z_{-}\cdot z_{+}\right)(\tau,y)\right|}{|x-y|^3}dy\bigg\|_{L^2(\mathbb{R}^2)}^2\!\!\!\!\!\!d\tau}_{\mathbf{I_{21}}} +\underbrace{\int_{0}^{t}\bigg\|\int_{|x-y|\geqslant 1}\!\!\!\frac{\left|\left(z_{-}\cdot z_{+}\right)(\tau,y)\right|}{|x-y|^{3-\frac{3(1+\sigma)}{2}}}dy\bigg\|_{L^2(\mathbb{R}^2)}^2\!\!\!\!\!\!d\tau}_{\mathbf{I_{22}}}.
\end{align*}
For $\mathbf{I_{21}}$, since $\dfrac{1}{|x|^3}\text{\Large$\chi$}_{|x|\geqslant 1}\in L^1(\mathbb{R}^2)$, we can derive
\begin{align*}
	\mathbf{I_{21}}&
	\stackrel{\text{Young}}{\lesssim}
	\int_{0}^{t}\left\|\frac{1}{|x|^3}\text{\Large$\chi$}_{|x|\geqslant 1}\right\|_{L^1(\mathbb{R}^2)}^2 \left\|\langle u_{+}\rangle^{{1+\sigma}}z_{-}\right\|_{L^\infty(\mathbb{R}^2)}^2\left\|\frac{\langle u_{-}\rangle^{{1+\sigma}}}{\langle u_{+}\rangle^{\frac{{1+\sigma}}{2}}}z_{+}\right\|_{L^2(\mathbb{R}^2)}^2d\tau \\
	&\stackrel{\eqref{eq:Sobolev}}{\lesssim}\left(C_1\right)^2\varepsilon^2\,\int_{0}^{t}\left\|\frac{\langle u_{-}\rangle^{{1+\sigma}}}{\langle u_{+}\rangle^{\frac{{1+\sigma}}{2}}}z_{+}\right\|_{L^2(\mathbb{R}^2)}^2d\tau
	\stackrel{\eqref{eq:flux}}{\lesssim} \left(C_1\right)^4\varepsilon^4.
\end{align*}
For $\mathbf{I_{22}}$, since $\dfrac{1}{|x|^{3-\frac{3(1+\sigma)}{2}}}\text{\Large$\chi$}_{|x|\geqslant 1}\in L^2(\mathbb{R}^2)$ holds when ${1+\sigma}\in(1,\frac{4}{3})$ (this is the place we have constraint on $\sigma$), then we have
\begin{align*}
	\mathbf{I_{22}}&
	\stackrel{\text{Young}}{\lesssim}
	\int_{0}^{t}\Bigg\|\frac{1}{|x|^{3-\frac{3(1+\sigma)}{2}}}\text{\Large$\chi$}_{|x|\geqslant 1}\Bigg\|_{L^2(\mathbb{R}^2)}^2\left\|z_{-}z_{+}\right\|_{L^1(\mathbb{R}^2)}^2d\tau\\
	&\lesssim \int_{0}^{t}\left\|z_{-} z_{+}\right\|_{L^1(\mathbb{R}^2)}^2d\tau
=\int_{0}^{t}\left\|\frac{1}{\langle u_{+}\rangle^{{1+\sigma}}\langle u_{-}\rangle^{{1+\sigma}}}\langle u_{+}\rangle^{{1+\sigma}}z_{-}\langle u_{-}\rangle^{{1+\sigma}} z_{+}\right\|_{L^1(\mathbb{R}^2)}^2d\tau\\
	&\stackrel{\eqref{eq:product}}{\lesssim}  \int_{0}^{t}\frac{1}{\left(1+\left|\tau+a\right|\right)^{2(1+\sigma)}}\left\|\langle u_{+}\rangle^{{1+\sigma}}z_{-}\right\|_{L^2(\mathbb{R}^2)}^2\left\|\langle u_{-}\rangle^{{1+\sigma}}z_{+}\right\|_{L^2(\mathbb{R}^2)}^2d\tau\\
	&\lesssim \left(C_1\right)^4\varepsilon^4\int_{0}^{t}\frac{1}{\left(1+\left|\tau+a\right|\right)^{2(1+\sigma)}}d\tau	\lesssim \left(C_1\right)^4\varepsilon^4.
\end{align*}
Therefore we can similarly obtain
\begin{equation}\label{eq:I2 2d}
	\mathbf{I_2}\lesssim \left(C_1\right)^4\varepsilon^4.
\end{equation}

Adding \eqref{eq:I1 2d}, \eqref{eq:I3 2d} and \eqref{eq:I2 2d} together, and then using \eqref{eq:I'}, we can similarly obtain
\begin{equation*}
	\mathbf{I'}\lesssim \left(C_1\right)^4\varepsilon^4.
\end{equation*}

\subsection{Modification on the proof of Corollary 2.12}

In fact, Corollary 2.12 \cite{Li-Yu} in the 2D case is the following Corollary \ref{coro:bound on p 2d}:

\begin{corollary}\label{coro:bound on p 2d}
	For the solution $\left(z_+,z_-\right)$ constructed in Theorem \ref{Main Energy Estimates MHD 2d}, for $l=1,2$, for all $(\tau,x)\in \mathbb{R}\times \mathbb{R}^2$, we have the following estimates on the pressure term:
	\begin{equation*}
		\left|\nabla^l p\,(\tau,x)\right|\lesssim\frac{\varepsilon^2}{\left(1+\left|\tau+a\right|\right)^{1+\sigma}}.
	\end{equation*}
\end{corollary}
\begin{proof}
	For $l=1$, we will estimate $\nabla p$ by using \eqref{decomposition of nabla p 2d}. Next, we will bound $\mathbf{A_1}$, $\mathbf{A_2}$ and $\mathbf{A_3}$ one by one. For $\mathbf{A_1}$, we can also infer
	\begin{align*}
		\langle u_-\rangle^{1+\sigma}\langle u_+\rangle^{{1+\sigma}}\left|\mathbf{A_1}\right|
		&=\int_{|x-y|\leqslant 2}\frac{\left(\langle u_-\rangle^{1+\sigma}\langle u_+\rangle^{{1+\sigma}}\right)(\tau,x)\left|\left(\nabla z_{-}\cdot\nabla z_{+}\right)(\tau,y)\right|}{|x-y|}dy\\
		&\stackrel{\eqref{eq:xgeqleq}}{\lesssim}\!
		\int_{|x-y|\leqslant 2}\frac{\left(\langle u_-\rangle^{1+\sigma}\langle u_+\rangle^{ {1+\sigma}}\right)(\tau,y)\left|\left(\nabla z_{-}\cdot\nabla z_{+}\right)(\tau,y)\right|}{|x-y|}dy
		\stackrel{\eqref{eq:Sobolev}}{\lesssim}\!
		\int_{|x-y|\leqslant 2}\frac{\varepsilon^2}{|x-y|}dy\lesssim \varepsilon^2.
	\end{align*}
	We repeat the estimate on $\mathbf{A_1}$ to the estimate on $\mathbf{A_3}$. It then follows that 
	\begin{equation*}
		\langle u_+\rangle^{1+\sigma}\langle u_-\rangle^{1+\sigma}\left|\mathbf{A_3}\right|
		\lesssim\varepsilon^2.
	\end{equation*}
	For $\mathbf{A_2}$, we can similarly use \eqref{eq:product} to obtain $1+|\tau+a| \lesssim \langle u_{+}\rangle\langle u_{-}\rangle$, and hence there holds 
	\begin{align*}
		\left(1+|\tau+a|\right)^{{{1+\sigma}}}\left|\mathbf{A_2}\right|&\lesssim\int_{|x-y|\geqslant 1}\frac{\langle u_+\rangle^{1+\sigma}(\tau,y)\langle u_-\rangle^{1+\sigma}(\tau,y)\left| z_{-}(\tau,y)\right|\left|z_{+}(\tau,y)\right|}{|x-y|^3}dy
		\lesssim \int_{|x-y|\geqslant 1}\frac{\varepsilon^2}{|x-y|^3}dy
		\lesssim \varepsilon^2.
	\end{align*}
	Putting all the estimates on $\mathbf{A_i}$ together, we conclude that 
	\begin{equation*}\label{nablap 2d}
		|\nabla p(\tau,x)|\lesssim\frac{\varepsilon^2}{\left(1+\left|\tau+a\right|\right)^{1+\sigma}}.
	\end{equation*}

	In order to bound $\nabla^2 p$, we take $\nabla$ on both sides of \eqref{eq:nabla p 2d}. Similar to the derivation of \eqref{decomposition of nabla p 2d}, we can derive
	\begin{align*}
		|\nabla^2 p(\tau,x)|
		\lesssim
		& \underbrace{\sum_{l_{1},l_{2}=1}^2 \int_{|x-y|\leqslant 2}\frac{1}{|x-y|}\left|\left(\nabla^{l_1} z_{-}\cdot\nabla^{l_2} z_{+}\right)(\tau,y)\right|dy}_{\mathbf{B_1}}\\
		&+\underbrace{\int_{|x-y|\geqslant 1}\frac{1}{|x-y|^3}\left| \left(z_{-}\cdot\nabla z_{+}\right)(\tau,y)\right|dy}_{\mathbf{B_2}}	+\underbrace{\int_{1\leqslant |x-y|\leqslant 2}\frac{1}{|x-y|^2}\left|\left(z_{-}\cdot\nabla z_{+}\right)(\tau,y)\right|dy}_{\mathbf{B_3}},
	\end{align*}
	where  $(l_1,l_2)=(1,1),$ $(1,2)$ or $(2,1)$. We can repeat the above estimate on $\mathbf{A_i}$ to give the estimate on $\mathbf{B_i}$, and thus imply the estimate on $\nabla^2 p$.
\end{proof}

\end{appendix}

\end{document}